\documentclass[12pt]{amsart}
\usepackage{amscd,amssymb,amsthm,amsmath,amssymb,textcomp,supertabular,rotating,longtable,enumerate,rotating,mathrsfs,mathtools,multirow,comment,hyperref}
\usepackage[all]{xy}
\usepackage{tikz}
\usepackage{tikz-cd}
\usetikzlibrary{automata, positioning} %For drawing

\newtheorem{theorem}[equation]{Theorem}

\newtheorem{proposition}[equation]{Proposition}
\newtheorem{lemma}[equation]{Lemma}
\newtheorem{corollary}[equation]{Corollary}
\newtheorem*{corollary*}{Corollary}
\newtheorem{conjecture}[equation]{Conjecture}
\newtheorem*{conjecture*}{Conjecture}

\newtheorem*{maintheorem*}{Main Theorem}
\newtheorem*{maincorollary*}{Main Corollary}

\theoremstyle{definition}
\newtheorem{example}[equation]{Example}
\newtheorem{definition}[equation]{Definition}

\theoremstyle{remark}
\newtheorem{remark}[equation]{Remark}

\makeatletter\@addtoreset{equation}{section} \makeatother

\newcommand{\mumu}{\boldsymbol{\mu}}
\newcommand{\vin}{\rotatebox[origin=c]{90}{$\in$}}

\title{K-stability of Casagrande--Druel varieties}

\author{Ivan Cheltsov, Tiago Duarte Guerreiro, Kento Fujita, \\ Igor Krylov, Jesus Martinez-Garcia}

\address{\emph{Ivan Cheltsov}
\newline
\textnormal{University of Edinburgh,  Edinburgh, Scotland}
\newline
\textnormal{\texttt{i.cheltsov@ed.ac.uk}}}

\address{ \emph{Tiago Duarte Guerreiro}
\newline
\textnormal{University of Essex, Colchester, England}
\newline
\textnormal{\texttt{t.duarteguerreiro@essex.ac.uk}}}

\address{\emph{Kento Fujita}
\newline
\textnormal{Osaka University, Osaka, Japan}
\newline
\textnormal{\texttt{fujita@math.sci.osaka-u.ac.jp}}}

\address{\emph{Igor Krylov}
\newline
\textnormal{Institute for Basic Science, Pohang, Korea}
\newline
\textnormal{\texttt{krylov.igor.o@gmail.com}}}

\address{\emph{Jesus Martinez-Garcia}
\newline
\textnormal{University of Essex, Colchester, England}
\newline
\textnormal{\texttt{jesus.martinez-garcia@essex.ac.uk}}}

\begin{document}

\maketitle

\begin{abstract}
We introduce a new subclass of Fano varieties (Casagrande--Druel varieties),
that are $n$-dimensional varieties constructed from Fano double covers of dimension~$n-1$.
We conjecture that a Casagrande--Druel variety is K-polystable if the~double cover
and its base space are K-polystable.
We prove this for smoothable Casagrande--Druel threefolds,
and for Casagrande--Druel varieties constructed from double covers of $\mathbb{P}^{n-1}$
ramified over smooth hypersurfaces of degree $2d$ with $n>d>\frac{n}{2}>1$.
As an application, we describe the connected components of the K-moduli space
parametrizing smoothable K-polystable Fano threefolds in the families \textnumero 3.9 and \textnumero 4.2 in the Mori-Mukai classification.
\end{abstract}

%\tableofcontents

\bigskip

Throughout this paper, all varieties are defined over~$\mathbb{C}$.

\section{Introduction}
\label{section:intro}

Let $V$ be a Fano variety with Kawamata log terminal singularities,
and let $L$ be a line bundle on $V$ such that the~divisor $-(K_V+L)$ is ample,
and $|2L|$ contains a non-zero effective divisor.
Let $R$ be a divisor in $|2L|$, and let $\eta\colon B\to V$ be the~double cover ramified over $R$.
Then $B$ can be explicitly constructed as follows. Let
$$
Y=\mathbb{P}\big(\mathcal{O}_{V}\oplus \mathcal{O}_{V}(L)\big),
$$
let $\pi\colon Y\to V$ be the~natural projection,  and let $\xi$ be the~tautological line bundle on $Y$.
Set $H=\pi^*(L)$.
Then we have isomorphisms:
\begin{align*}
H^0\big(Y, \mathcal{O}_{Y}(\xi)\big)&\cong  H^0\big(V,\mathcal{O}_V\big)\oplus H^0\big(V,\mathcal{O}_V(L)\big),\\
H^0\big(Y, \mathcal{O}_{Y}(\xi-H)\big)&\cong H^0\big(V,\mathcal{O}_V\big)\oplus H^0\big(V,\mathcal{O}_V(-L)\big).
\end{align*}
Using these isomorphisms, fix sections $u^+\in H^0(Y,\mathcal{O}_{Y}(\xi))$ and $u^-\in H^0(Y, \mathcal{O}_{Y}(\xi-H))$
that correspond to $1\in H^0(V,\mathcal{O}_V)$ under the isomorphisms above.
Set $S^{\pm}=\{u^{\pm}=0\}$.
Then we have $S^-\cap S^+=\varnothing$ and $S^+\sim S^-+H$.
Take $f\in H^0(V,\mathcal{O}_V(2L))$ that defines $R$.
Then we can identify $B$ with the~divisor
$$
\big\{\pi^*(f)(u^-)^2=(u^+)^2\big\}\in|2S^+|,
$$
where the~double cover $\eta$ is induced by  $\pi$.

\begin{remark}
\label{remark:Fano-double-cover}
We allow $R$ to be singular, so $B$ can be very singular (and even reducible).
However, if the~log pair $(V,\frac{1}{2}R)$ has Kawamata log terminal singularities,
then the~double cover $B$ is a Fano variety with Kawamata log terminal singularities \cite{Ko97}.
So, for simplicity, we will always say that $B$ is a Fano double cover (even if $B$ is non-normal or reducible).
\end{remark}

Let $F=\pi^*(R)$, and let $\phi\colon X\to Y$ be the~blow up of the~intersection $S^+\cap F$. Then
\begin{center}
$X$ is smooth $\iff$ $Y$ and $B$ are smooth $\iff$ $V$ and $R$ are smooth.
\end{center}
Moreover, the variety $X$ is also a Fano variety (see Section~\ref{section:CD-properties}).

\begin{definition}
\label{definition:CD-variety}
If the Fano variety $X$ has at most Kawamata log terminal singularities,
then $X$ is called \emph{the~Casagrande--Druel variety} constructed from $\eta\colon B\to V$
(or, from the~ramification divisor $R\subset V$). Note that $L\in\operatorname{Pic} V$ is uniquely determined by $R$.
\end{definition}

The~group $\mathrm{Aut}(Y)$ contains a subgroup $\Gamma\cong\mathbb{G}_m$ that fixes both $S^-$ and $S^+$  pointwise,
and the~action of $\Gamma$ lifts to $\mathrm{Aut}(X)$, so we can identify $\Gamma$ with a subgroup in~$\mathrm{Aut}(X)$.
In~Section~\ref{section:CD-properties}, we will show that $\mathrm{Aut}(X)$ also contains an~involution $\iota$ such that
$$
\langle\Gamma,\iota\rangle\cong\mathbb{G}_m\rtimes\mumu_2,
$$
and $\iota$ swaps the~proper transforms of the~sections $S^-$ and $S^+$.
Set $G=\langle\Gamma,\iota\rangle$ and $\theta=\pi\circ\phi$. Then we have commutative diagram:
\begin{equation}
\label{equation:diagram-symmetric}
\xymatrix{
&X\ar[dd]_{\theta}\ar[dl]_{\phi}\ar[rr]^{\iota}&&X\ar[dd]^{\theta}\ar[dr]^{\phi}&\\
Y\ar[dr]_{\pi} &&&&Y\ar[dl]^{\pi} \\
&V\ar@{=}[rr]&& V&}
\end{equation}
and the~composition $\theta$ is a $G$-equivariant conic bundle such that $G$ acts trivially on $V$.

\begin{remark}
\label{remark:CD-varieties}
Our construction of Casagrande--Druel varieties is inspired by the~paper~\cite{CD}.
See \cite[Lemma 3.1 (iii)]{CD}. 
But it goes back to the~construction of de Jonquieres involutions using hyperelliptic curves instead of Fano double covers.
See also \cite{maruyama2,BCW,tsukioka,divpt}.
\end{remark}

The del Pezzo surface of degree $6$ (blow up of $\mathbb{P}^2$ at three general points) is the~unique smooth Casagrande-Druel surface.
Smooth Casagrande--Druel threefolds form $3$ families.
To present them, we use labeling of smooth Fano threefolds from \cite{Book}.

\begin{example}
\label{example:3-19}
Let $V=\mathbb{P}^2$, let $L=\mathcal{O}_{\mathbb{P}^2}(1)$, let $R$ be an arbitrary smooth conic in $|2L|$.
Then $B\cong\mathbb{P}^1\times\mathbb{P}^1$,
and $X$ is the~unique smooth Fano threefold in the~family \textnumero 3.19.
\end{example}

\begin{example}
\label{example:3-9}
Let $V=\mathbb{P}^2$, let $L=\mathcal{O}_{\mathbb{P}^2}(2)$, let $R$ be any smooth quartic curve in $|2L|$.
Then $B$ is a del Pezzo surface of degree $2$, and $X$ is a Fano threefold in the~family \textnumero 3.9.
\end{example}

\begin{example}
\label{example:4-2}
Let $V=\mathbb{P}^1\times\mathbb{P}^1$, let $L=\mathcal{O}_V(1,1)$, let $R$ be any~smooth curve in $|2L|$.
Then $B$ is a del Pezzo surface of degree $4$, and $X$ is a Fano threefold in the~family \textnumero 4.2.
\end{example}

All smooth Casagrande--Druel threefolds are K-polystable, see \cite[Theorem 6.1]{IltenSuess}~and~\cite{Book}.
In fact, K-polystable Casagrande--Druel varieties exist in every dimension:

\begin{example}[{\cite{Delcroix2020,Delcroix2022}}]
\label{example:Thibaud}
Suppose that $V=\mathbb{P}^{n-1}$, $L=\mathcal{O}_{\mathbb{P}^{n-1}}(1)$, $R$ is smooth, $n\geqslant 2$.
Then $X$ can be obtained by blowing up the $n$-dimensional smooth quadric at two points.
The variety $X$ is spherical, and it is known that $X$ is K-polystable \cite[4.4.2]{Delcroix2022}.
\end{example}

In this paper, we prove the following theorem:

\begin{theorem}
\label{theorem:double-spaces}
Suppose that $V=\mathbb{P}^{n-1}$, $L=\mathcal{O}_{\mathbb{P}^{n-1}}(r)$, $R$ is smooth, and $n>r>\frac{n}{2}>1$.
Then $X$ is K-polystable.
\end{theorem}

We obtain this result as an application of the~following K-polystability criteria:

\begin{theorem}
\label{theorem:delta}
Suppose that both $V$ and $R$ are smooth (or equivalently $X$ is smooth), and~\mbox{$-K_V\sim_{\mathbb{Q}} aL$},
where $a\in\mathbb{Q}_{>0}$ such that $a>1$.
Let $\mu$ be the~smallest rational number such that $\mu L$ is very ample.
Set $n=\mathrm{dim}(X)$ (so $\mathrm{dim}(V)=n-1$), set $d=L^{n-1}$, set
$$
k_n(a,d,\mu)=\frac{a^{n+1}-(a-1)^{n+1}}{(n+1)(a^n-(a-1)^n)}d\mu^{n-2}+\frac{a^{n+1}-(a+n)(a-1)^{n}}{2(n+1)(a^n-(a-1)^n)}
$$
and set
$$
\gamma=\mathrm{min}\Bigg\{\frac{1}{k_n(a,d,\mu)},\frac{(n+1)(a^n-(a-1)^n)}{(n+1-a)a^n+(a-1)^{n+1}},\frac{a\delta(V)(n+1)(a^{n}-(a-1)^{n})}{n(a^{n+1}-(a-1)^{n+1})}\Bigg\},
$$
where $\delta(V)$ is the~$\delta$-invariant of the~Fano variety $V$.
If $n\geqslant 3$, $d\mu^{n-2}\geqslant 2$ and $\gamma>1$, then the Casagrande--Druel variety $X$ is K-polystable.
\end{theorem}

\begin{remark}
\label{remark:dmu-1}
In the notations of Theorem~\ref{theorem:delta}, if $n\geqslant 2$ and $d\mu^{n-2}<2$,
then~\mbox{$d\mu^{n-2}=1$}, which gives $V=\mathbb{P}^{n-1}$ and $L=\mathcal{O}_{\mathbb{P}^{n-1}}(1)$,
so  $X$ is K-polystable, see Example~\ref{example:Thibaud}.
\end{remark}

In this paper, we also prove the~following two theorems about K-polystability of several singular Casagrande--Druel 3-folds:

\begin{theorem}
\label{theorem:4-2}
Suppose $V=\mathbb{P}^1\times\mathbb{P}^1$, $L=\mathcal{O}_{V}(1,1)$, and $R$ is one of the~following~curves:
\begin{enumerate}
\item[$(\mathrm{1})$] $C_1+C_2$, where $C_1$ and $C_2$ are smooth curves in $|L|$ such that $|C_1\cap C_2|=2$;
\item[$(\mathrm{2})$] $\ell_1+\ell_2+\ell_3+\ell_4$, where $\ell_1$ and $\ell_2$ are two distinct smooth curves of degree $(1,0)$,
and $\ell_3$ and $\ell_4$ are two distinct smooth curves of degree $(0,1)$;
\item[$(\mathrm{3})$] $2C$, where $C$ is a smooth curve in $|L|$.
\end{enumerate}
Then $X$ is K-polystable.
\end{theorem}

\begin{theorem}
\label{theorem:3-9}
Suppose $V=\mathbb{P}^2$, $L=\mathcal{O}_{\mathbb{P}^2}(2)$, and $R$ is one of the~following curves:
\begin{enumerate}
\item[$(\mathrm{1})$] a singular reduced curve~in $|2L|$ with at most $\mathbb{A}_1$ or $\mathbb{A}_2$ singularities;
\item[$(\mathrm{2})$] $C_1+C_2$, where $C_1$ and $C_2$ are smooth conics that are tangent at two points;
\item[$(\mathrm{3})$] $C+\ell_1+\ell_2$, where $C$ is a smooth conic, $\ell_1$ and $\ell_2$ are distinct lines tangent to $C$;
\item[$(\mathrm{4})$] $2C$, where $C$ is a smooth conic.
\end{enumerate}
Then $X$ is K-polystable.
\end{theorem}

To present their applications,
let $\mathcal{M}^{\operatorname{Kss}}_{n,v}$ be the K-moduli
functor of Fano~varieties that have dimension $n$ and anticanonical volume $v\in\mathbb{Q}_{>0}$ in the sense of \cite[Theorem 2.17]{XZ}.
Then $\mathcal{M}^{\operatorname{Kss}}_{n,v}$ is an Artin stack of finite type.
Moreover, as in \cite[Theorem 1.3]{LXZ}, it admits a~good moduli space
$\mathcal{M}^{\operatorname{Kss}}_{n,v}\longrightarrow M^{\operatorname{Kps}}_{n,v}$
in the sense of \cite{Alper}, where $M^{\operatorname{Kps}}_{n,v}$ is a projective scheme
whose points paramertize K-polystable Fano varieties of dimension $n$ and anticanonical volume~$v$.
Let~$M^{\operatorname{Kps}}_{(3.9)}$ and $M^{\operatorname{Kps}}_{(4.2)}$ be the~closed subvarieties of $M^{\operatorname{Kps}}_{3,26}$ and $M^{\operatorname{Kps}}_{3,28}$
whose general points parametrize smooth Fano theeefolds in the~families \textnumero 3.9 and \textnumero 4.2, respectively.
Then Theorems~\ref{theorem:4-2} and \ref{theorem:3-9} imply the following two results (see Section \ref{section:moduli} and cf. \cite{HeubergerPetracci}).

\begin{corollary}
\label{corollary:4-2}
Let $V=\mathbb{P}^1\times\mathbb{P}^1$, let $L=\mathcal{O}_{V}(1,1)$,
let $\Gamma=\left(\operatorname{SL}_2(\mathbb{C})\times\operatorname{SL}_2(\mathbb{C})\right)\rtimes\mumu_2$,
let~$T=\mathbb{P}\left(H^0\left(V,\mathcal{O}_V(2,2)\right)^\vee\right)$,
let $T^{\operatorname{ss}}\subset T$ be the GIT semistable
open subset with respect to the natural $\Gamma$-action,
and let $M$ be the GIT quotient $T^{\operatorname{ss}}\mathbin{/\mkern-6mu/}\Gamma$.
Then there is a morphism
\begin{eqnarray*}
    \Phi\colon M &\to&M^{\operatorname{Kps}}_{3,28} \\
            \vin \  &     &  \ \ \vin  \\
    {[}f{]} &\mapsto& {[}X_f{]},
\end{eqnarray*}
where $X_f$ is the Casagrande--Druel threefold that is constructed from $R=\{f=0\}\in |2L|$.
Furthermore, the morphism $\Phi$ is an isomorphism onto $M^{\operatorname{Kps}}_{(4.2)}$, and
$M^{\operatorname{Kps}}_{(4.2)}$ is a connected component of the scheme $M^{\operatorname{Kps}}_{3,28}$.
\end{corollary}

\begin{corollary}
\label{corollary:3-9}
Let $V=\mathbb{P}^2$, $L=\mathcal{O}_{\mathbb{P}^2}(2)$,
let $\Gamma=\operatorname{SL}_3(\mathbb{C})$, let $T=\mathbb{P}\left(H^0\left(\mathbb{P}^2, \mathcal{O}_{\mathbb{P}^2}(4)\right)^\vee\right)$,
let $T^{\operatorname{ss}}\subset T$ be the GIT semistable open subset with respect to the natural $\Gamma$-action,
and let $M$ be the GIT quotient $T^{\operatorname{ss}}\mathbin{/\mkern-6mu/}\Gamma$.
Then there exists a morphism
\begin{eqnarray*}
    \Phi\colon M &\to&M^{\operatorname{Kps}}_{3,26} \\
        \vin  \     &   & \ \ \vin \\
    {[}f{]} &\mapsto& {[}X_f{]},
\end{eqnarray*}
where $X_f$ is the Casagrande--Druel threefold that is constructed from $R=\{f=0\}\in |2L|$.
Furthermore, the morphism $\Phi$ is an isomorphism onto $M^{\operatorname{Kps}}_{(3.9)}$,
and $M^{\operatorname{Kps}}_{(3.9)}$ is a connected component of the scheme $M^{\operatorname{Kps}}_{3,26}$.
\end{corollary}

If $B$ is the~smooth del Pezzo surface from Examples~\ref{example:3-19}, \ref{example:3-9}, \ref{example:4-2},
then $B$ is K-polystable.
If $B$ is the~Fano manifold from Theorem~\ref{theorem:double-spaces},
then $B$ is K-polystable \cite[Theorem 1.1]{Dervan2}.
If~$B$ is the~singular del Pezzo surface from Theorems~\ref{theorem:4-2} and \ref{theorem:3-9}
such that $R$ is reduced, then $B$ is also K-polystable \cite{OdakaSpottiSun}.
Inspired by this, we pose

\begin{conjecture}
\label{conjecture:CS-K-stability}
If $V$ and $B$ are K-polystable Fano varieties, then $X$ is K-polystable.
\end{conjecture}

If $B$ is a K-polystable Fano variety, the log Fano pair $(V,\frac{1}{2}R)$ is also K-polystable \cite{Liu-Zhu}.
Thus, our conjecture is closely related to the following recent result:

\begin{theorem}[{\cite{Mallory}}]
\label{theorem:Yuchen-Daniel}
Suppose that $-K_V\sim_{\mathbb{Q}} aL$, where $a\in\mathbb{Q}_{>0}$ such that $a>1$.
Set
$$
\lambda_n(a)=\frac{a^{n+1}-(a+n)(a-1)^n}{2(n+1)(a^n-(a-1)^n)},
$$
where $n=\mathrm{dim} X$. Then $X$ is K-semistable $\iff$ $(V,\lambda_n(a)R)$ is K-semistable.
\end{theorem}

The K-polystability of $V$ in Conjecture \ref{conjecture:CS-K-stability} is necessary.

\begin{example}[Yuchen Liu]
\label{example:Yuchen}
Let $V=\mathbb{P}(1,1,4)$, let $L=\mathcal{O}_{V}(4)$, let $R$ be a general curve in $|2L|$,
and let $\lambda\in\left(0,\frac{3}{4}\right)\cap\mathbb{Q}$.
Then $(V,\lambda R)$ is a log Fano pair. One can show that
\begin{center}
$\delta(V,\lambda R)\geqslant 1$ ($\delta(V,\lambda R)>1$, respectively) $\iff$ $\lambda\geqslant\frac{3}{8}$ ($\lambda>\frac{3}{8}$, respectively),
\end{center}
so that the singular del Pezzo surface $B$ is K-polystable,
but $\left(V,\frac{9}{52}R\right)$ is not K-semistable.
Hence, the threefold $X$ is not K-semistable by Theorem \ref{theorem:Yuchen-Daniel}.
\end{example}

Let us say few words about the~proofs of Theorems~\ref{theorem:delta} and \ref{theorem:3-9}.
In Section~\ref{section:CD-properties}, we will show that
$X/\iota\cong Y$, and we have the~following commutative diagram:
$$
\xymatrix{
X\ar[dr]_{\theta}\ar[rr]^{\rho}&&Y\ar[dl]^{\pi} \\
&V&}
$$
where $\rho$ is the~quotient map, which is a double cover ramified over our divisor $B\in |2S^+|$.
Thus, using \cite{Liu-Zhu}, we see that
\begin{center}
$X$ is K-polystable $\iff$ the~log Fano pair $\left(Y, \frac{1}{2}B\right)$ is K-polystable.
\end{center}
In Section~\ref{section:delta}, we will prove the~following result,
which implies Theorem~\ref{theorem:delta}.

\begin{theorem}
\label{theorem:log-delta}
Suppose that $V$ and $R$ are smooth (so $B$ is smooth), and $-K_V\sim_{\mathbb{Q}} aL$,
where $a\in\mathbb{Q}_{>0}$ such that $a>1$.
Let $\mu$ be a~rational number such that  $\mu L$ is very ample.
Set~$n=\mathrm{dim} Y$ (so $\mathrm{dim} V =n-1$) and $d=L^{n-1}$.
Suppose $n\geqslant 3$ and $d\mu^{n-2}\geqslant 2$.
Then
$$
\delta\Big(Y,\frac{1}{2}B\Big)\geqslant\mathrm{min}\Bigg\{\frac{1}{k_n(a,d,\mu)},\frac{(n+1)(a^n-(a-1)^n)}{(n+1-a)a^n+(a-1)^{n+1}},\frac{a\delta(V)(n+1)(a^{n}-(a-1)^{n})}{n(a^{n+1}-(a-1)^{n+1})}\Bigg\},
$$
where $k_n(a,d,\mu)$ is defined in Theorem~\ref{theorem:delta}.
\end{theorem}

Let us describe the~structure of this paper.
First, in Section~\ref{section:CD-properties}, we will prove few basic properties of Casagrande--Druel varieties.
Then, in Section~\ref{section:delta}, we will prove Theorem~\ref{theorem:log-delta}.
In Sections~\ref{section:4-2} and \ref{section:3-9}, we will give proofs of Theorem~\ref{theorem:4-2} and Theorem~\ref{theorem:3-9}, respectively.
Finally, in Section \ref{section:moduli}, we will prove Corollary~\ref{corollary:4-2}, and we will show that $M^{\operatorname{Kps}}_{(4.2)}\cong\mathbb{P}(1,2,3)$.
We~will omit the proof of Corollary~\ref{corollary:3-9}, since it is similar to the proof of Corollary~\ref{corollary:4-2}.

\medskip
\noindent
\textbf{Acknowledgements.}
This paper was written during our visit to the~G\"okova Geometry Topology Institute in April 2023.
We are very grateful to the~institute for its hospitality.
We would like to thank Yuchen Liu for his help with the~proof of Corollary~\ref{corollary:4-2},
and we would like to thank Noam Elkies for his proof of Proposition~\ref{proposition:Elkies}.

Ivan Cheltsov was supported by EPSRC grant EP/V054597/1,
Tiago Duarte Guerreiro was supported by EPSRC grant EP/V055399/1, Kento Fujita was supported by JSPS KAKENHI Grant Number 22K03269,
Igor Krylov was supported by IBS-R003-D1 grant, and Jesus Martinez-Garcia was supported by EPSRC grant EP/V055399/1.

\section{Preliminaries}
\label{section:CD-properties}

Let $V$ be a (possibly non-projective) variety, let $L_1$ and $L_2$ be line bundles on $V$ such that $L_1+L_2\not\sim 0$ and $|L_1+L_2|\ne\varnothing$,
and let $f\in H^0(V,\mathcal{O}_V(L_1+L_2))$ that defines a~nonzero effective divisor $R$ on $V$.
Set
\begin{align*}
Y_1&=\mathbb{P}\big(\mathcal{O}_V\oplus\mathcal{O}(L_1)\big),\\
Y_2&=\mathbb{P}\big(\mathcal{O}_V\oplus\mathcal{O}(L_2)\big).
\end{align*}
Now, let $\pi_1\colon Y_1\to V$ and $\pi_2\colon Y_2\to V$ be the~natural projections,
and let $\xi_1$ and $\xi_2$~be the~tautological line bundles on $Y_1$ and $Y_2$, respectively.
We have isomorphisms:
\begin{align*}
H^0\big(Y_1, \mathcal{O}_{Y_1}(\xi_1)\big)&\cong  H^0\big(V,\mathcal{O}_V\big)\oplus H^0\big(V,\mathcal{O}_V(L_1)\big),\\
H^0\big(Y_1, \mathcal{O}_{Y_1}(\xi_1-\pi_1^*(L_1))\big)&\cong H^0\big(V,\mathcal{O}_V\big)\oplus H^0\big(V,\mathcal{O}_V(-L_1)\big),\\
H^0\big(Y_2, \mathcal{O}_{Y_2}(\xi_2)\big)&\cong H^0\big(V,\mathcal{O}_V\big)\oplus H^0\big(V,\mathcal{O}_V(L_2)\big),\\
H^0\big(Y_2, \mathcal{O}_{Y_2}(\xi_2-\pi_2^*(L_2))\big)&\cong H^0\big(V,\mathcal{O}_V\big)\oplus H^0\big(V,\mathcal{O}_V(-L_2)\big).
\end{align*}
Using these isomorphisms, fix sections
\begin{align*}
u_{1}^+& \in H^0\big(Y_1,\mathcal{O}_{Y_1}(\xi_1)\big),\\
u_{1}^-&\in H^0\big(Y_1, \mathcal{O}_{Y_1}(\xi_1-\pi_1^*(L_1))\big),\\
u_{2}^+&\in H^0\big(Y_2, \mathcal{O}_{Y_2}(\xi_2)\big),\\
u_{2}^-&\in H^0\big(Y_2, \mathcal{O}_{Y_2}(\xi_2-\pi_2^*(L_2))\big),
\end{align*}
that correspond to the~section $1\in H^0(V,\mathcal{O}_V)$.
Let
\begin{align*}
S_1^-&=\{u_1^-=0\}\subset Y_1,\\
S_1^+&=\{u_1^+=0\}\subset Y_1,\\
S_2^-&=\{u_2^-=0\}\subset Y_2,\\
S_2^+&=\{u_2^+=0\}\subset Y_2.
\end{align*}
For $i\in\{1,2\}$, the~divisors $S_i^-$ and $S_i^+$ are disjoint sections of the~natural projection~$\pi_i$
such that $S_i^-\vert_{S_i^-}\sim -L_i\sim -S_i^+\vert_{S_i^+}$, where we use isomorphisms $S_i^-\cong V\cong S_i^+$ induced~by~$\pi_i$.

Now, we set $Q=Y_1\times_V Y_2$.
Then we have canonical isomorphisms
$$
\mathbb{P}\Big(\mathcal{O}_{Y_1}\oplus\mathcal{O}_{Y_1}\big(\pi_1^*(L_2)\big)\Big)\cong Q\cong\mathbb{P}\Big(\mathcal{O}_{Y_2}\oplus\mathcal{O}_{Y_2}\big(\pi_2^*(L_1)\big)\Big),
$$
so that we have commutative Cartesian diagram
$$
\xymatrix{
&&Q\ar[dll]_{\rho_1}\ar[drr]^{\rho_2} &&\\
Y_1\ar[drr]_{\pi_1} &&&&Y_2\ar[dll]^{\pi_2} \\
&&V&&}
$$
where $\rho_1$ and $\rho_2$ are natural projections. Set $\vartheta=\pi_1\circ\rho_1=\pi_2\circ\rho_2$.

Set $F_1=\pi_1^*(R)\subset Y_1$.
Let $\phi_1\colon X\to Y_1$ be the~blowup along the~intersection $F_1\cap S_1^+$,
and let $E_1$ be the~$\phi_1$-exceptional divisor.
Since $F_1+S_{1}^-$ corresponds to
$$
\pi_1^*(f)u_{1}^-\in H^0\left(Y_1,\mathcal{O}_{Y_1}\big(\xi_1+\pi_1^*(L_2)\right),
$$
there~is~a~natural closed embedding $X\hookrightarrow Q$ over $V$ such that its image is the~effective divisor
defined by the~zeroes of the~section
$$
\vartheta^*(f)u_{1}^-u_{2}^--u_{1}^+u_{2}^+\in H^0\Big(Q,\mathcal{O}_{Q}\big(\rho_1^*(\xi_1)+\rho_2^*(\xi_2)\big)\Big),
$$
where we identified $H^0(Q,\mathcal{O}_{Q}(\rho_i^*(D)))=H^0(Y_i,\mathcal{O}_{Y_i}(D))$ for every $D\in\mathrm{Pic}(Y_i)$.

Let us identify $X$ with its image in $Q$. Set $\theta=\pi_1\circ\phi_1$.
Then $\theta$ is induced by $\vartheta$,
it is a~conic bundle, and $R$ is its discriminant divisor.~Set
\begin{align*}
S_1&=\phi_1^*(S_{1}^-),\\
S_2&=\phi_1^*(S_{1}^+)-E_1,\\
E_2&=\phi_1^*(F_1)-E_1.
\end{align*}
Then $S_1$, $S_2$, $E_2$ are effective Cartier divisors on the variety $X$
--- these are the~proper transforms of the~divisors $S_1^-$, $S_1^+$, $F_1$, respectively.
Moreover, the~divisors $S_1$ and $S_2$ are mutually disjoint sections of the~conic bundle $\theta$.
Furthermore, we have
\begin{center}
$S_1\big\vert_{S_1}\sim -L_1$ and $S_2\big\vert_{S_2}\sim -L_2$
\end{center}
where we use isomorphisms $S_1\cong V$ and $S_2\cong V$ induced by $\theta$.
Similarly, we see that the~divisor $E_1+E_2$ is given by zeroes of the~section
$$
\theta^*(f)\in H^0\Big(X,\mathcal{O}_X\big(\theta^*(L_1+L_2)\big)\Big)\cong H^0\big(V,\mathcal{O}_V(L_1+L_2)\big).
$$

Set $F_2=\pi_2^*(R)\subset Y_2$,
and let $\phi_2\colon X\to Y_2$ be the~morphism induced by $\rho_2\colon Q\to Y_2$.
Since the~defining equation of $X\subset Q$ is symmetric,
we conclude that $\phi_2$ is the~blowup along the~scheme-theoretic intersection $F_2\cap S_2^+$,
the~$\phi_2$-exceptional divisor is $E_2$, and there exists the~following commutative diagram:
\begin{equation}
\label{equation:diagram}
\xymatrix{
&&X\ar[dll]_{\phi_1}\ar[drr]^{\phi_2}\ar[dd]_{\theta}&&\\
Y_1\ar[drr]_{\pi_1} &&&&Y_2\ar[dll]^{\pi_2} \\
&&V&&}
\end{equation}
This is an~elementary transformation of the~$\mathbb{P}^1$-bundle $\pi_1$ in the~sense of Maruyama~\cite{maruyama2}.
Now, using \cite[Theorem 1.4]{maruyama2} and \cite[Proposition 1.6]{maruyama2}, we see that
\begin{align*}
S_1&=\phi_2^*(S_{2}^+)-E_2,\\
S_2&=\phi_2^*(S_{2}^-),\\
E_1&=\phi_2^*(F_1)-E_2.
\end{align*}

\begin{remark}
\label{remark:Kento}
Let $U=\mathbb{P}(\mathcal{O}_V\oplus\mathcal{O}_V(-L_1)\oplus\mathcal{O}_V(-L_2))$,
let $\xi_U$ be the tautological line bundle on the~variety $U$,
let $\pi_U\colon U\to V$ be the natural projection.
We have isomorphisms:
\begin{align*}
H^0\big(U,\mathcal{O}_U(\xi_U)\big)&\cong H^0\big(V, \mathcal{O}_V\big)\oplus H^0\big(V, \mathcal{O}_V(-L_1)\big)\oplus H^0\big(V,\mathcal{O}_V(-L_2)\big), \\
H^0\big(U,\mathcal{O}_U(\xi_U+\pi_U^*(L_1))\big)&\cong H^0\big(V, \mathcal{O}_V\big)\oplus  H^0\big(V, \mathcal{O}_V(L_1)\big)\oplus H^0\big(V,\mathcal{O}_V(L_1-L_2)\big), \\
H^0\big(U,\mathcal{O}_U(\xi_U+\pi_U^*(L_2))\big)&\cong H^0\big(V,\mathcal{O}_V\big)\oplus H^0\big(V, \mathcal{O}_V(L_2)\big)\oplus H^0\big(V, \mathcal{O}_V(L_2-L_1\big).
\end{align*}
Using these isomorphisms, fix sections
\begin{align*}
v_0&\in H^0\big(U,\mathcal{O}_U(\xi_U)\big),\\
v_1&\in H^0\big(U,\mathcal{O}_U(\xi_U+\pi_U^*(L_1))\big),\\
v_2&\in H^0\big(U,\mathcal{O}_U(\xi_U+\pi_U^*(L_2))\big),
\end{align*}
which correspond to the section $1\in H^0(V,\mathcal{O}_V)$.
One can show that there exists a closed embedding $X\hookrightarrow U$ over $V$
such that the image of $X$ is defined by
$$
\pi_U^*(f)v_0^2-v_1v_2=0,
$$
so that we can idendity $X$ with a Cartier divisor on $U$ such that $X\sim 2\xi_U+\pi_U^*(L_1+L_2)$.
\end{remark}

Starting from now, we assume, in addition, that $V$ is projective.

\begin{proposition}
\label{proposition:CD-Fano}
Suppose that $V$ is normal, and $K_V$ is $\mathbb{Q}$-Cartier.
Then $X$~is~normal, and $K_X$ is $\mathbb{Q}$-Cartier.
Moreover, the following assertion holds:
\begin{center}
$-K_X$ is ample $\iff$ $-K_V$, $-K_V-L_1$, $-K_V-L_2$ are~ample.
\end{center}
\end{proposition}

\begin{proof}
The normality of the~variety $X$ follows from Remark~\ref{remark:Kento} and \cite[Proposition~5.24]{SwansonHuneke}.
Similarly, using notations introduced in Remark~\ref{remark:Kento}, we see that
$$
K_U\sim_{\mathbb{Q}}-3\xi_U+\pi_U^*\big(K_V-L_1-L_2),
$$
so $K_X$ is $\mathbb{Q}$-Cartier by the adjunction formula,
because $X$ is a Cartier divisor on $U$.

To prove the remaining assertion,
suppose that $-K_V$, $-K_V-L_1$, $-K_V-L_2$ are
ample. Then
$\xi_U+\pi^*_U(-K_V)$ in Remark~\ref{remark:Kento} is
ample. Then so is
$-K_X\sim_{\mathbb{Q}}\left(\xi_U+\pi^*_U(-K_V)\right)|_X$.
Alternatively, we can prove the ampleness of $-K_X$ directly.
Namely, observe that
\begin{equation}
\label{equation:anticanonical-divisor}
-K_X\sim_{\mathbb{Q}} S_1+S_2+\theta^*(-K_V).
\end{equation}
Moreover, applying the~adjunction formula to the~sections $S_1$ and $S_2$, we get
\begin{align*}
-K_X\big\vert_{S_1}&\sim_{\mathbb{Q}}-K_V-L_1,\\
-K_X\big\vert_{S_2}&\sim_{\mathbb{Q}}-K_V-L_2,
\end{align*}
where we used $S_1\cong V$ and $S_2\cong V$.
Hence, if  $-K_V$, $-K_V-L_1$, $-K_V-L_2$ are ample,
then the~divisor $-K_X$ is also ample by Kleiman's ampleness criterion.

This also shows that both divisors $-K_V-L_1$ and $-K_V-L_2$ are ample if $-K_X$ is ample.
Observe that $E_1\cap E_2\cong R$.
Using this isomorphism and \eqref{equation:anticanonical-divisor}, we get $-K_V\vert_{R}\sim -K_X\vert_{R}$.
On the~other hand, we have
$$
-2K_V\sim_{\mathbb{Q}}\big(-K_V-L_1\big)+\big(-K_V-L_2\big)+R.
$$
Hence, using Kleiman's criterion again, we see that $-K_V$ is ample if $-K_X$ is ample.
\end{proof}

\begin{example}
\label{example:4-7}
Suppose  $V=\mathbb{P}^1\times\mathbb{P}^1$,
and $L_1$ and $L_2$ are divisors of degrees $(1,0)$ and~$(0,1)$,
and $R$ is a smooth divisor in $|L_1+L_2|$.
Then $X$ is a smooth Fano 3-fold by Proposition~\ref{proposition:CD-Fano}.
One can show that $X$ is the~unique smooth Fano 3-fold in the~deformation family \textnumero 4.7.
Note that $X$ is K-polystable \cite[\S 3.3]{Book}.
\end{example}

\begin{remark}[{\cite[Lemma 9.8]{div-stability}}]
\label{remark:K-instability}
Suppose that $V$ is a smooth Fano variety, and $-K_V\sim_{\mathbb{Q}} a L$,
where $L$ is an~ample divisor in $\mathrm{Pic}(V)$, and $a\in\mathbb{Q}_{>0}$.
Suppose $R$ and $X$ are smooth, and
\begin{align*}
L_1&\sim_{\mathbb{Q}} a_1L,\\
L_2&\sim_{\mathbb{Q}} a_2L,
\end{align*}
where $a_1$ and $a_2$ are rational numbers such that $a_1\geqslant a_2$.
It follows from Proposition~\ref{proposition:CD-Fano} that $X$ is a Fano variety $\iff$ $a>a_1$.
Further, if $X$ is a Fano~variety, then it follows from the~proof of \cite[Lemma 9.8]{div-stability} that
\begin{center}
$\beta(S_2)<0$ $\iff$ $a_1>a_2$.
\end{center}
Therefore, if $a>a_1>a_2$, then $X$ is a K-unstable Fano variety.
\end{remark}

From now on, we also assume that $L_1=L_2$. Set $L=L_1$. Then $R\in |2L|$.
Set
$$
Y=\mathbb{P}\big(\mathcal{O}_V\oplus\mathcal{O}(L)\big).
$$
let $\pi\colon Y\to V$ be the~natural projection,
and let $\xi$ be the~tautological line bundle on $Y$.
Note that $Y\cong Y_1\cong Y_2$.
Using the isomorphisms
\begin{align*}
H^0\big(Y, \mathcal{O}_{Y}(\xi)\big)&\cong  H^0\big(V,\mathcal{O}_V\big)\oplus H^0\big(V,\mathcal{O}_V(L)\big),\\
H^0\big(Y, \mathcal{O}_{Y}(\xi-\pi^*(L))\big)&\cong H^0\big(V,\mathcal{O}_V\big)\oplus H^0\big(V,\mathcal{O}_V(-L)\big),
\end{align*}
fix  $u^+\in H^0(Y,\mathcal{O}_{Y}(\xi))$ and $u^-\in H^0(Y, \mathcal{O}_{Y}(\xi-\pi^*(L)))$
that correspond to $1\in H^0(V,\mathcal{O}_V)$.
Let $S^-=\{u^-=0\}$ and $S^+=\{u^+=0\}$. Then $S^+\sim S^-+\pi^*(L)$.

\begin{proposition}
\label{proposition:quotient}
There is a double cover $X\to Y$ ramified in a divisor $B\in |2S^+|$ such that
the projection $\pi$ induces a double cover $B\to V$ that is ramified in $R$.
\end{proposition}

\begin{proof}
Let $T=\mathbb{P}(\mathcal{O}_V\oplus\mathcal{O}_V(-L))\oplus\mathcal{O}_V(-2L))$,
let $\varpi\colon T\to V$ be the~natural projection, and let $\xi_T$ be the~tautological line bundle on~$T$.
Observe that
\begin{align*}
H^0\big(T, \mathcal{O}_T(\xi_T)\big) &\cong H^0\big(V, \mathcal{O}_V\big)\oplus H^0\big(V,\mathcal{O}_V(-L)\big)\oplus H^0\big(V,\mathcal{O}_V(-2L)\big), \\
H^0\big(T, \mathcal{O}_T(\xi_T+\varpi^*(L))\big)&\cong H^0\big(V, \mathcal{O}_V\big)\oplus H^0\big(V, \mathcal{O}_V(L)\big)\oplus H^0\big(V, \mathcal{O}_V(-L)\big), \\
H^0\big(T, \mathcal{O}_T(\xi_T+\varpi^*(2L)\big) &\cong H^0\big(V, \mathcal{O}_V\big)\oplus H^0\big(V, \mathcal{O}_V(2L)\big)\oplus H^0\big(V, \mathcal{O}_V(L)\big).
\end{align*}
Using these isomorphisms, fix sections
\begin{align*}
t_0&\in H^0\big(T, \mathcal{O}_T(\xi_T)\big),\\
t_1&\in H^0\big(T, \mathcal{O}_T(\xi_T+\varpi^*(L))\big),\\
t_2&\in H^0\big(T, \mathcal{O}_T(\xi_T+\varpi^*(2L)\big)
\end{align*}
that corresponds to $1\in H^0(V,\mathcal{O}_V)$.
Then
\begin{align*}
\{t_0=0\}&\cong\mathbb{P}\big(\mathcal{O}_V(-L))\oplus\mathcal{O}_V(-2L)\big),\\
\{t_1=0\}&\cong\mathbb{P}\big(\mathcal{O}_V\oplus\mathcal{O}_V(-2L)\big),\\
\{t_2=0\}&\cong\mathbb{P}\big(\mathcal{O}_V\oplus\mathcal{O}_V(-L))\big).
\end{align*}

Now, we consider the~homomorphism
\begin{equation}
\label{involution-quotient}
\mathcal{O}_Q\oplus \mathcal{O}_Q\big(\vartheta^*(L)\big)\oplus\mathcal{O}_Q\big(\vartheta^*(2L)\big)\to \mathcal{O}_{Q}\big(\rho_1^*(\xi_1)+\rho_2^*(\xi_2)\big)
\end{equation}
defined by the~composition of
$$\renewcommand\arraystretch{1.3}
\begin{pmatrix}
1 & 0 & 0\\
0 & \frac{1}{2} & 0\\
0 & \frac{1}{2} & 0\\
0 & 0 & 1
\end{pmatrix}\colon
\mathcal{O}_Q\oplus \mathcal{O}_Q\big(\vartheta^*(L)\big)\oplus\mathcal{O}_Q\big(\vartheta^*(2L)\big)\to
\mathcal{O}_Q\oplus\mathcal{O}_Q\big(\vartheta^*(L)\big)\oplus\mathcal{O}_Q\big(\vartheta^*(L)\big)\oplus\mathcal{O}_Q\big(\vartheta^*(2L)\big)
$$
and the~surjection
$$
\mathcal{O}_Q\oplus\mathcal{O}_Q\big(\vartheta^*(L)\big)\oplus\mathcal{O}_Q\big(\vartheta^*(L)\big)\oplus\mathcal{O}_Q\big(\vartheta^*(2L)\big)\twoheadrightarrow \mathcal{O}_{Q}\big(\rho_1^*(\xi_1)+\rho_2^*(\xi_2)\big)
$$
obtained by the~tensor product of the~pullbacks of the~following
natural surjections
\begin{align*}
\mathcal{O}_{Y_1}\oplus\mathcal{O}_{Y_1}\big(\pi_1^*(L_1)\big)&\twoheadrightarrow\mathcal{O}_{Y_1}(\xi_1), \\
\mathcal{O}_{Y_2}\oplus\mathcal{O}_{Y_2}\big(\pi_2^*(L_2)\big)&\twoheadrightarrow\mathcal{O}_{Y_2}(\xi_2).
\end{align*}
Then~\eqref{involution-quotient} is surjective.
This gives the~morphism $\rho\colon Q\to T$ over $V$
with
\begin{align*}
\rho^*(t_0)&=u_{1}^-u_{2}^-,\\
\rho^*(t_1)&=\frac{1}{2}\left(u_{1}^+u_{2}^-+u_{1}^-u_{2}^+\right),\\
\rho^*(t_2)&=u_{1}^+u_{2}^+,
\end{align*}
where we identified $H^0(Q,\mathcal{O}_{Q}(\rho_i^*(D)))=H^0(Y_i,\mathcal{O}_{Y_i}(D))$ for $D\in\mathrm{Pic}(Y_i)$.

Using the~local criterion for flatness, we see that $\rho$ is flat.
Further, $\rho$ is finite of degree~$2$.
Now, using \cite[I~(6.11)]{takao} and \cite[I~(6.12)]{takao},
we see that the~morphism $\rho$ is branched over the~divisor $B_T\in|2(\xi_T+\varpi^*(L))|$ that is given by $t_1^2-t_0t_2=0$.

Let $Y_0$ be the~divisor in $|\xi_T+\varpi^*(2L)|$ that is given by
$$
\varpi^*(f)t_0-t_2=0,
$$
and let $\pi_0\colon Y_0\to V$ be the~morphism induced by $\varpi$.
Then $X=\rho^*(Y_0)$  as Cartier divisors,
so that the~restriction $X\to Y_0$ is a double cover branched over $B_T\vert_{Y_0}$.
Moreover, using the~exact sequence
$$
0\to\mathcal{O}_{V}(-2L)\xrightarrow{\begin{pmatrix}f\\0\\-1\end{pmatrix}}
\mathcal{O}_V\oplus\mathcal{O}_{V}(-L)\oplus\mathcal{O}_{V}(-2L)\xrightarrow{\begin{pmatrix}
1&0&f\\0&1&0\end{pmatrix}}\mathcal{O}_V\oplus\mathcal{O}_{V}(-L)\to 0,
$$
we get an isomorphism $Y_0\cong Y$ over $V$. Hence, we identify $Y=Y_0$.

Set $B=B_T\vert_Y$. Then $B$ is defined by
$$
(u^+)^2-\pi^*(f)(u^-)^2=0,
$$
which implies the remaining assertions of the proposition.
\end{proof}

Let $\iota\in\mathrm{Aut}(X)$ be the~Galois involution of the~double cover $X\to Y$ in \mbox{Proposition~\ref{proposition:quotient}}.
Then $\iota(S_1)=S_2$ and $\iota(E_1)=E_2$,
and it follows from the~proof of Proposition~\ref{proposition:quotient}
that the~conic bundle $\theta\colon X\to V$ is $\langle\iota\rangle$-equivariant with $\iota$ acting trivially on $V$.

\begin{proposition}
\label{proposition:unobstruction}
Suppose that $V$ is smooth, $L$ is nef, $X$ has Kawamata log terminal singularities, and $-K_X$ is ample.
Then the deformations of $X$ are unobstructed.
\end{proposition}

\begin{proof}
By Remark \ref{remark:Kento}, $X$ can be embedded into $U=\mathbb{P}_V\left(\mathcal{O}_V\oplus \mathcal{O}_V(-L)\oplus \mathcal{O}_V(-L)\right)$
such that $X\in |2\xi_U+2\pi^*_U(L)|$, where $\xi_U$ is the tautological line bundle and $\pi_U$ is the natural projection.
Therefore, since $U$ is smooth, the variety $X$ has at worst canonical singularities,
and $X$ has at worst local complete intersection singularities.
Hence, it follows from \cite[Theorem 2.3.2]{sernesi}, \cite[Theorem~2.4.1]{sernesi}, \cite[Corollary 2.4.2]{sernesi},
\cite[Proposition~2.4]{taro}, \cite[Proposition~2.6]{taro} that
the deformations of $X$ are unobstructed if
$\operatorname{Ext}^2_{\mathcal{O}_X}\left(\Omega_X^1,\mathcal{O}_X\right)=0$.

Let us show that $\operatorname{Ext}^2_{\mathcal{O}_X}\left(\Omega_X^1,\mathcal{O}_X\right)=0$.
Set $n=\mathrm{dim}(X)$. As in \cite[\S 1.2]{taro}, we have
$$
\operatorname{Ext}^2_{\mathcal{O}_X}
\left(\Omega_X^1,\mathcal{O}_X\right)\simeq
\operatorname{Ext}^2_{\mathcal{O}_X}
\left(\Omega_X^1\otimes\omega_X,\omega_X\right)\simeq
H^{n-2}\left(X, \Omega_X^1\otimes\omega_X\right)^\vee.
$$
Since $-K_V$ and $-K_V-L$ are ample and $L$ is nef, we see that $\xi_U+\pi^*_U(-K_V)$ is~ample, and $\xi_U+\pi^*_U(L)$ is nef.
In particular, both divisors
\begin{eqnarray*}
    -K_U&\sim&3\xi_U+\pi^*_U(-K_V+2L), \\
    -K_U-X&\sim&\xi_U+\pi^*_U(-K_V)
\end{eqnarray*}
are ample. On the other hand, using the exact sequence of sheaves
$$
0\longrightarrow\mathcal{O}_U(-X)\big|_X\longrightarrow\Omega_U^1\big|_X\longrightarrow\Omega_X^1\longrightarrow 0,
$$
we get the following exact sequence:
$$
H^{n-2}\left(X, \Omega_U^1\big|_X\otimes\omega_X\right)\longrightarrow
H^{n-2}\left(X,\Omega_X^1\otimes\omega_X\right)\longrightarrow
H^{n-1}\left(X,\mathcal{O}_U(-X)\big|_X\otimes\omega_X\right).
$$
Moreover, using the Kodaira-type vanishing theorem, we get
$$
H^{n-1}\left(X,\mathcal{O}_U(-X)\big|_X\otimes\omega_X\right)\simeq H^1\left(X, K_X+(-K_U)\big|_X\right)^\vee=0.
$$
Furthermore, using the~exact sequence of sheaves
$$
0\longrightarrow \Omega_U^1\otimes\omega_U\longrightarrow\Omega_U^1\otimes\omega_U(X)\longrightarrow\Omega_U^1\big|_X\otimes\omega_X\longrightarrow 0,
$$
we get the exact sequence
$$
H^{n-2}\left(U, \Omega_U^1\otimes\omega_U(X)\right)\longrightarrow
H^{n-2}\left(X,\Omega_U^1|_X\otimes\omega_X\right)\longrightarrow
H^{n-1}\left(U,\Omega_U^1\otimes\omega_U\right).
$$
Since both $\omega_U$ and $\omega_U(X)$ are anti-ample,
the Akizuki--Nakano vanishing theorem gives
$$
H^{n-2}\left(U, \Omega_U^1\otimes\omega_U(X)\right)=H^{n-1}\left(U,\Omega_U^1\otimes\omega_U\right)=0.
$$
This gives $\operatorname{Ext}^2_{\mathcal{O}_X}\left(\Omega_X^1,\mathcal{O}_X\right)=0$, which completes the proof.
\end{proof}

\section{K-polystability criteria}
\label{section:delta}

The goal of this section is to prove Theorem~\ref{theorem:log-delta}.
To do this, fix a positive integer~$n\geqslant 3$.
Let $V$ be a smooth projective variety of dimension $n-1$, and let $L$ be an ample Cartier divisor on $V$.
Set $d=L^{n-1}$. Fix $\mu\in\mathbb{Q}_{>0}$ such that $\mu L$ is very ample. Let
$$
Y=\mathbb{P}\big(\mathcal{O}_{V}\oplus \mathcal{O}_{V}(L)\big),
$$
and let $\pi\colon Y\to V$ be the~natural projection.
Set $H=\pi^*(L)$. Let $S^-$ and $S^+$ be disjoint sections of the~projection $\pi$ such that $S^+\sim S^-+H$.

\begin{remark}
\label{remark:warning}
Unlike Section~\ref{section:intro}, we do not assume that $V$ is a Fano variety.
\end{remark}

Fix a positive rational number $a\geqslant 1$. Let $D(a)=S^-+aH$. Then $D(a)$ is nef and big.
Moreover, if $a>1$, then $D(a)$ is ample.

\begin{lemma}[{cf. \cite{ZhangZhou}}]
\label{lemma:Kento}
Let $P$ be a point in $S^-$. Then
$$
\delta_P(Y;D(a))\geqslant\mathrm{min}\Bigg\{\frac{(n+1)(a^n-(a-1)^n)}{(n+1-a)a^n+(a-1)^{n+1}},\frac{\delta(V;L)(n+1)(a^{n}-(a-1)^{n})}{n(a^{n+1}-(a-1)^{n+1})}\Bigg\},
$$
where $\delta_P(Y;D(a))$  is the~(local) $\delta$-invariant of the~variety $Y$ polarized by the~divisor $D(a)$,
and $\delta(V;L)$ is  the~$\delta$-invariant of $V$ polarized by $L$. Further, if $\delta(V;L)\leqslant a$, then
$$
\delta_P(Y;D(a))\geqslant\frac{\delta(V;L)(n+1)(a^{n}-(a-1)^{n})}{n(a^{n+1}-(a-1)^{n+1})}.
$$
\end{lemma}

\begin{proof}
It follows from \cite{AbbanZhuang,Book} that
$$
\delta_P(Y;D(a))\geqslant\min\left\{\frac{1}{S_{D(a)}(S^-)},\inf_{\substack{F/S^-\\P\in C_{S^-}(F)}} \frac{A_{S^-}(F)}{S(W^{S^-}_{\bullet,\bullet};F)}\right\},
$$
where $S(W^{S^-}_{\bullet,\bullet};F)$ is defined in \cite[Section~1.7]{Book},
and the~infimum is taken over all prime divisors over $S^-$ whose centers on $S^-$ contain $P$.
This easily implies the~required assertion.

Indeed, take $u\in\mathbb{R}_{\geqslant 0}$.
Then $D(a)-uS^-\sim_{\mathbb{R}}(1-u)S^-+aH$, so that
\begin{center}
$D(a)-uS^-$ is nef $\iff$ $D(a)-uS^-$ is pseudo-effective $\iff$ $u\leqslant 1$.
\end{center}
Thus, since $\mathrm{vol}(D(a))=D(a)^n=d(a^n-(a-1)^n)$, we have
\begin{multline*}
S_{D(a)}(S^-)=\frac{1}{D(a)^n}\int_0^{\infty} \mathrm{vol}(D(a)-uS^-)du=\\
=\frac{1}{d(a^n-(a-1)^n)}\int_0^{1}((1-u-a)^n(-1)^{n+1}d+a^nd)du=\frac{(n+1-a)a^n+(a-1)^{n+1}}{(n+1)(a^n-(a-1)^n)}.
\end{multline*}
Using $S^-\cong V$, we get $(D(a)-uS^-)\vert_{S^-}\sim_{\mathbb{R}}(a+u-1)H\vert_{S^-}\sim_{\mathbb{R}}(a+u-1)L$.

Let $F$ be any prime divisor over $S^-$. Then it follows from \cite[Section~1.7]{Book} that
\begin{align*}
    S(W^{S^-}_{\bullet,\bullet};F) &= \frac{n}{D(a)^n}\int_0^1\int_0^{\infty} \mathrm{vol}((D(a)-uS^-)\vert_{S^-}-vF)dvdu \\
    &=\frac{n}{D(a)^n}\int_0^1\int_0^{\infty} \mathrm{vol}((a+u-1)L-vF)dvdu \\
     &=\frac{n}{D(a)^n}\int_0^1(a+u-1)^n\int_0^{\infty} \mathrm{vol}(L-vF)dvdu \\
      &=\frac{n}{d(a^n-(a-1)^n)} \cdot \frac{a^{n+1}-(a-1)^{n+1}}{n+1}\int_0^{\infty} \mathrm{vol}(L-vF)dv \\
       &=\frac{n}{n+1}\frac{a^{n+1}-(a-1)^{n+1}}{d(a^n-(a-1)^n)} \cdot L^{n-1} S_L(F) \\
       &=\frac{n}{n+1}\frac{a^{n+1}-(a-1)^{n+1}}{a^n-(a-1)^n} S_L(F).
\end{align*}
This gives
$$
\frac{A_{S^-}(F)}{ S(W^{S^-}_{\bullet,\bullet};F)} = \frac{A_{S^-}(F)}{ S_L(F)} \cdot \frac{n+1}{n}\cdot \frac{a^n-(a-1)^n}{a^{n+1}-(a-1)^{n+1}}\leqslant\delta_P(V;L)\cdot \frac{n+1}{n}\cdot \frac{a^n-(a-1)^n}{a^{n+1}-(a-1)^{n+1}},
$$
which implies the~first part of the~assertion.

We now assume $\delta(V;L) \leqslant a$ and we want to show
$$
\frac{(n+1)(a^n-(a-1)^n)}{(n+1-a)a^n+(a-1)^{n+1}}\geqslant \frac{\delta(V;L)(n+1)(a^{n}-(a-1)^{n})}{n(a^{n+1}-(a-1)^{n+1})}.
$$
This inequality is equivalent to
$$
\delta(V;L) \leqslant \frac{n(a^{n+1}-(a-1)^{n+1})}{(n+1-a)a^n+(a-1)^{n+1}}
$$
We must show that the~right hand side of the~inequality above is at least $a$. But
$$
\frac{n(a^{n+1}-(a-1)^{n+1})}{(n+1-a)a^n+(a-1)^{n+1}} > a \iff a^{n+1}(a-1)-(a-1)^{n+1}(a+n)>0,
$$
which is clearly true.
\end{proof}

Now, fix a smooth divisor $B\in |2S^+|$.
Let $\eta\colon B\to V$ be the~morphism induced by $\pi$.
Suppose that $\eta$ is the double cover ramified over a smooth divisor $R\in |2L|$.
Set
$$
\Delta=\frac{1}{2}B.
$$
Note that $B\cap S^-=\varnothing$.  Let $k_n(a,d,\mu)$ be the~number defined in Theorem~\ref{theorem:delta}.

\begin{proposition}
\label{proposition:delta}
Let $P$ be a point in $Y\setminus S^-$. Suppose that $d\mu^{n-2}\geqslant 2$. Then
$$
\delta_P(Y,\Delta;D(a))\geqslant \frac{1}{k_n(a,d,\mu)},
$$
where $\delta_P(Y,\Delta;D(a))$  is the~(local) $\delta$-invariants of the~pair  $(Y,\Delta)$ polarized by $D(a)$.
\end{proposition}

This result together with Lemma~\ref{lemma:Kento} implies Theorem~\ref{theorem:log-delta}.

\begin{proof}[Proof of Theorem~\ref{theorem:log-delta}.]
Note
that $V$ is a Fano variety and $-K_V\sim_{\mathbb{Q}} aL$.
Then
$$
-K_Y\sim 2S^+-\pi^*(K_V+L)\sim_{\mathbb{Q}} 2S^++(a-1)H,
$$
which gives
$$
-(K_Y+\Delta)\sim_{\mathbb{Q}} S^++(a-1)H\sim_{\mathbb{Q}} S^-+aH=D(a),
$$
so that $(Y,\Delta)$ is the~log Fano pair and
$$
\delta(Y,\Delta)=\delta(Y,\Delta;D(a)),
$$
where $\delta(Y,\Delta)$ is the~$\delta$-invariant of the~log Fano pair $(Y,\Delta)$.
Now, we can apply Lemma~\ref{lemma:Kento} and Proposition~\ref{proposition:delta} to get the~required assertion.
 \end{proof}

In the~remaining part of the~section, we will prove Proposition~\ref{proposition:delta} by induction on $n$

\subsection{Base of induction}
\label{subsection:threefolds}

Let $V$ be a smooth projective surface, let $L$ be an ample Cartier divisor on $V$,
let $\mu$ be the~smallest rational number such that $\mu L$ is very ample, let
$$
Y=\mathbb{P}\big(\mathcal{O}_{V}\oplus \mathcal{O}_{V}(L)\big),
$$
and let $\pi\colon Y\to V$ be the~natural projection. Set $H=\pi^*(L)$.
Let $S^-$ and $S^+$ be disjoint sections of the~projection $\pi$ such that $S^+\sim S^-+H$,
and let $B$ be an irreducible normal surface in $|2S^+|$
such that $\pi$ induces a double cover $B\to V$ which is ramified in a reduced curve $R\in |2L|$.
Fix $a\in\mathbb{Q}$ such that $a\geqslant 1$. Let
$$
D(a)=S^-+aH.
$$
Then $D(a)$ is nef and big, and $D(a)$ is ample for $a>1$. Set $\Delta=\frac{1}{2}B$ and $d=L^2$.

\begin{remark}
\label{remark:d-mu}
Since $\mu L$ is very ample and $L$ is Cartier, we have $d\mu=(\mu L)\cdot L\in\mathbb{Z}_{>0}$ and
$$
d\mu^2=(\mu L)^2\in\mathbb{Z}_{>0}.
$$
Moreover, if $d\mu=1$, then $\mu=1$, $d=L^2=1$, $V=\mathbb P^2$ and $L=\mathcal{O}_{\mathbb P^2}(1)$.
\end{remark}

Suppose, in addition, that $d\mu\geqslant 2$. Set
$$
k_3(a,d,\mu)=\frac{8d \mu a^3+6(1-2d\mu)a^2+8(d\mu-1)a-2d \mu+3}{8(3a^2-3a+1)}.
$$
Let $P$ be a point in $Y$ such that $P\not\in S^-$ and $P\not\in\mathrm{Sing}(B)$.

\begin{proposition}
\label{proposition:induction-step}
One has $\delta_P(Y,\Delta;D(a)) \geqslant\frac{1}{k_3(a,d,\mu)}$.
\end{proposition}

In the~remaining part of this subsection, we will prove this result.
We will only consider the~case $P\in B$, because the~case $P\not\in B$ is much simpler.

Let $V_1$ be a general curve in $|\mu L|$ that contains the~point $\pi(P)$,
and let $Y_1=\pi^*(V_1)$.
Then $V_1$ is a smooth curve, and $Y_1$ is a smooth surface. For simplicity, we set $D=D(a)$.
Take $u\in \mathbb R_{\geqslant 0}$.
Then
$$
D-uY_1 \sim_{\mathbb{R}} S^-+(a-\mu u)H,
$$
so that $D-uY_1$ is pseudo-effective $\iff$ $u\leqslant \frac{a}{\mu}$. We have
$$
(D-uY_1)\big\vert_{S^-}\sim_{\mathbb{R}} (S^-+(a-\mu u)H)\big\vert_{S^-}\sim_{\mathbb{R}} (a-1-\mu u)L,
$$
where we use isomorphism $S^-\cong V$ induced by $\pi$.
Hence, the~divisor $D-uY_1$ is nef if and only if $u \leqslant \frac{a-1}{\mu}$.
Moreover, the~Zariskis decomposition of $D-uY_1$ is
\begin{align*}
P(u)\equiv
\begin{cases}
S^-+(a-\mu u)H & \text{if}\ u\in[0,\frac{a-1}{\mu}],\\
(a-\mu u)(S^-+H)=(a-\mu u)S^+ & \text{if}\ u\in[\frac{a-1}{\mu},\frac{a}{\mu}],
\end{cases}
\end{align*}
and
\begin{align*}
N(u)=
\begin{cases}
0 & \text{if}\ u\in [0,\frac{a-1}{\mu}],\\
(\mu u +1-a)S^- & \text{if}\ u\in [\frac{a-1}{\mu},\frac{a}{\mu}],
\end{cases}
\end{align*}
where $P(u)$ is the~positive part, and $N(u)$ is the~negative part.

Note that $H^3=0$, $H^2 \cdot S^-=d$, $H\cdot (S^{-})^2=-d$, $(S^-)^3=d$. Then
\begin{align*}
S_{D}(Y_1)& =\frac{1}{D^3}\int_0^{\frac{a}{\mu}} \mathrm{vol}(D-uY_1)du\\
& =\frac{1}{(S^-+aH)^3}\Bigg(\int_0^{\frac{a-1}{\mu}} (S^-+(a-\mu u)H)^3du+ \int_{\frac{a-1}{\mu}}^{\frac{a}{\mu}} ((a-\mu u)(S^-+H))^3du \Bigg) \\
&=\frac{(2a-1)(2a^2-2a+1)}{4\mu(3a^2-3a+1)}.
\end{align*}

Let $f$ be the~fiber of the~$\mathbb{P}^1$-bundle $\pi$ that contains $P$.
Then there are two cases to consider: either $B$ intersects $f$ transversely at $P$ or tangentially.
For each case, we consider an appropriate plt blow up $h \colon \widetilde{Y}_1\rightarrow Y_1$ at the~point $P$
with smooth exceptional curve $E$.
We let $\Delta_1=\Delta\vert_{Y_1}$, and we denote by $\widetilde{\Delta}_1$ the~proper transform on $\widetilde{Y}_1$ of the~divisor $\Delta_1$.
Then it follows from \cite{AbbanZhuang,Book,Fujita2021} that
$$
\delta_P(Y,\Delta)\geqslant \min\bigg\{\frac{1}{S_{D}(Y_1)}, \frac{A_{Y_1,\Delta_1}(E)}{S(V^{Y_1}_{\bullet,\bullet};E)},
\inf_{\substack{Q \in E}} \frac{A_{E,\Delta_E}(Q)}{S(V^{\widetilde{Y}_1,E}_{\bullet,\bullet,\bullet};Q)}\bigg\}.
$$
where $S(V^{Y_1}_{\bullet,\bullet};E)$ and $S(V^{\widetilde{Y}_1,E}_{\bullet,\bullet,\bullet};Q)$ are defined in \cite[Section~1.7]{Book},
and $\Delta_E$ is the~different computed via the~adjunction formula
$$
K_E+\Delta_E=\big(K_{\widetilde{Y}_1}+\widetilde{\Delta}_1\big)\vert_{E}.
$$
For instance, if $h$ is the~ordinary blow up at the~point $P$, then $\Delta_E=\widetilde{\Delta}_1\vert_{E}$.
For simplicity, we rewrite the~last inequality as
\begin{equation}
\label{equation:base-induction-delta}
\frac{1}{\delta_P(Y,\Delta)}\leqslant \max\bigg\{S_{D}(Y_1),
\frac{S(V^{Y_1}_{\bullet,\bullet};E)}{A_{Y_1,\Delta_1}(E)}, \sup_{\substack{Q\in E}} \frac{S(V^{\widetilde{Y}_1,E}_{\bullet,\bullet,\bullet};Q)}{A_{E,\Delta_E}(Q)} \bigg\}.
\end{equation}
Thus, to prove Proposition~\ref{proposition:induction-step},
it is enough to bound each term in  \eqref{equation:base-induction-delta} by $k_3(a,d,\mu)$.

We set $S_1^-=S^-\vert_{Y_1}$, $H_1:=H\vert_{Y_1}$,  $B_1:=B\vert_{Y_1}$, $D_1=P(u)\vert_{Y_1}$. Note that $H_1 \equiv d \mu f$ and
\begin{align*}
D_1\equiv
\begin{cases}
S_1^-+(a-\mu u)d\mu f & \text{if}\ u\in[0,\frac{a-1}{\mu}],\\
(a-\mu u)(S_1^-+d \mu f) & \text{if}\ u\in[\frac{a-1}{\mu},\frac{a}{\mu}].
\end{cases}
\end{align*}
We denote by  $\widetilde{S}_1^-$, $\widetilde{B}_1$, $\widetilde{f}$ the~proper transforms on $\widetilde{Y}_1$ of the~curves $S_1^-$, $B_1$, $f$, respectively.

\begin{lemma}
\label{lem:pinBdim3-transversal}
Suppose $B$ intersects $f$ transversally. Then $\delta_P(Y,\Delta;D(a)) \geqslant\frac{1}{k_3(a,d,\mu)}$.
\end{lemma}

\begin{proof}
Let $h\colon \widetilde{Y}_1\rightarrow Y_1$ be the~ordinary blow up at $P$.
Recall that $E$ is the~$h$-exceptional curve.
We have $\widetilde{S}_1^-\sim h^* (S_1^-)$ and $\widetilde{f}\sim h^*(f)-E$.
Take $v\in\mathbb{R}_{\geqslant 0}$. Then
\begin{align*}
h^*(D_1)-vE\equiv
\begin{cases}
\widetilde{S}_1^-+(a-\mu u)d\mu \widetilde{f}+ ((a-\mu u)d\mu-v)E & \text{if}\ u\in[0,\frac{a-1}{\mu}],\\
(a-\mu u)(\widetilde{S}_1^-+d \mu \widetilde{f})+((a-\mu u)d\mu-v)E & \text{if}\ u\in[\frac{a-1}{\mu},\frac{a}{\mu}].
\end{cases}
\end{align*}
We have the~following intersection numbers:
\begin{center}
\renewcommand\arraystretch{1.6}
\begin{tabular}{|c||c|c|c|}
 \hline
 $\bullet$& $\widetilde{S}_1^-$ & $\widetilde{f}$ & $E$ \\
 \hline
\hline
 $\widetilde{S}_1^-$ & $-d\mu$ & $1$ & $0$ \\
 \hline
 $\widetilde{f}$ & $1$ & $-1$ & $1$ \\
 \hline
 $E$ & $0$ & $1$ & $-1$ \\
 \hline
\end{tabular}
\end{center}
This shows that $h^*(D_1)-vE$ is pseudo-effective $\iff$ $v\leqslant (a-\mu u)d\mu$.

If  $u\in [0,\frac{a-1}{\mu}]$, the~positive part of the~Zariski decomposition of $h^*(D_1)-vE$ is
\begin{align*}\hspace*{-1cm}
\widetilde{P}(u,v)\equiv
\begin{cases}
\widetilde{S}_1^-+(a-\mu u)d\mu \widetilde{f}+ ((a-\mu u)d\mu-v)E & \text{if }  v\in [0,1]\\
\widetilde{S}_1^- +((a-\mu u)d\mu+1-v)\widetilde{f}+  ((a-\mu u)d\mu-v)E     & \text{if }   v\in [1,1-d\mu^2u+ad\mu-d\mu] \\
\frac{-d\mu^2u+ad\mu-v}{d\mu-1}(\widetilde{S}_1^-+d\mu \widetilde{f}+(d\mu-1)E)    & \text{if } v\in [1-d\mu^2u+ad\mu-d\mu,(a-\mu u)d\mu],
\end{cases}
\end{align*}
and the~negative part is
\begin{align*}
\widetilde{N}(u,v)=
\begin{cases}
0 & \text{if }  v\in [0,1]\\
(v-1)\widetilde{f}      & \text{if }    v\in [1,1-d\mu^2u+ad\mu-d\mu] \\
\frac{d\mu(\mu u-a+v)}{d\mu -1}\widetilde{f}+\frac{d\mu^2u-ad\mu+d\mu+v-1}{d\mu -1}\widetilde{S}_1^-    & \text{if } v\in [1-d\mu^2u+ad\mu-d\mu,(a-\mu u)d\mu].
\end{cases}
\end{align*}
Similarly, if $u\in [\frac{a-1}{\mu},\frac{a}{\mu}]$, the~positive part of the~Zariski decomposition of $h^*(D_1)-vE$ is
\begin{align*}
\widetilde{P}(u,v)\equiv
\begin{cases}
(a-\mu u)(\widetilde{S}_1^-+d \mu \widetilde{f})+((a-\mu u)d\mu-v)E  & \text{if } v\in [0,a-\mu u] \\
 \frac{1}{d\mu -1}(-d\mu^2u+ad\mu-v)(\widetilde{S}_1^-+d\mu\widetilde{f}+(d\mu-1)E)      & \text{if } v\in [a-\mu u,(a-\mu u)d\mu].
\end{cases}
\end{align*}
and the~negative part is
\begin{align*}
\widetilde{N}(u,v)=
\begin{cases}
0   & \text{if } v\in [0,a-\mu u] \\
\frac{1}{d\mu -1}(d\mu(\mu u-a+v)\widetilde{f}+(\mu u-a+v)\widetilde{S}_1^-)   & \text{if } v\in [a-\mu u,(a-\mu u)d\mu].
\end{cases}
\end{align*}

Now, using results from \cite[Section~1.7]{Book}, we compute
\begin{align*}
S(W_{\bullet,\bullet}^{\widetilde{Y_1}};E) =\frac{3}{D^3}\int\limits_0^{\frac{a}{\mu}} \int\limits_0^{(a-\mu u)d\mu} \mathrm{vol}(D_1-vF)dvdu & =\frac{3}{(S^-+aH)^3}\int\limits_0^{\frac{a}{\mu}} \int\limits_0^{(a-\mu u)d\mu}\widetilde{P}(u,v)^2dvdu \\
&=\frac{4a^3d\mu+6(1-d\mu)a^2+4(d\mu-2)a-d\mu+3}{4(3a^2-3a+1)}.
\end{align*}
Moreover, we have $A_{Y_1,\Delta_1}(E)=2-\frac{1}{2}=\frac{3}{2}$, so that
$$
\frac{S(W_{\bullet,\bullet}^{Y_1};E)}{A_{Y_1,\Delta_1}(E)}=\frac{4a^3d\mu+6(1-d\mu)a^2+4(d\mu-2)a-d\mu+3}{6(3a^2-3a+1)}.
$$

Let $Q$ be a point in $E$. Then, using results from \cite[Section~1.7]{Book}, we compute
\begin{align*}
S(W_{\bullet,\bullet,\bullet}^{\widetilde{Y_1},E};Q)  & =\frac{3}{(S^-+aH)^3}\int\limits_0^{\frac{a}{\mu}} \int\limits_0^{(a-\mu u)d\mu}\big (\widetilde{P}(u,v) \cdot E \big)^2 dvdu+F_q(W_{\bullet,\bullet,\bullet}^{\widetilde{Y_1},E}) \\
&=\frac{6a^2-8a+3}{4(3a^2-3a+1)}+F_Q\big(W_{\bullet,\bullet,\bullet}^{\widetilde{Y_1},E}\big),
\end{align*}
where
$$
F_Q\big(W_{\bullet,\bullet,\bullet}^{\widetilde{Y_1},E}\big)=\frac{6}{(S^-+aH)^3}\int\limits_0^{\frac{a}{\mu}}\int\limits_0^{(a-\mu u)d\mu}\big(\widetilde{P}(u,v) \cdot E\big)\cdot \mathrm{ord}_Q\big(\widetilde{N}(u,v)\vert_E\big)dvdu,
$$
because $P\not\in\mathrm{Supp}(N(u))$ for  $u\in[0,\frac{a}{\mu}]$.
Notice that $F_Q(W_{\bullet,\bullet,\bullet}^{\widetilde{Y_1},E}) \not = 0$ only when $Q\in \widetilde{f}$.
Thus, there are three cases to consider.
\begin{itemize}
\item $Q=E\cap \widetilde{f}$. Then
$$
F_Q\big(W_{\bullet,\bullet,\bullet}^{\widetilde{Y_1},E}\big)=\frac{3-8a+6a^2+d\mu-4ad\mu+6a^2d\mu-4a^3d\mu}{4(3a^2-3a+1)}
$$
and $A_{E,\Delta_E}(Q)=1$ since $Q\not\in \widetilde{B}_1$. Hence, we have
$$
\frac{S(W_{\bullet,\bullet,\bullet}^{\widetilde{Y_1},E};Q)}{A_{E,\Delta_E}(Q)}=\frac{d\mu(2a-1)(2a^2-2a+1)}{4(3a^2-3a+1)}.
$$

\item $Q\in E \cap \widetilde{B}_1$. Then $A_{E,\Delta_E}(Q)=\frac{1}{2}$, so that
$$
\frac{S(W_{\bullet,\bullet,\bullet}^{\widetilde{Y_1},E};Q)}{A_{E,\Delta_E}(Q)}=\frac{6a^2-8a+3}{2(3a^2-3a+1)}.
$$

\item $Q\in E$ away from $\widetilde{f}$ and $\widetilde{B}_1$. Then $A_{E,\Delta_E}(Q)=1$, so that
$$
\frac{S(W_{\bullet,\bullet,\bullet}^{\widetilde{Y_1},E};Q)}{A_{E,\Delta_E}(Q)}=\frac{6a^2-8a+3}{4(3a^2-3a+1)}.
$$
\end{itemize}
The third case is smaller than the~previous one (exactly half) so we do not consider it.
So, using \eqref{equation:base-induction-delta}, we obtain the~inequality
\begin{multline}
\label{ineq:trans}
\frac{1}{\delta_P(Y,\Delta)}\leqslant\max\bigg\{\frac{(2a-1)(2a^2-2a+1)}{4\mu(3a^2-3a+1)},\\
\frac{4a^3d\mu+6(1-d\mu)a^2+4(d\mu-2)a-d\mu+3}{6(3a^2-3a+1)}, \\
\frac{d\mu(2a-1)(2a^2-2a+1)}{4(3a^2-3a+1)}, \frac{6a^2-8a+3}{2(3a^2-3a+1)}\bigg\}.
\end{multline}

Recall from Remark~\ref{remark:d-mu} that $d\mu^2\geqslant 1$.
This allows us to conclude
$$
\frac{d\mu(2a-1)(2a^2-2a+1)}{4(3a^2-3a+1)} \geqslant \frac{(2a-1)(2a^2-2a+1)}{4\mu(3a^2-3a+1)}
$$
so we can discard the~first term in \eqref{ineq:trans}.
Moreover, since $d\mu\geqslant 2$, we have
\begin{align*}
\frac{4a^3d\mu+6(1-d\mu)a^2+4(d\mu-2)a-d\mu+3}{6(3a^2-3a+1)}&\leqslant k_3(a,d,\mu),\\
\frac{d\mu(2a-1)(2a^2-2a+1)}{4(3a^2-3a+1)}&\leqslant k_3(a,d,\mu),\\
\frac{6a^2-8a+3}{2(3a^2-3a+1)}&\leqslant k_3(a,d,\mu),
\end{align*}
which gives $\delta_P(Y,\Delta;D(a)) \geqslant\frac{1}{k_3(a,d,\mu)}$.
\end{proof}

Now, we deal with the~case when $f$ is tangent to $B$ at the~point $P$.

\begin{lemma}
\label{lem:pinBdim3-tangential}
Suppose $B$ and $f$ are tangent at $P$. Then $\delta_P(Y,\Delta;D(a)) \geqslant\frac{1}{k_3(a,d,\mu)}$.
\end{lemma}

\begin{proof}
Now, we let $h\colon \widetilde{Y}_1\rightarrow Y_1$ be the~$(1,2)$-weighted blowup of the~point $P$
such that the~curves $\widetilde{B}_1$ and $\widetilde{f}$ are disjoint.
Then $\widetilde{f}=h^*(f)-2E$. Take $v\in\mathbb{R}_{\geqslant 0}$. Then
\begin{align*}
h^*(D_1)-vE\equiv
\begin{cases}
\widetilde{S}_1^-+(a-\mu u)d\mu \widetilde{f}+ (2(a-\mu u)d\mu-v)E & \text{if}\ u\in[0,\frac{a-1}{\mu}],\\
(a-\mu u)(\widetilde{S}_1^-+d \mu \widetilde{f})+(2(a-\mu u)d\mu-v)E & \text{if}\ u\in[\frac{a-1}{\mu},\frac{a}{\mu}].
\end{cases}
\end{align*}
Moreover, we have the~following intersection numbers:
 \begin{center}
\renewcommand\arraystretch{1.6}
\begin{tabular}{|c||c|c|c|}
 \hline
$\bullet$ & $\widetilde{S}_1^-$ & $\widetilde{f}$ & $E$ \\
 \hline
 \hline
 $\widetilde{S}_1^-$ & $-d\mu$ & $1$ & $0$ \\
 \hline
 $\widetilde{f}$ & $1$ & $-2$ & $1$ \\
 \hline
 $E$ & $0$ & $1$ & $-\frac{1}{2}$ \\
 \hline
\end{tabular}
\end{center}
Thus, the~divisor $h^*(D_1)-vE$ is pseudo-effective $\iff$ $v\leqslant 2(a-\mu u)d\mu$.

If $u\in [0,\frac{a-1}{\mu}]$, the~positive part of the~Zariski decomposition of $h^*(D_1)-vE$ is
\begin{align*}\hspace*{-1.2cm}
\widetilde{P}(u,v)\equiv
\begin{cases}
\widetilde{S}_1^-+(a-\mu u)d\mu \widetilde{f}+ (2(a-\mu u)d\mu-v)E & \text{if }    v\in [0,1]\\
\widetilde{S}_1^- +((a-\mu u)d\mu+\frac{1-v}{2})\widetilde{f}+  (2(a-\mu u)d\mu-v)E     & \text{if }    v\in [1,-2d\mu^2u+2ad\mu - 2d\mu+1] \\
\frac{-2d\mu^2u+2ad\mu-v}{2d\mu-1}(\widetilde{S}_1^-+d\mu \widetilde{f}+(2d\mu-1)E)    & \text{if }    v\in [-2d\mu^2u+2ad\mu - 2d\mu+1,2(a-\mu u)d\mu],
\end{cases}
\end{align*}
and the~negative part is
\begin{align*}
\widetilde{N}(u,v)=
\begin{cases}
0 & \text{if }    v\in [0,1]\\
\frac{v-1}{2}\widetilde{f}      & \text{if }     v\in [1,-2d\mu^2u+2ad\mu - 2d\mu+1] \\
\frac{d\mu(\mu u-a+v)}{2d\mu -1}\widetilde{f}+\frac{2d\mu^2u-2ad\mu+2d\mu+v-1}{2d\mu -1}\widetilde{S}_1^-    & \text{if }      v\in [-2d\mu^2u+2ad\mu - 2d\mu+1,2(a-\mu u)d\mu].
\end{cases}
\end{align*}
Similarly, if $u\in [\frac{a-1}{\mu},\frac{a}{\mu}]$, the~positive part of the~Zariski decomposition of $h^*(D_1)-vE$ is
\begin{align*}
\widetilde{P}(u,v)\equiv
\begin{cases}
(a-\mu u)(\widetilde{S}_1^-+d \mu \widetilde{f})+(2(a-\mu u)d\mu-v)E  & \text{if } v\in [0,a-\mu u] \\
 \frac{-2d\mu^2u+2ad\mu-v}{2d\mu -1}(\widetilde{S}_1^-+d\mu\widetilde{f}+(2d\mu-1)E)      & \text{if }  v\in [a-\mu u,2(a-\mu u)d\mu],
\end{cases}
\end{align*}
and the~negative part is
\begin{align*}
\widetilde{N}(u,v)=
\begin{cases}
0   & \text{if }   v\in [0,a-\mu u] \\
\frac{d\mu(\mu u-a+v)}{2d\mu -1}\widetilde{f}+\frac{\mu u-a+v}{2d\mu -1}\widetilde{S}_1^-   & \text{if }  v\in [a-\mu u,2(a-\mu u)d\mu].
\end{cases}
\end{align*}

Now, using results from \cite[Section~1.7]{Book}, we compute
\begin{align*}
S(W_{\bullet,\bullet}^{Y_1};E) =\frac{3}{D^3}\int\limits_0^{\frac{a}{\mu}} \int\limits_0^{2(a-\mu u)d\mu} \mathrm{vol}(D_1-vF)dvdu & =\frac{3}{(S^-+aH)^3}\int\limits_0^{\frac{a}{\mu}} \int\limits_0^{2(a-\mu u)d\mu}\widetilde{P}(u,v)dvdu \\
&=\frac{1}{4}\cdot \frac{8a^3d\mu+6(1-2d\mu)a^2+8(d\mu-1)a-2d\mu+3}{3a^2-3a+1}
\end{align*}
Moreover, since $A_{Y_1,\Delta_1}(E)=2$, we have
$$
\frac{S(W_{\bullet,\bullet}^{Y_1};E)}{A_{Y_1,\Delta_1}(E)}=\frac{1}{8}\cdot \frac{8a^3d\mu+6(1-2d\mu)a^2+8(d\mu-1)a-2d\mu+3}{3a^2-3a+1}.
$$

Let $Q$ be a point in $E$. Using results from \cite[Section~1.7]{Book}, we get
\begin{align*}
S(W_{\bullet,\bullet,\bullet}^{\widetilde{Y_1},E};Q)  & =\frac{3}{(S^-+aH)^3}\int\limits_0^{\frac{a}{\mu}} \int\limits_0^{2(a-\mu u)d\mu}\big (\widetilde{P}(u,v) \cdot E \big)^2 dvdu+F_Q\big(W_{\bullet,\bullet,\bullet}^{\widetilde{Y_1},E}\big) \\
&=\frac{1}{8}\cdot \frac{6a^2-8a+3}{3a^2-3a+1}+F_Q\big(W_{\bullet,\bullet,\bullet}^{\widetilde{Y_1},E}\big)
\end{align*}
where
$$
F_Q\big(W_{\bullet,\bullet,\bullet}^{\widetilde{Y_1},E}\big)=\frac{6}{(S^-+aH)^3}\int\limits_0^{\frac{a}{\mu}}\int\limits_0^{2(a-\mu u)d\mu}(\widetilde{P}(u,v) \cdot E)\cdot \mathrm{ord}_Q(\widetilde{N}(u,v)\vert_E)dvdu.
$$
There are three cases to consider.
\begin{itemize}
\item $Q=E\cap \widetilde{f}$. Then
$$
F_Q(W_{\bullet,\bullet,\bullet}^{\widetilde{Y_1},E})=\frac{1}{8}\frac{8a^3d \mu-6(2d\mu-1)a^2+8(d\mu+1)a-2d\mu-3}{3a^2-3a+1}
$$
and $A_{E,\Delta_E}(Q)=1$ since $Q\not \in \widetilde{B}_1$. Hence, we have
$$
\frac{S(W_{\bullet,\bullet,\bullet}^{\widetilde{Y_1},E};Q) }{A_{E,\Delta_E}(Q)}=\frac{d\mu}{4}\cdot \frac{(2a-1)(2a^2-2a+1)}{3a^2-3a+1}.
$$

\item $Q\in E \cap \widetilde{B}$. Then $A_{E,\Delta_E}(Q)=\frac{1}{2}$, so that
$$
\frac{S(W_{\bullet,\bullet,\bullet}^{\widetilde{Y_1},E};Q) }{A_{E,\Delta_E}(Q)}=\frac{1}{4}\cdot \frac{6a^2-8a+3}{3a^2-3a+1}.
$$

\item $Q\in E$ is the~$\mathbb{A}_1$ singularity. Then $A_{E,\Delta_E}(Q)=\frac{1}{2}$ and so
$$
\frac{S(W_{\bullet,\bullet,\bullet}^{\widetilde{Y_1},E};Q) }{A_{E,\Delta_E}(Q)}=\frac{1}{4}\cdot \frac{6a^2-8a+3}{3a^2-3a+1}.
$$
\end{itemize}
We have the~inequality:
\begin{multline} \label{ineq:tangent}
   \frac{1}{\delta_P(Y,\Delta)}\leqslant \max\bigg\{\frac{(2a-1)(2a^2-2a+1)}{4\mu(3a^2-3a+1)},\\
   \frac{1}{8}\cdot \frac{8a^3d\mu+6(1-2d\mu)a^2+8(d\mu-1)a-2d\mu+3}{3a^2-3a+1}, \\
    \frac{d\mu}{4}\cdot \frac{(2a-1)(2a^2-2a+1)}{3a^2-3a+1},\frac{1}{4}\cdot \frac{6a^2-8a+3}{3a^2-3a+1}\bigg\}.
\end{multline}
Now, arguing as in the~end of the~proof of Lemma~\ref{lem:pinBdim3-transversal}, we find
$$
\frac{1}{\delta_P(Y,\Delta)} \leqslant \frac{1}{8}\cdot \frac{8a^3d\mu+6(1-2d\mu)a^2+8(d\mu-1)a-2d\mu+3}{3a^2-3a+1},
$$
and the~result follows.
\end{proof}

Proposition~\ref{proposition:induction-step} is proved.

\subsection{The induction.}
\label{subsection:induction}

Let us use all assumption and notations introduced in Section~\ref{section:delta}.
Recall that $\mu$ is the~smallest rational number for which $\mu L$ is a very ample Cartier divisor on the~variety $V$ and $d=L^{n-1}$.
Then
$$
\mu^{n-1}d=(\mu L)^{n-1}\geqslant 1.
$$
Let us prove Proposition~\ref{proposition:delta} by induction on $\mathrm{dim}(Y)=n\geqslant 3$ --- the~base of induction (the case when $n=3$) is done in Section~\ref{section:delta}.

Therefore, we suppose that Proposition~\ref{proposition:delta} holds for varieties of dimension $n-1\geqslant 3$.
Let $P$ be a point in $Y$ such that $P\not\in S^-$. We must prove that
$$
\delta_P(Y,\Delta;D(a))\geqslant \frac{1}{k_n(a,d,\mu)},
$$
where $k_n(a,d,\mu)$ is presented in Theorem~\ref{theorem:delta}.
We will only consider the~case~when~$P\in B$, since the~case~$P\not\in B$ is simpler and similar.
Thus, we suppose that $P\in B$.

Let $V_{n-1}$ be a general divisor in $|\mu L|$ that contains the~point $\pi(P)$.
Set $Y_{n-1}=\pi^*(V_{n-1})$. For simplicity, set $D=D(a)$.
First, let us compute $S_{D}(Y_{n-1})$. Take $u\in\mathbb{R}_{\geqslant 0}$. Then
$$
D(a)-uY_{n-1}\sim_{\mathbb{R}} S^-+(a-\mu u)H,
$$
so $D(a)-uY_{n-1}$ is pseudo-effective $\iff$ $u\leqslant\frac{a}{\mu}$.
For $u\in[0,\frac{a}{\mu}]$, let $P(u)$ be the~positive part of the~Zariski decomposition of $D(a)-uY_{n-1}$,
and let $N(u)$ be its~negative part. Then
\begin{align*}
P(u)\equiv
\begin{cases}
S^-+(a-\mu u)H = D(a-\mu u) & \text{if}\ u\in[0,\frac{a-1}{\mu}],\\
(a-\mu u)(S^-+H)=(a-\mu u)D(1) & \text{if}\ u\in[\frac{a-1}{\mu},\frac{a}{\mu}],
\end{cases}
\end{align*}
and
\begin{align*}
N(u)=
\begin{cases}
0 & \text{if}\ u\in [0,\frac{a-1}{\mu}],\\
(\mu u +1-a)S^- & \text{if}\ u\in [\frac{a-1}{\mu},\frac{a}{\mu}].
\end{cases}
\end{align*}
Recall that $S^-\cap S^+=\varnothing$. Note that $(S^-)^n=(-1)^{n+1}d$ and $(S^+)^n=d$. Hence, we have
$$
D(a)^n=(S^-+aH)^n = ((1-a)S^-+aS^+)^n=d(a^n-(a-1)^n).
$$
Now, we compute
\begin{align*}
S_{D}(Y_{n-1})&=\frac{1}{D(a)^n}\int_0^{\infty} \mathrm{vol}(D(a)-uY_{n-1})du\\ &=
\frac{1}{D(a)^n}\int_0^{\frac{a-1}{\mu}}(S^-+(a-\mu u)H)^ndu
+\frac{1}{D(a)^n}\int_{\frac{a-1}{\mu}}^{\frac{a}{\mu}} ((a-\mu u)(S^-+H))^ndu\\&=
\frac{1}{D(a)^n}\int_0^{\frac{a-1}{\mu}} d((-1)^{n+1}(1-a+\mu u)^n+(a-\mu u)^n) du+\frac{1}{D(a)^n}\int_{\frac{a-1}{\mu}}^{\frac{a}{\mu}} d(a-\mu u)^n du\\&=\frac{a^{n+1}-(a-1)^{n+1}}{\mu (n+1)(a^n-(a-1)^n)}.
\end{align*}
Set
$$
\mathrm{Res}_n(a)=\frac{a^{n+1}-(a+n)(a-1)^{n}}{2(n+1)(a^n-(a-1)^n)}.
$$

\begin{lemma}
\label{res}
One has $k_n(a,d,\mu)=S_{D(a)}(Y_{n-1})d\mu^{n-1}+\mathrm{Res}_n(a)$ and  $\mathrm{Res}_n(a)>0$.
\end{lemma}

\begin{proof}
The equality follows from the~formulas for $k_n(a,d,\mu)$ and $S_{D(a)}(Y_{n-1})$.

Let us show that $\mathrm{Res}_n(a)>0$. We may assume that $a>1$. The denominator is clearly positive. Hence, we only need to verify that $a^{n+1}-(a+n)(a-1)^{n}>0$.
But
\[
\bigg(\frac{a}{a-1} \bigg)^n = \bigg(1+\frac{1}{a-1}
\bigg)^n=\sum_{i=0}^n \binom{n}{i} \bigg(\frac{1}{a-1}
\bigg)^i> 1+\frac{n}{a-1}> 1+\frac{n}{a}= \frac{a+n}{a},
\]
which gives $a^{n+1}-(a+n)(a-1)^{n}>0$. This shows that $\mathrm{Res}_n(a)>0$.
\end{proof}

Set $\Delta_{n-1}=\Delta\vert_{Y_{n-1}}$. Then $S_{D}(Y_{n-1})\leqslant k_{n}(a,d,\mu)$ by  Lemma~\ref{res}, since  $d\mu^{n-1}\geqslant 1$.
Therefore, using  \cite{AbbanZhuang}, we see that $\delta_P(Y,\Delta;D)\geqslant\frac{1}{k_{n}(a,d,\mu)}$
provided that
\begin{equation}
\label{equation:induction}
S(V^{Y_{n-1}}_{\bullet,\bullet};E)\leqslant k_n(a,d,\mu)A_{Y_{n-1},\Delta_{n-1}}(E),
\end{equation}
for every prime divisor $E$ over the~variety $Y_{n-1}$ such that its center on $Y_{n-1}$ contains $P$,
where $A_{Y_{n-1},\Delta_{n-1}}(E)$ is the~log discrepancy, and  $S(V^{Y_{n-1}}_{\bullet,\bullet};E)$ is defined in \cite[Section~1.7]{Book}.

Suppose that $n\geqslant 4$.
Let us prove \eqref{equation:induction} using Proposition~\ref{proposition:delta} applied to $(Y_{n-1},\Delta_{n-1})$.

Let $E$ be a prime divisor over $Y_{n-1}$ whose center in $Y_{n-1}$ contains $P$.
Since $P\not\in S^-$, it follows from \cite[Corollary~1.108]{Book} that
\begin{multline*}
S(V^{Y_{n-1}}_{\bullet,\bullet};E)=\frac{n}{D^n}\int\limits_0^{\frac{a}{\mu}}\Bigg( \int\limits_0^{\infty}\mathrm{vol}(P(u)\vert_{Y_{n-1}}-vE)dv \Bigg )du=\\
=\frac{n}{D^n}\int\limits_0^{\frac{a-1}{\mu}}\int\limits_0^{\infty}\mathrm{vol}(S^-+(a-\mu u)H-vE)dvdu+\frac{n}{D^n}\int\limits_{\frac{a-1}{\mu}}^{\frac{a}{\mu}}\int\limits_0^{\infty}\mathrm{vol}((a-\mu u)(S^-+H)-vE)dvdu=\\
=\frac{n}{D^n}\int\limits_0^{\frac{a-1}{\mu}}\int\limits_0^{\infty}\mathrm{vol}(S^-+(a-\mu u)H-vE)dvdu+\frac{n}{D^n}\int\limits_{\frac{a-1}{\mu}}^{\frac{a}{\mu}}(a-\mu u)^n\int\limits_0^{\infty}\mathrm{vol} (S^-+H-vE)dvdu.
\end{multline*}
Now, applying Proposition~\ref{proposition:delta} (induction step), we get
$$
\int\limits_0^{\infty}\mathrm{vol}(S^-+(a-\mu u)H-vE)dv\leqslant k_{n-1}(a-\mu u,d\mu,\mu)(S^-+(a-\mu u)H)^{n-1}A_{Y_{n-1},\Delta_{n-1}}(E)
$$
and
$$
\int\limits_0^{\infty}\mathrm{vol} (S^-+H-vE)dv\leqslant k_{n-1}(1,d\mu,\mu)(S^-+H)^{n-1}A_{Y_{n-1},\Delta_{n-1}}(E).
$$
Hence, combining, we obtain
\begin{multline*}
S(V^{Y_{n-1}}_{\bullet,\bullet};E)\leqslant \frac{n}{D^n}\int\limits_0^{\frac{a-1}{\mu}}k_{n-1}(a-\mu u,d\mu,\mu)(S^-+(a-\mu u)H)^{n-1}A_{Y_{n-1},\Delta_{n-1}}(E)du+\\
+\frac{n}{D^n}\int\limits_{\frac{a-1}{\mu}}^{\frac{a}{\mu}}(a-\mu u)^nk_{n-1}(1,d\mu,\mu)(S^-+H)^{n-1}A_{Y_{n-1},\Delta_{n-1}}(E)du=\\
=A_{Y_{n-1},\Delta_{n-1}}(E)\frac{n}{D^n}\int\limits_0^{\frac{a-1}{\mu}}k_{n-1}(a-\mu u,d\mu,\mu)(S^-+(a-\mu u)H)^{n-1}du+\\
+A_{Y_{n-1},\Delta_{n-1}}(E)\frac{n}{D^n}\int\limits_{\frac{a-1}{\mu}}^{\frac{a}{\mu}}(a-\mu u)^nk_{n-1}(1,d\mu,\mu) (S^-+H)^{n-1}du.
\end{multline*}
Let us compute these two integrals separately. We have
\begin{multline*}
A_1:=\int\limits_0^{\frac{a-1}{\mu}}k_{n-1}(a-\mu u,d\mu,\mu)(S^-+(a-\mu u)H)^{n-1}du=\\
=d\mu^{n-1}\int\limits_0^{\frac{a-1}{\mu}}\frac{d\mu ((-1)^{n-1}(1-a+\mu u)^{n}+(a-\mu u)^{n})}{\mu n}  du+\quad\quad\quad\quad\quad\\
\quad\quad\quad\quad+\int\limits_0^{\frac{a-1}{\mu}}\frac{d\mu ((a-\mu u)^n-(a-\mu u+n-1)(a-\mu u -1)^{n-1})}{2n}  du=\\
=\frac{d^2\mu^{n-1}}{\mu n(n+1)}(a^{n+1}-(a-1)^{n+1}-1)+ \frac{d}{2n(n+1)}(a^{n+1}-(a+n)(a-1)^{n}-1)
\end{multline*}
and
$$
A_2:=\int\limits_{\frac{a-1}{\mu}}^{\frac{a}{\mu}}(a_n-\mu u)^nk_{n-1}(1,d\mu,\mu) (S^-+H)^{n-1} du = \frac{d(2 d\mu^{n-2}+1)}{2n(n+1)}=\frac{d^2\mu^{n-1}}{\mu n(n+1)}+\frac{d}{2n(n+1)}.
$$
Adding these two integrals we get
\begin{align*}
    \frac{n}{D(a)^n}(A_1+A_2) &= \frac{d\mu^{n-1}}{\mu (n+1)}\frac{a^{n+1}-(a-1)^{n+1}}{a^n-(a-1)^n}+ \frac{1}{2(n+1)}\frac{a^{n+1}-(a+n)(a-1)^{n}}{a^n-(a-1)^n} \\
    &=S_{D(a)}(Y_{n-1})d\mu^{n-1}+\mathrm{Res}_n(a).
\end{align*}
This gives $S(V^{Y_{n-1}}_{\bullet,\bullet};E)\leqslant k_{n}(a,d,\mu)A_{Y_{n-1},\Delta_{n-1}}(E)$ by Lemma~\ref{res},
which proves \eqref{equation:induction} and completes the~proof of Proposition~\ref{proposition:delta}.

\subsection{Applications}
\label{subsection:fourfolds}
The only application of Theorem~\ref{theorem:delta} we could find is Theorem~\ref{theorem:double-spaces}.
Let us use assumptions and notations of Theorem~\ref{theorem:delta}.
Let $V=\mathbb P^{n-1}$ and $L=\mathcal{O}_{\mathbb P^{n-1}}(r)$.
Suppose that $1<\frac{n}{2}<r<n$.
Then $\mu=\frac{1}{r}$, $d=r^{n-1}$ and $a=\frac{n}{r}$.

\begin{lemma}
\label{lemma:Pn-Kn}
One has $k_n(a,d,\mu)<1$.
\end{lemma}

\begin{proof}
One has
$$
k_n(a,d,\mu)=
\frac{(2d\mu^{n-2}+1)a^{n+1}-(a+n)(a-1)^n-2d\mu^{n-2}(a-1)^{n+1}}{2(n+1)(a^n-(a-1)^n)}.
$$
Thus, it is enough to show that
$$
2(n+1)(a^n-(a-1)^n)-\big((2d\mu^{n-2}+1)a^{n+1}-(a+n)(a-1)^n-2d\mu^{n-2}(a-1)^{n+1}\big)>0.
$$
Substituting $\mu=\frac{1}{r}$, $d=r^{n-1}$, $a=\frac{n}{r}$, and multiplying by $r^{n+1}$, we get the~inequality
$$
(n^n-(n-r)^n(r+1))(2r-n)>0,
$$
which holds since $2r-n>0$ and $n>r>\frac{n}{2}$ by assumption.
\begin{comment}
\[
n^n-(n-r)^n(r+1) > 0 \Leftrightarrow \Big( \frac{n}{n-r} \Big)^n > r+1.
\]
The inequality on the~right holds since for $n \geqslant 2$ and $r \geqslant 1$ we have
\[
\Big( \frac{n}{n-r} \Big)^n > 2^n > n+1 > r+1.
\]
\end{comment}
\end{proof}

\begin{lemma}
\label{lemma:Pn-Kento-easy}
One has
$$
\frac{(n+1)(a^n-(a-1)^n)}{(n+1-a)a^n+(a-1)^{n+1}}>1.
$$
\end{lemma}

\begin{proof}
The inequality is equivalent to
$$
(n+1)(a^n-(a-1)^n)>(n+1-a)a^n+(a-1)^{n+1}.
$$
Substituting $a=\frac{n}{r}$, multiplying by $r^{n}$, and dividing by $n$, we get
$n^n-(r+1)(n-r)^n>0$, which holds since $1<\frac{n}{2}<r<n$.
\end{proof}

\begin{lemma}
\label{lemma:Pn-Kento}
One has
$$
\frac{a\delta(V)(n+1)(a^{n}-(a-1)^{n})}{n(a^{n+1}-(a-1)^{n+1})}>1.
$$
\end{lemma}

\begin{proof}
We have $\delta(V)=\delta(\mathbb{P}^{n-1})=1$. Thus, the~required inequality is equivalent to
$$
n(a^{n+1}-(a-1)^{n+1})-a(n+1)(a^n-(a-1)^n)<0.
$$
Substituting $a=\frac{n}{r}$, multiplying by $r^{n+1}$, and dividing by $n$, we get
$n^n-(r+1)(n-r)^n>0$, which holds since $1<\frac{n}{2}<r<n$.
\end{proof}

Theorem~\ref{theorem:double-spaces} follows from Lemmas~\ref{lemma:Pn-Kn}, \ref{lemma:Pn-Kento-easy}, \ref{lemma:Pn-Kento} and Theorem~\ref{theorem:delta}.

\section{Proof of Theorem~\ref{theorem:4-2}}
\label{section:4-2}

The goal of this section is to prove Theorem~\ref{theorem:4-2} and  describe
singular K-polystable limits of smooth Fano 3-folds in the~deformation family \textnumero 4.2.
We start with the~following (probably well-known) result, which we fail to find in the~literature.

\begin{proposition}
\label{proposition:GIT-2-2-curves}
Let $C$ be a $(2,2)$-curve in $\mathbb P^1 \times \mathbb P^1$. Then $C$ is
\begin{itemize}
\item GIT stable for $\mathrm{PGL}_2(\mathbb{C})\times \mathrm{PGL}_2(\mathbb{C})$-action $\iff$ it is smooth,
\item GIT strictly polystable $\iff$ it is one of the~curves in Theorem~\ref{theorem:4-2}.
\end{itemize}
\end{proposition}

\begin{proof}
Choose homogeneous coordinates $x,\,y$ of degree $(1,0)$ on $\mathbb P^1 \times \mathbb P^1$, and choose homogeneous coordinates $u,\,v$ of degree $(0,1)$.
Then $C$ is given by
$$
 \sum_{i=0}^{2}\sum_{j=0}^{2} a_{ij}x^{2-i}y^iu^{2-j}v^j=0.
$$
Observe that any one parameter subgroup $\lambda \colon \mathbb{C}^* \rightarrow \mathrm{PSL}_2(\mathbb C) \times  \mathrm{PSL}_2(\mathbb C)$ is conjugate to a diagonal one of the~form
$$
t \longmapsto
\Bigg(\begin{pmatrix}
t^{r_0} &    0        \\
  0      & t^{-r_0}   \\

\end{pmatrix}
,
\begin{pmatrix}
               t^{r_1}  & 0\\
              0        & t^{-r_1}
\end{pmatrix}\Bigg)
$$
for some integers $r_1 \geqslant r_0 \geqslant 0$ and $r_1 >0$, which we will write as $\lambda=(r_0,-r_0,r_1,-r_1)$.
Then the~Hilbert--Mumford function is
$$
\mu(f,\lambda)= \max \{ r_0(2-2i)+r_1(2-2j),\, a_{ij} \not = 0\}.
$$
Clearly, if $\mu(f,\lambda) \leqslant 0$, then $a_{00}=a_{10}=a_{01}=0$.
Moreover, if this inequality is strict, then we additionally have $a_{11}=0$.
Furthermore, we have
$$
\mu(x^2v^2,\lambda)=-\mu(y^2u^2,\lambda).
$$
So, at least one of $a_{20}$ and $a_{02}$ is zero.
Without loss of generality, we assume that $a_{20}=0$.
Therefore, if $\mu(f,\lambda) < 0$, then $a_{00}=a_{10}=a_{01}=a_{11}=a_{20}=0$.

Suppose that $C$ is singular at the~point $([1:0],[1:0])$, so that $a_{00}=a_{10}=a_{01}=0$,
and consider the~one parameter subgroup $\lambda=(1,-1,1,-1)$. Then
$$
\mu(f,\lambda)=4-2(i+j),
$$
which is non-positive if and only if $i+j\geqslant 2$.
But, since $a_{ij}=0$ whenever $i+j<2$, we conclude that $\mu(f,\lambda) \leqslant 0$ and $C$ is not stable.

Conversely, suppose there exists a one parameter subgroup $\lambda$ for which $\mu(f,\lambda) \leqslant 0$. Note that
$$
\mu(x^{2-i}y^iu^{2-j}v^j, \lambda)>0
$$
for any one parameter subgroup $\lambda$ provided that $i+j< 2$.
This gives $a_{00}=a_{10}=a_{01}=0$,
so that the~curve $C$ is singular at $([1:0],[1:0])$.

Now, let us describe the~unstable locus.
Suppose that $a_{00}=a_{10}=a_{01}=a_{11}=a_{20}=0$.
Consider the~one parameter subgroup $\lambda = (1,-1,2,-2)$. Then
$$
\mu(f,\lambda)=6-2(i+2j)
$$
which is negative if and only if $i+2j>3$. But since $a_{ij}=0$ whenever $i+2j\leqslant 3$ it follows that  $\mu(f,\lambda)<0$.
Similarly, one can show that $C$ is GIT-unstable if it can be given by
$$
a_{02}x^2v^2+a_{12}xyv^2+a_{21}y^2uv+a_{22}y^2v^2=0.
$$
This describes all possibilities for the~curve $C$ to be GIT-semistable, which easily implies the~description of GIT-polystable $(2,2)$-curves.
\end{proof}

Now, we set $V=\mathbb{P}^1\times\mathbb{P}^1$.
Let $L=\mathcal{O}_{V}(1,1)$, let $R$ be a curve in $|2L|$, set
$$
Y=\mathbb{P}\big(\mathcal{O}_{V}\oplus \mathcal{O}_{V}(L)\big),
$$
let $\pi\colon Y\to V$ be the~natural projection, let~$S^-$ and $S^+$ be disjoint sections of $\pi$ such that
$$
S^+\sim S^-+\pi^*(L).
$$
Finally, we~set $F=\pi^*(R)$,
and let $\phi\colon X\to Y$ be the~blow up at the~intersection $S^+\cap F$.
If~$R$ is smooth, then $X$ is K-polystable \cite{Book}.
Theorem~\ref{theorem:4-2} says that $X$ is also K-polystable
in the~case when $R$ is one of the~following singular curves:
\begin{enumerate}
\item[$(\mathrm{1})$] $C_1+C_2$, where $C_1$ and $C_2$ are smooth curves in $|L|$ such that $|C_1\cap C_2|=2$;
\item[$(\mathrm{2})$] $\ell_1+\ell_2+\ell_3+\ell_4$, where $\ell_1$ and $\ell_2$ are two distinct smooth curves of degree $(1,0)$,
and $\ell_3$ and $\ell_4$ are two distinct smooth curves of degree $(0,1)$;
\item[$(\mathrm{3})$] $2C$, where $C$ is a smooth curve in $|L|$.
\end{enumerate}

Now, let us prove Theorem~\ref{theorem:4-2}. We start with

\begin{remark}
\label{remark:4-2-toric}
Suppose that $R=\ell_1+\ell_2+\ell_3+\ell_4$, where $\ell_1$ and $\ell_2$ are two distinct smooth curves in $V$ of degree $(1,0)$,
and $\ell_3$ and $\ell_4$ are two distinct smooth curves of degree $(0,1)$.
Then $X$ is toric,
and it corresponds to the moment polytope in $M_{\mathbb{R}}$ whose vertices are
$$
\begin{matrix}
(0,0,1),& (1,0,1),& (1,1,1),& (0,1,1),\\
(1,1,0),& (-1,1,0),& (-1,-1,0),& (1,-1,0),\\
(0,0,-1),& (-1,0,-1),& (-1,-1,-1),& (0,-1,-1).
\end{matrix}
$$
The barycenter of the moment polytope is the origin, so $X$ is K-polystable.
\end{remark}

Our next step is the~following simple lemma:

\begin{lemma}
\label{lemma:4-2-PGL-2}
Suppose $R=2C$ for a smooth curve $C\in |L|$.
Then $X$ is K-polystable.
\end{lemma}

\begin{proof}
In this case, the~morphism $\phi$ is a weighted blow up at the~intersection $\pi^*(C)\cap S^+$,
and $X$ has non-isolated singularities along a smooth  curve, which we will denote by $\overline{C}$.
The threefold $X$ can be obtained in a slightly different way.
Let us describe it.

Set $W=V\times\mathbb{P}^1$,
let $\varpi\colon W\to V$ be the~natural projection,
let $\widetilde{S}^-$ and $\widetilde{S}^+$ be its disjoint sections,
and let $\widetilde{E}=\varpi^*(C)$.
Then there exists commutative diagram
$$\xymatrix{
&U\ar[dl]_{\alpha}\ar[dr]^{\psi}&\\
W\ar[dr]_{\varpi} &&X\ar[dl]^{\pi\circ\phi} \\
&V&}
$$
where $\alpha$ blows up the~intersection curves $\widetilde{E}\cap\widetilde{S}^-$ and $\widetilde{E}\cap\widetilde{S}^+$,
and~$\psi$ contracts the~proper transform of the~surface $\widetilde{E}$ to the~curve $\overline{C}$.
Moreover, we may assume that $\phi\circ\psi$ maps the~proper transforms of the~surfaces $\widetilde{S}^-$ and $\widetilde{S}^+$
to the~surfaces $S^-$ and $S^+$, respectively.

Let $\widehat{E}$ be the~proper transform on the~threefold $U$ of the~surface $\widetilde{E}$.
We may assume that the~curve $C$ is the~diagonal curve in $V=\mathbb{P}^1\times\mathbb{P}^1$.Using this, we see that
$$
\mathrm{Aut}(X)\cong\mathrm{Aut}(U)\cong\mathrm{Aut}\big(W,\widetilde{E}+\widetilde{S}^-+\widetilde{S}^+\big)\cong\mathrm{PGL}_2(\mathbb{C})\times(\mathbb{G}_m\rtimes\mumu_2)\times\mumu_2,
$$
and $\widehat{E}$ is the~only $\mathrm{Aut}(X)$-invariant prime divisor over $X$.
Thus, using \cite{Zhuang}, we conclude that the~threefold $X$ is K-polystable if $\beta(\widehat{E})>0$.
Let us compute $\beta(\widehat{E})$.

We let $F^-$ and $F^+$ be $\alpha$-exceptional surfaces such that $\alpha(F^-)\subset\widetilde{S}^-$ and $\alpha(F^+)\subset\widetilde{S}^+$,
let $\widehat{S}^-$ and $\widehat{S}^+$ be the~proper transforms on $U$ of the~surfaces $S^-$ and $S^+$, respectively.
Further, set
$H_1=(\mathrm{pr}_1\circ\alpha)^*(\mathcal{O}_{\mathbb{P}^1}(1))$,
$H_2=(\mathrm{pr}_2\circ\alpha)^*(\mathcal{O}_{\mathbb{P}^1}(1))$,
$H_3=(\mathrm{pr}_3\circ\alpha)^*(\mathcal{O}_{\mathbb{P}^1}(1))$,
where $\mathrm{pr}_1$, $\mathrm{pr}_2$, $\mathrm{pr}_3$ are projections $W\to\mathbb{P}^1$
such that $\mathrm{pr}_1$ and $\mathrm{pr}_2$ factors through $\varpi$.
Then
$$
\psi^*(-K_X)\sim -K_U\sim 2(H_1+H_2+H_3)-F^--F^+\sim 2\widehat{E}+\widehat{S}^-+\widehat{S}^++2(F^-+F^+).
$$
Now, we take $u\in\mathbb{R}_{\geqslant 0}$. Then the~divisor $\psi^*(-K_X)-u\widehat{E}$ is $\mathbb{R}$-rationally equivalent to
$$
(2-u)(H_1+H_2)+2H_3+(u-1)(F^-+F^+)\sim_{\mathbb{R}}(2-u)\widehat{E}+\widehat{S}^-+\widehat{S}^++2(F^-+F^+),
$$
and $\widehat{S}^-+\widehat{S}^++2(F^-+F^+)$ is not big,
so $\psi^*(-K_X)-u\widehat{E}$ is pseudoeffective $\iff$ $u\leqslant 2$.
Moreover, if $u\in[0,1]$, then the~divisor $\psi^*(-K_X)-u\widehat{E}$  is nef.
Furthermore, if $u\in[1,2]$, then the~Zariski decomposition of the~divisor  $\psi^*(-K_X)-u\widehat{E}$ is given by
$$
\psi^*(-K_X)-u\widehat{E}\sim_{\mathbb{R}}\underbrace{(2-u)(H_1+H_2)+2H_3}_{\text{positive part}}+\underbrace{(u-1)(F^-+F^+)}_{\text{negative part}}.
$$
Hence, we have
\begin{multline*}
\beta(\widehat{E})=1-\frac{1}{(-K_X)^3}\int_{0}^{2}\mathrm{vol}\Big(\psi^*(-K_X)-u\widehat{E}\Big)du=\\
=1-\frac{1}{28}\int_{0}^{1}\Big((2-u)(H_1+H_2)+2H_3+(u-1)(F^-+F^+)\Big)^3du-\frac{1}{28}\int_{1}^{2}\Big((2-u)(H_1+H_2)+2H_3\Big)^3du=\\
=1-\int_{0}^{1}8u^3-24u^2+28du-\int_{1}^{2}12(2-u)^2du=\frac{1}{14}>0,
\end{multline*}
which implies that $X$ is K-polystable.
\end{proof}

To complete the~proof of Theorem~\ref{theorem:4-2},
let us present $X$ as a codimension two complete intersection in a toric variety.
Let $T=(\mathbb{C}^7\setminus Z(I))/\mathbb{G}_m^2$, where the~$\mathbb{G}_m^2$-action is given by
$$
\left(\begin{array}{ccccccc}
x & y & z & w & u & v & s  \\
1 & 1 & 1 & 1 & 1 & 0 & 2 \\
0 & 0 & 0 & 0 & 1 & 1 & 0
\end{array}\right),
$$
and $I$ is the~irrelevant ideal $\langle x,y,z,w,s \rangle \cap \langle u, v\rangle$.
Let $\widetilde{\mathbb{P}} = \mathrm{Proj} ( \mathcal{O}_{\mathbb{P}^3} \oplus \mathcal{O}_{\mathbb{P}^3}(1))$.
Then we can identify $\widetilde{\mathbb{P}}$ with the~hypersurface in $T$ given by
$$
s=f(x,y,z,w),
$$
where $f(x,y,z,w)$ is any non-zero homogeneous polynomial of degree $2$.
Since $Y$ can be obtained by blowing up the~quadric cone over the~surface $\{xy=zw\}\subset\mathbb{P}^3$ at the~vertex,
we can identify $Y$ with the~complete intersection in $T$ given by
$$\left\{\aligned
&xy=zw , \\
&s=f(x,y,z,w),
\endaligned
\right.
$$
Then the~projection $\pi\colon T\to V$ is given by
$$
(x,y,z,w,u,v,s)\mapsto (x,y,z,w),
$$
where we identify $V$ with $\{xy=zw\}\subset\mathbb{P}^3$.
Then the~surface $S^-$ is cut out on $Y$ by~$v=0$.
Moreover, we can assume that $S^+$ is cut out on $Y$ by $u=0$,
and we can identify $R$ with the~curve in $S^+$ that is cut out by $s=0$.

Let $\varphi\colon \overline{T} \to T$ be the blow up of $T$ along $u = s = 0$.
Then $\overline{T}=(\mathbb{C}^8\setminus Z(\overline{I}))/\mathbb{G}_m^3$,
where the~torus action is given by the~matrix
$$
\left(\begin{array}{cccccccc}
x & y & z & w & u & v & s & t \\
1 & 1 & 1 & 1 & 1 & 0 & 2 & 0 \\
0 & 0 & 0 & 0 & 1 & 1 & 0 & 0\\
0 & 0 & 0 & 0 & 1 & 0 & 1 & -1
\end{array}\right),
$$
and the~irrelevant ideal
$$
\overline{I} = \langle x,y,z,w,s \rangle \cap \langle x,y,z,w,t \rangle \cap \langle u,v \rangle \cap \langle u,s \rangle \cap \langle v, t \rangle.
$$
Then $\varphi$ induces the~blow up of $Y$ along $R$.
Thus, we can identify $X$ with the~complete intersection in the~toric variety $\overline{T}$ given by
$$\left\{\aligned
&xy=zw , \\
&st=f(x,y,z,w).
\endaligned
\right.
$$
Now, the~subgroup $\Gamma\cong\mathbb{G}_m$ of the~group $\mathrm{Aut}(X)$ mentioned in Section~\ref{section:intro} can be explicitly seen
--- it consists of all automorphisms
$$
(x,y,z,w,u,v,s,t)\mapsto (x,y,z,w, \lambda u,v,s,t),
$$
where $\lambda\in\mathbb{C}^\ast$.
Similarly, we can choose the~involution $\iota\in\mathrm{Aut}(X)$ to be the~involution
$$
(x,y,z,w,u,v,s,t)\mapsto(x,y,z,w,v,u,t,s).
$$
Note that $\iota$ is not canonically defined, since we can conjugate it with an element in $\Gamma$.

Suppose that $R=C_1+C_2$, where $C_1$ and $C_2$ are smooth curves in $|L|$ that meet transversally at two points.
Then, up to a change of coordinates, we may assume that
$$
f(x,y,z,t)=xy-\lambda (z^2+w^2).
$$
where $\lambda\in\mathbb{C}$ such that $\lambda\not\in\{0,2,-2\}$.
Then $X$ is the~complete intersection in $\overline{T}$ given by
$$\left\{\aligned
&xy=zw, \\
&st=xy-\lambda (z^2+w^2),
\endaligned
\right.
$$
Note that $\mathrm{Aut}(X)$ contains automorphisms
$$
(x,y,z,w,u,v,s,t)\mapsto\Big(\mu x,\frac{y}{\mu},z,w,u,v,s,t\Big),
$$
where $\mu\in\mathbb{C}^\ast$. Similarly, the~group $\mathrm{Aut}(X)$ contains two involutions:
$$
(x,y,z,w,u,v,s,t)\mapsto (y,x,z,w,u,v,s,t)
$$
and
$$
(x,y,z,w,u,v,s,t)\mapsto (x,y,w,z,u,v,s,t).
$$
Let $G$ be the~subgroup in $\mathrm{Aut}(X)$ that is generated by all automorphisms described above.
Then $G\cong\mathbb{G}_m^2\rtimes\mumu_2$, and we have the~following result:

\begin{lemma}
\label{lemma:4-2-group-action}
The following assertions hold:
\begin{enumerate}
\item[$(\mathrm{a})$] $X$ does not contain $G$-fixed points,
\item[$(\mathrm{b})$] $X$ does not contain $G$-invariant irreducible curves,
\item[$(\mathrm{c})$] $X$ contains two $G$-invariant irreducible surfaces --- they are cut out by $z\pm w=0$.
\end{enumerate}
\end{lemma}

\begin{proof}
Left to the~reader.
\end{proof}

Now, we can complete the~proof of Theorem~\ref{theorem:4-2}.
Suppose that $X$ is not K-polystable.
Using \cite{Zhuang}, we see that there is a $G$-invariant prime divisor $\mathbf{F}$ over $X$ such that $\beta(\mathbf{F})\leqslant 0$.
Let $Z$ be the~center of this divisor on $X$.
By Lemma~\ref{lemma:4-2-group-action}, $Z$ is a surface and
$$
Z\sim (\pi\circ\phi)^*(L).
$$
Then, as in \cite{div-stability}, we compute $\beta(\mathbf{F})=\beta(Z)>0$.
This shows that $X$ is K-polystable.

\section{Proof of Theorem~\ref{theorem:3-9}}
\label{section:3-9}

In this section, we prove Theorem~\ref{theorem:3-9}.
This result describes all singular K-polystable limits of smooth Fano 3-folds in the~family~\textnumero 3.9.
To show this, we need

\begin{theorem}[{\cite[Theorem 2]{Stability-quartic-curves}, \cite[Example 7.13]{mukai-book}, \cite{Unstable-quartics}}]
\label{theorem:stable-quartics}
Let $C$ be a quartic curve in $\mathbb{P}^2$. Then the~curve $C$ is
\begin{itemize}
\item GIT stable for $\mathrm{PGL}_3(\mathbb{C})$-action $\iff$ it is smooth or has $\mathbb{A}_1$ or $\mathbb{A}_2$-singularities,
\item GIT strictly polystable $\iff$ it is one of the~remaining curves in Theorem~\ref{theorem:3-9}.
\end{itemize}
\end{theorem}

Let us prove Theorem~\ref{theorem:3-9}. Set $V=\mathbb{P}^2$, set $L=\mathcal{O}_{\mathbb{P}^2}(2)$,
and set $Y=\mathbb{P}(\mathcal{O}_{V}\oplus \mathcal{O}_{V}(L))$.
Let $\pi\colon Y\to V$ be the~natural projection, set $H=\pi^*(L)$, let~$S^-$ and $S^+$ be disjoint sections of $\pi$
such that $S^+\sim S^-+H$, and let $R$ be one of the~following curves:
\begin{enumerate}
\item[$(\mathrm{1})$] a reduced quartic curve with at most $\mathbb{A}_1$ or $\mathbb{A}_2$ singularities;
\item[$(\mathrm{2})$] $C_1+C_2$, where $C_1$ and $C_2$ are smooth conics that are tangent at two points;
\item[$(\mathrm{3})$] $C+\ell_1+\ell_2$, where $C$ is a smooth conic, $\ell_1$ and $\ell_2$ are distinct lines tangent to $C$;
\item[$(\mathrm{4})$] $2C$, where $C$ is a smooth conic in $|L|$.
\end{enumerate}
Set~$F=\pi^*(R)$, and let $\phi\colon X\to Y$ be the~blow up at the~complete intersection $S^+\cap F$.
Then $X$ is a singular Fano threefold, and our Theorem~\ref{theorem:3-9} claims that $X$ is K-polystable.
To prove this, we start with the~most singular (and the~most symmetric case).

\begin{lemma}
\label{lemma:3-9-PGL-2}
Suppose that $R=2C$ for a smooth conic $C\subset\mathbb{P}^2$. Then $X$ is K-polystable.
\end{lemma}

\begin{proof}
In this case, the~threefold $X$ has non-isolated singularities along a smooth curve,
and the~proof is very similar to the~proof of Lemma~\ref{lemma:4-2-PGL-2}.
Namely, we have
\begin{equation}
\label{equation:3-9-PGL-2-Gm}
\mathrm{Aut}(X)\cong\mathrm{PGL}_2(\mathbb{C})\times(\mathbb{G}_m\rtimes\mumu_2),
\end{equation}
and there exists exactly one $\mathrm{Aut}(X)$-invariant prime divisor over $X$ ---
the exceptional divisor of the~blow up of $X$ along the~curve $\mathrm{Sing}(X)$.
So, to check that  $X$ is K-polystable, it is enough
to compute the~$\beta$-invariant of this prime divisor. Let us give details.

As in the~proof of Lemma~\ref{lemma:4-2-PGL-2}, we set $W=V\times\mathbb{P}^1$.
Let $\varpi\colon W\to V$ be the~natural projection,
let $\widetilde{S}^-$ and $\widetilde{S}^+$ be its disjoint sections,
and let $\widetilde{E}=\varpi^*(C)$.
Then there exists the~following commutative diagram:
\begin{equation}
\label{equation:3-9-PGL-2}
\xymatrix{
&U\ar[dl]_{\alpha}\ar[dr]^{\psi}&\\
W\ar[dr]_{\varpi} &&X\ar[dl]^{\pi\circ\phi} \\
&V&}
\end{equation}
such that
\begin{itemize}
\item $\alpha$ is a blow up along the~curves $\widetilde{E}\cap\widetilde{S}^-$ and $\widetilde{E}\cap\widetilde{S}^+$,
\item $\psi$ is a contraction of the~proper transform of $\widetilde{E}$ to the~curve $\mathrm{Sing}(X)$,
\item $\phi\circ\psi$ maps the~proper transforms of $\widetilde{S}^-$ and $\widetilde{S}^+$ to $S^-$ and $S^+$, respectively.
\end{itemize}
This easily implies \eqref{equation:3-9-PGL-2-Gm}.
Similarly, we see that \eqref{equation:3-9-PGL-2} is $\mathrm{Aut}(X)$-equivariant.

Let $\widehat{E}$ be the~$\psi$-exceptional divisor.
Then $\widehat{E}$ is the~only $\mathrm{Aut}(X)$-invariant prime divisor over the~threefold $X$.
Thus,  if $\beta(\widehat{E})>0$, them $X$ is K-polystable \cite{Zhuang}.

We let $F^-$ and $F^+$ be $\alpha$-exceptional surfaces such that $\alpha(F^-)\subset\widetilde{S}^-$ and $\alpha(F^+)\subset\widetilde{S}^+$,
let $\widehat{S}^-$ and $\widehat{S}^+$ be the~proper transforms on $U$ of the~surfaces $S^-$ and $S^+$, respectively.
Set~$H_1=(\mathrm{pr}_1\circ\alpha)^*(\mathcal{O}_{\mathbb{P}^1}(1))$ for
the~projection $\mathrm{pr}_1\colon W\to\mathbb{P}^1$, set $H_2=(\varpi\circ\alpha)^*(\mathcal{O}_{V}(1))$.
Then $\widehat{E}\sim 2H_2-F^--F^+$, which gives
$$
\psi^*(-K_X)\sim -K_U\sim 2H_1+3H_2-F^--F^+\sim_{\mathbb{Q}} 2H_1+\frac{3}{2}\widehat{E}+\frac{1}{2}(F^-+F^+).
$$
Take $u\in\mathbb{R}_{\geqslant 0}$. Then
$$
\psi^*(-K_X)-u\widehat{E}\sim_{\mathbb{R}}
2H_1+(3-2u)H_2+(u-1)(F^-+F^+)\sim_{\mathbb{R}} 2H_1+\frac{3-2u}{2}\widehat{E}+\frac{1}{2}(F^-+F^+).
$$
This shows that $\psi^*(-K_X)-u\widehat{E}$ is pseudoeffective $\iff$ $u\leqslant \frac{3}{2}$.
Moreover, if $u\in[0,1]$, then the~divisor $\psi^*(-K_X)-u\widehat{E}$  is nef.
If $1<u\leqslant\frac{3}{2}$, its Zariski decomposition is
$$
\psi^*(-K_X)-u\widehat{E}\sim_{\mathbb{R}}\underbrace{2H_1+(3-2u)H_2}_{\text{positive part}}+\underbrace{(u-1)(F^-+F^+)}_{\text{negative part}}.
$$
Hence, we have
\begin{multline*}
\beta(\widehat{E})=1-\frac{1}{(-K_X)^3}\int_{0}^{\frac{3}{2}}\mathrm{vol}\Big(\psi^*(-K_X)-u\widehat{E}\Big)du=\\
=1-\frac{1}{26}\int_{0}^{1}\Big(2H_1+(3-2u)H_2+(u-1)(F^-+F^+)\Big)^3du-\frac{1}{26}\int_{1}^{\frac{3}{2}}\Big(2H_1+(3-2u)H_2\Big)^3du=\\
=1-\frac{1}{26}\int_{0}^{1}16u^3-36u^2+26du-\frac{1}{26}\int_{1}^{\frac{3}{2}}24u^2-72u+54du=\frac{7}{26}>0,
\end{multline*}
which implies that $X$ is K-polystable.
\end{proof}

Similarly, we can show that $X$ is K-polystable if $R=C_1+C_2$, where $C_1$ and $C_2$ are smooth conics that are tangent at two points.
Indeed, in this case, the~full automorphism group $\mathrm{Aut}(X)$ contains a subgroup $G$ such that
$$
G\cong\big(\mathbb{G}_m\big)^2\rtimes\mumu_2^2,
$$
the~threefold $X$ does not contains $G$-fixed points,
and the~only $G$-invariant irreducible curve in $X$ is a smooth fiber of the~conic bundle $\pi\circ\phi$.
Therefore, arguing exactly as in the~proofs of \cite[Lemma~4.64]{Book} and \cite[Lemma~4.66]{Book},
we see that $X$ is K-polystable.

However, this approach fails in the~case when $R$ has a singular point of type $\mathbb{A}_1$ or $\mathbb{A}_2$.
To overcome this difficulty, we will use another approach described in the~end of Section~\ref{section:intro}.
Namely, we proved in Section~\ref{section:CD-properties} that $\mathrm{Aut}(X)$ contains an~involution $\iota$
such that $\iota$ swaps the~proper transforms of $S^-$ and $S^+$, $X/\iota\cong Y$,
and the~following diagram commutes:
$$\xymatrix{
&X\ar[dl]_{\phi}\ar[dr]^{\rho}&\\
Y\ar[dr]_{\pi} &&Y\ar[dl]^{\pi} \\
&V&}
$$
where $\rho$ is the~quotient map.
Moreover, we also proved that the~double cover $\rho$ is ramified over a divisor $B\in |2S^+|$ such that
the morphism $B\to V$ induced by $\pi$ is a double cover ramified in the~curve $R$.
Set $\Delta=\frac{1}{2}B$. Then
$$
-K_X\sim_{\mathbb{Q}}\rho^*\big(K_Y+\Delta\big),
$$
and $(Y,\Delta)$ has Kawamata log terminal singularities. Therefore, $(Y,\Delta)$ is a log Fano pair.
Moreover, it follows from \cite{Liu-Zhu} that
\begin{center}
$X$ is K-polystable $\iff$ $\left(Y, \frac{1}{2}B\right)$ is K-polystable.
\end{center}
However, everything in life comes with a price:
the action of the~group $\Gamma\cong\mathbb{G}_m$ described earlier in Section~\ref{section:intro}
does not descent to $Y$ via $\rho$, because $\Gamma$ does not commute with $\iota$.
Thus, the~group $\mathrm{Aut}(Y,\Delta)$ is much smaller than the~group $\mathrm{Aut}(X)$.

To explicitly describe $B\subset Y$,
consider $Y$ as the~toric variety $(\mathbb{C}^5\setminus Z(I))/\mathbb{G}_m^2$ such that
the~torus action is given by the~matrix
$$
\left(\begin{array}{ccccc}
x_1 & x_2 & x_3 & x_4 & x_5 \\
1  &  1  &  1  &  2  &  0  \\
0  &  0  &  0  &  1  &  1
\end{array}\right),
$$
with irrelevant ideal $I=\langle x_1, x_2, x_3 \rangle \cap\langle x_4, x_5 \rangle$.
Let us also consider $x_1,x_2,x_3$ as coordinates on $V=\mathbb{P}^2$,
so that the~projection $\pi$ is given by
$$
(x_1,x_2,x_3,x_4,x_5)\mapsto (x_1,x_2,x_3).
$$
Then $S^-=\{x_5=0\}$. Moreover, we may assume that $S^+=\{x_4=0\}$,
and $B$ is given by
$$
x_4^2-f_4(x_1,x_2,x_3)x_5^2=0,
$$
where $f_4(x_1,x_2,x_3)$ is a quartic polynomial such that $R=\{f_4(x_1,x_2,x_3)=0\}$.

In the~remaining part of the~section, we will prove that the~pair $(Y,\Delta)$ is K-polystable.
Recall that $H=\pi^*(L)$. Note also that
$$
-(K_Y+\Delta)\sim_{\mathbb{Q}}S^-+\frac{3}{2}H.
$$
We will split the~proof in several lemmas and propositions. We start with

\begin{lemma}
\label{lemma:3-9-S-minus}
Let $P$ be a point in $S^-$. Then $\delta_P(Y,\Delta)>1$.
\end{lemma}

\begin{proof}
Let us apply Lemma~\ref{lemma:Kento}. We have
$$
\delta_P(Y,\Delta)=\delta_P(Y;D(a))\geqslant
\mathrm{min}\Big\{\frac{4(a^3-(a-1)^3)}{(4-a)a^3+(a-1)^{4}},\frac{4(a^3-(a-1)^3)}{3(a^{4}-(a-1)^4)}\delta(V;L)\Big\},
$$
where $D(a)=-(K_Y+\Delta)$ and $a=\frac{3}{2}$. Thus, we have
$$
\delta_P(Y,\Delta)\geqslant\mathrm{min}\Big\{\frac{26}{17},\frac{13}{15}\delta(V;L)\Big\}.
$$
But
$$
\delta(V;L)=\delta\Big(V;\frac{2}{3}(-K_V)\Big)=\frac{3}{2}\delta(V;-K_V)=\frac{3}{2}\delta(V)=\frac{3}{2}\delta(\mathbb{P}^2)=\frac{3}{2},
$$
so that $\delta_P(Y,\Delta)\geqslant\frac{13}{10}$.
\end{proof}

Similarly, applying  Proposition~\ref{proposition:induction-step}, we obtain

\begin{lemma}
\label{lemma:3-9-P-not-in-B}
Let $P$ be a point $Y$ such that $P\not\in\mathrm{Sing}(B)$.
Then $\delta_P(Y,\Delta)>1$.
\end{lemma}

\begin{proof}
By Lemma~\ref{lemma:3-9-S-minus}, we may assume that $P\not\in S^-$.
Then Proposition~\ref{proposition:induction-step} gives
$$
\delta_P(Y,\Delta)=\delta_P(Y;D(a))\geqslant\frac{8(3a^2-3a+1)}{8d \mu a^3+6(1-2d\mu)a^2+8(d\mu-1)a-2d \mu+3},
$$
where $D(a)=-(K_Y+\Delta)$, $a=\frac{3}{2}$, $d=L^2=4$, $\mu=\frac{1}{2}$. This gives $\delta_P(Y,\Delta)\geqslant\frac{52}{49}$.
\end{proof}

The two most difficult parts of the~proof that $(Y,\Delta)$ is K-polystable are the~following two propositions,
which will be proved in Subsections~\ref{section:A1} and \ref{section:A2} later.

\begin{proposition}
\label{proposition:A1}
Let $P$ be a point in $B$ that such $B$ has singular point of type $\mathbb{A}_1$ at $P$,
and let $\mathbf{F}$ be a prime divisor over $Y$ such that $P=C_Y(\mathbf{F})$.
Then $\beta_{Y,\Delta}(\mathbf{F})>0$.
\end{proposition}

\begin{proposition}
\label{proposition:A2}
Let $P$ be a point in $B$ such that $B$ has singular point of type $\mathbb{A}_2$ at $P$,
and let $\mathbf{F}$ be a prime divisor over $Y$ such that $P=C_Y(\mathbf{F})$.
Then $\beta_{Y,\Delta}(\mathbf{F})>0$.
\end{proposition}

By Lemmas~\ref{lemma:3-9-S-minus} and \ref{lemma:3-9-P-not-in-B} and Propositions~\ref{proposition:A1} and \ref{proposition:A2},
the log pair $(Y,\Delta)$ is K-stable
in the~case when $R$ is a reduced plane quartic curve that has at most $\mathbb{A}_1$ or $\mathbb{A}_2$ singularities.
Therefore, to complete the~proof, we may assume that $R$ is one of the~following curves:
\begin{enumerate}
\item[$(\mathrm{2})$] $C_1+C_2$, where $C_1$ and $C_2$ are smooth conics that are tangent at two points;
\item[$(\mathrm{3})$] $C+\ell_1+\ell_2$, where $C$ is a smooth conic, $\ell_1$ and $\ell_2$ are distinct lines tangent to $C$;
\item[$(\mathrm{4})$] $2C$, where $C$ is a smooth conic in $|L|$.
\end{enumerate}
Hence, appropriately changing coordinates $x_1,x_2,x_3$, we may assume that
$$
f_4(x_1,x_2,x_3)=(x_1x_2-x_3^2)(x_1x_2-\lambda x_3^2),
$$
where one of the~following three cases holds:
\begin{enumerate}
\item[$(\mathrm{2})$] $\lambda\not\in\{0,1\}$, $R=C_1+C_2$, where $C_1=\{x_1x_2=x_3^2\}$ and $C_2=\{x_1x_2=\lambda x_3^2\}$;
\item[$(\mathrm{3})$] $\lambda=0$, $R=C+\ell_1+\ell_2$, where $C=\{x_1x_2=x_3^2\}$, $\ell_1=\{x_1=0\}$ and $\ell_2=\{x_2=0\}$;
\item[$(\mathrm{4})$] $\lambda=1$, $R=2C$, where $C=\{x_1x_2=x_3^2\}$.
\end{enumerate}
In each case,  the~group $\mathrm{Aut}(Y,\Delta)$ contains an involution $\tau$ such that
$$
\tau(x_1,x_2,x_3,x_4,x_5)=(x_2,x_1,x_3,x_4,x_5).
$$

\begin{lemma}
\label{lemma:3-9-cat-eye}
Suppose that $\lambda\not\in\{0,1\}$. Then $(Y,\Delta)$ is K-polystable.
\end{lemma}

\begin{proof}
Suppose $(Y,\Delta)$ is not K-polystable.
It follows from \cite{Zhuang} that there is a $\langle\tau\rangle$-invariant prime divisor $\mathbf{F}$ over $Y$ such that $\beta_{Y,\Delta}(\mathbf{F})\leqslant 0$.
Let $P$ be a general point in $C_Y(\mathbf{F})$. Then
$$
\delta_P\big(Y,\Delta\big)\leqslant 1.
$$
But $P\not\in\mathrm{Sing}(B)$, since $\mathrm{Sing}(B)$ consists of two singular points that are swapped by $\tau$.
Then $\delta_P(Y,\Delta)>1$ by Lemmas~\ref{lemma:3-9-S-minus} and \ref{lemma:3-9-P-not-in-B},
which is a contradiction.
\end{proof}

\begin{lemma}
\label{lemma:3-9-ox}
Suppose $\lambda=0$. Then $(Y,\Delta)$ is K-polystable.
\end{lemma}

\begin{proof}
The surface $B$ has a singular point of type $\mathbb{A}_1$,
and two singular points of type~$\mathbb{A}_3$, that are swapped by $\tau$.
Arguing as in the~proof of Lemma~\ref{lemma:3-9-cat-eye} and using Propositions~\ref{proposition:A1},
we see that $X$ is K-polystable.
\end{proof}

\begin{lemma}[{cf. Lemma~\ref{lemma:3-9-PGL-2}}]
\label{lemma:3-9-PGL2-revisited}
Suppose $\lambda=1$. Then $(Y,\Delta)$ is K-polystable.
\end{lemma}

\begin{proof}
In this case, we have $R=2C$, where $C$ is an irreducible conic. Then $B=B_1+B_2$,
where $B_1$ and $B_2$ are smooth surfaces in $|S^+|$
that intersect transversally along a smooth curve such that $\pi(B_1\cap B_2)=C$.

We already know from Lemma~\ref{lemma:3-9-PGL-2} that the~threefold $X$ is K-polystable in this case,
so~that $(Y,\Delta)$ is also K-polystable  \cite{Liu-Zhu}.
Let us prove this directly for consistency.

Let $W=V\times\mathbb{P}^1$,
let $\varpi\colon W\to V$ be the~natural projection,
let $\widetilde{S}^-$,  $\widetilde{B}_1$, $\widetilde{B}_2$ be its disjoint sections,
and let $\widetilde{E}=\varpi^*(C)$.
Then there exists the~following commutative diagram:
$$
\xymatrix{
&U\ar[dl]_{\alpha}\ar[dr]^{\psi}&\\
W\ar[dr]_{\varpi} &&Y\ar[dl]^{\pi} \\
&V&}
$$
such that $\alpha$ is a blow up along the~curve $\widetilde{E}\cap\widetilde{S}^-$,
the~morphism $\psi$ is a contraction of the~proper transform of the~surface $\widetilde{E}$ to the~intersection curve $B_1\cap B_2$
such that $\psi$ maps the~proper transforms of the~surfaces $\widetilde{S}^-$, $\widetilde{B}_1$, $\widetilde{B}_2$
to the~surfaces $S^-$, $B_1$, $B_2$, respectively. Then
$$
\mathrm{Aut}(Y,\Delta)\cong
\mathrm{Aut}(U)\cong\mathrm{Aut}\big(W,\widetilde{B}_1+\widetilde{B}_2+\widetilde{E}+\widetilde{S}^-\big)\cong
\mathrm{PGL}_2(\mathbb{C})\times\mumu_2.
$$
Note that the~commutative diagram above is $\mathrm{Aut}(Y,\Delta)$-equivariant.

Let $F$ be $\alpha$-exceptional surface, let $\widehat{E}$ be the~$\psi$-exceptional surface,
let $\widehat{B}_1$ and $\widehat{B}_2$ be the~proper transforms on $U$ of the~surfaces $B_1$ and $B_2$,
respectively. Set $\widehat{\Delta}=\frac{1}{2}(\widehat{B}_1+\widehat{B}_2)$. Then
$K_U+\widehat{\Delta}\sim_{\mathbb{Q}}\psi^*(K_Y+\Delta)$,
so that $\psi$ is log crepant for $(U,\widehat{\Delta})$.
Then $A_{Y,\Delta}(\widehat{E})=1$.

First, we compute $\beta_{Y,\Delta}(\widehat{E})$.
Set $H_1=(\mathrm{pr}_1\circ\alpha)^*(\mathcal{O}_{\mathbb{P}^1}(1))$ and $H_2=(\varpi\circ\alpha)^*(\mathcal{O}_{V}(1))$,
where  $\mathrm{pr}_1$ is the~natural projection $W\to\mathbb{P}^1$.
Then $\widehat{\Delta}\sim_{\mathbb{Q}} H_1$ and $\widehat{E}\sim 2H_2-F$, so that
$$
\psi^*\big(K_Y+\Delta\big)\sim_{\mathbb{Q}} K_U+\widehat{\Delta}\sim_{\mathbb{Q}} H_1+3H_2-F\sim_{\mathbb{Q}} H_1+\frac{3}{2}\widehat{E}+\frac{1}{2}F.
$$
Let $u$ be a non-negative real number. Then
$$
\psi^*\big(K_Y+\Delta\big)-u\widehat{E}\sim_{\mathbb{R}} H_1+(3-2u)H_2+(u-1)F\sim_{\mathbb{R}} H_1+\frac{3-2u}{2}\widehat{E}+\frac{1}{2}F,
$$
and this divisor is pseudoeffective $\iff$ $u\leqslant \frac{3}{2}$.
For $u\in[0,\frac{3}{2}]$, let $P(u)$ be the~positive part of the~Zariski decomposition of $\psi^*(K_Y+\Delta)-u\widehat{E}$,
and let $N(u)$ be the~negative part. Then
$$
P(u)\sim_{\mathbb{R}}\left\{\aligned
&H_1+(3-2u)H_2+(u-1)F\ \text{if $0\leqslant u\leqslant 1$}, \\
&H_1+(3-2u)H_2\ \text{if $1\leqslant u\leqslant \frac{3}{2}$},
\endaligned
\right.
$$
and
$$
N(u)=\left\{\aligned
&0\ \text{if $0\leqslant u\leqslant 1$}, \\
&(u-1)F\ \text{if $1\leqslant u\leqslant \frac{3}{2}$}.
\endaligned
\right.
$$
This gives
\begin{multline*}
\beta_{Y,\Delta}(\widehat{E})=A_{Y,\Delta}(\widehat{E})-\frac{1}{(-K_Y-\Delta)^3}\int_{0}^{\frac{3}{2}}\big(P(u)\big)^3du=\\
=1-\frac{1}{13}\int_{0}^{1}\Big(2H_1+(3-2u)H_2+(u-1)F\Big)^3du-\frac{1}{13}\int_{1}^{\frac{3}{2}}\Big(2H_1+(3-2u)H_2\Big)^3du=\\
=1-\int_{0}^{1}8u^3-18u^2+13du-\int_{1}^{\frac{3}{2}}12u^2-36u+27du=\frac{7}{26}>0.
\end{multline*}

Suppose that $(Y,\Delta)$ is not K-polystable.
By \cite{Zhuang}, there exists an $\mathrm{Aut}(Y,\Delta)$-invariant prime divisor $\mathbf{F}$ over $Y$ such that $\beta_{Y,\Delta}(\mathbf{F})\leqslant 0$.
Let $Z$ be its center  on $Y$. Then
$$
\delta_{P}(Y,\Delta)\leqslant 1
$$
for every point $P\in Z$. Hence, it follows from Lemmas~\ref{lemma:3-9-S-minus} and \ref{lemma:3-9-P-not-in-B} that $Z\subset B_1\cap B_2$.
Hence, since $Z$ is a $\mathrm{Aut}(Y,\Delta)$-invariant irreducible subvariety, we see that $Z=B_1\cap B_2$.

Let $\widehat{Z}$ be the~center of the~divisor $\mathbf{F}$ on the~threefold $U$.
Then $\widehat{Z}\ne\widehat{E}$, since $\beta(\widehat{E})>0$.
Moreover, since $\widehat{Z}\subset\widehat{E}$ and $\widehat{Z}$ is $\mathrm{Aut}(U)$-invariant,
we see that $\widehat{Z}$ is a $\mathrm{Aut}(U)$-invariant section of the~natural projection $\widehat{E}\to Z$.
Set $A=K_U+\widehat{\Delta}$. Then
$$
0\geqslant\beta_{Y,\Delta}(\mathbf{F})=A_{Y,\Delta}(\mathbf{F})-S_A(\mathbf{F})=A_{U,\widehat{\Delta}}(\mathbf{F})-S_A(\mathbf{F}),
$$
because $K_U+\widehat{\Delta}\sim_{\mathbb{Q}}\psi^*(K_Y+\Delta)$.
Moreover, it follows from \cite{AbbanZhuang,Book,Fujita2021} that
$$
1\geqslant\frac{A_{U,\widehat{\Delta}}(\mathbf{F})}{S_{A}(\mathbf{F})}
\geqslant\min\left\{\frac{1}{S_{A}(\widehat{E})},
\frac{1}{S_A\big(W_{\bullet,\bullet}^{\widehat{E}};\widehat{Z}\big)}\right\},
$$
where $S_A(W_{\bullet,\bullet}^{\widehat{E}};\widehat{Z})$ is defined in \cite[Section~1.7]{Book}.
But $S_{A}(\widehat{E})=\frac{19}{26}$, so $S_A(W_{\bullet,\bullet}^{\widehat{E}};\widehat{Z})\geqslant 1$.

Let us compute $S_A(W_{\bullet,\bullet}^{\widehat{E}};\widehat{Z})$.
Using \cite[Corollary 1.109]{Book}, we see that
$$
S_A\big(W_{\bullet,\bullet}^{\widehat{E}};\widehat{Z}\big)=
\frac{3}{A^3}\int_0^{\frac{3}{2}}\Big(P(u)\big\vert_{\widehat{E}}\Big)^2\mathrm{ord}_{\widehat{Z}}\big(N(u)\big\vert_{\widehat{E}}\big)+\frac{3}{A^3}\int_0^{\frac{3}{2}}\int_{0}^{\infty}\mathrm{vol}\big(P(u)\big\vert_{\widehat{E}}-v\widehat{Z}\big)dvdu,
$$
which is easy to compute, because $\widehat{E}\cong\mathbb{P}^1\times\mathbb{P}^1$. Let us do this.

Let $\mathbf{s}=F\cap \widehat{E}$.
Then $\mathbf{s}$ is a section of the~projection $\widehat{E}\to Z$.
Let $\mathbf{f}$ be a fiber of this projection. Then
$$
P(u)\big\vert_{\widehat{E}}=\left\{\aligned
&(6-4u)\mathbf{f}+u\mathbf{s}\ \text{if $0\leqslant u\leqslant 1$}, \\
&(6-4u)\mathbf{f}+\mathbf{s}\ \text{if $1\leqslant u\leqslant \frac{3}{2}$},
\endaligned
\right.
$$
and
$$
N(u)\big\vert_{\widehat{E}}=\left\{\aligned
&0\ \text{if $0\leqslant u\leqslant 1$}, \\
&(u-1)\mathbf{s}\ \text{if $1\leqslant u\leqslant \frac{3}{2}$}.
\endaligned
\right.
$$
Thus, we see that $S_A(W_{\bullet,\bullet}^{\widehat{E}};\widehat{Z})\leqslant S_A(W_{\bullet,\bullet}^{\widehat{E}};\mathbf{s})$ and
\begin{multline*}
S_A\big(W_{\bullet,\bullet}^{\widehat{E}};\mathbf{s}\big)=\frac{3}{13}\int_1^{\frac{3}{2}}\big((6-4u)\mathbf{f}+\mathbf{s}\big)^2(u-1)du+\frac{3}{13}\int_0^{1}\int_{0}^{u}\big((6-4u)\mathbf{f}+(u-v)\mathbf{s}\big)^2dvdu+\\
+\frac{3}{13}\int_1^{\frac{3}{2}}\int_{0}^{1}\big((6-4u)\mathbf{f}+(1-v)\mathbf{s}\big)^2dvdu=\frac{3}{13}\int_1^{\frac{3}{2}}2(6-4u)(u-1)du+\\
+\frac{3}{13}\int_0^{1}\int_{0}^{u}2(6-4u)(u-v)dvdu+\frac{3}{13}\int_1^{\frac{3}{2}}\int_{0}^{1}2(6-4u)(1-v)dvdu=\frac{5}{13}<1,
\end{multline*}
which is a contradiction.
\end{proof}

In the~remaining part of this sections, we will prove Proposition~\ref{proposition:A1} and \ref{proposition:A2}.

\subsection{Proof of Proposition~\ref{proposition:A1}}
\label{section:A1}

Let us use notations introduced in earlier in this section before Proposition~\ref{proposition:A1},
and let $P$ be an isolated ordinary double point of the~surface $B$.
Then, up to a change of coordinates, we may assume that $P=(0,0,1,0,1)$ and
$$
f_4(x_1,x_2,1)= x_1^2+x_2^2+\text{higher order terms}.
$$

Let $\rho \colon Y_0 \to Y$ be the~blow up at $P$.
Then $Y_0$ is the~toric variety $(\mathbb{C}^6\setminus Z(I_0))/\mathbb{G}_m^3$
for the~torus action given by
$$
M = \left(\begin{array}{cccccc}
x_0 & x_1 & x_2 & x_3 & x_4 & x_5 \\
 0  &  1  &  1  &  1  &  2  &  0  \\
 0  &  0  &  0  &  0  &  1  &  1  \\
 1  &  0  &  0  &  1  &  1  &  0
\end{array}\right)
$$
with irrelevant ideal $I_0=\langle x_1, x_2, x_3 \rangle \cap \langle x_1, x_2, x_4 \rangle \cap \langle x_4, x_5 \rangle \cap \langle x_0, x_3 \rangle \cap \langle x_0, x_5 \rangle$.
To describe its fan, denote the~vector generating the~ray corresponding to $x_i$ by $v_i$.
Then
\begin{align*}
v_0&= (1,1,1),&    v_1&= (1,0,0),& v_2&= (0,1,0),\\
v_3&= (-1,-1,-2),& v_4&= (0,0,1),& v_5&=(0,0,-1).
\end{align*}
The cone structure can be derived from the~irrelevant ideal $I_0$, and it can can be visualized via the~following diagram:
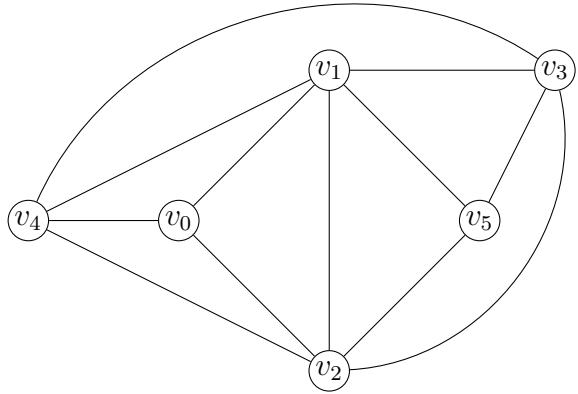
\begin{center}
\begin{tikzpicture}[node distance=1cm,auto]
\node[state, inner sep=1pt,minimum size=0pt] (q_0) {$v_0$};
\node[state, inner sep=1pt,minimum size=0pt] (q_1) at (2,2) {$v_1$};
\node[state,inner sep=1pt,minimum size=0pt] (q_2) at (2,-2) {$v_2$};
\node[state,inner sep=1pt,minimum size=0pt] (q_3) at (5,2) {$v_3$};
\node[state,inner sep=1pt,minimum size=0pt] (q_4) at (-2,0) {$v_4$};
\node[state,inner sep=1pt,minimum size=0pt] (q_5) at (4,0) {$v_5$};

\path[-] (q_0) edge node[swap] {} (q_1);
\path[-] (q_0) edge node[swap] {} (q_2);
\path[-] (q_0) edge node[swap] {} (q_4);
\path[-] (q_1) edge node[swap] {} (q_2);
\path[-] (q_1) edge node[swap] {} (q_3);
\path[-] (q_1) edge node[swap] {} (q_4);
\path[-] (q_1) edge node[swap] {} (q_5);
\path[-] (q_2) edge node[swap] {} (q_4);
\path[-] (q_2) edge node[swap] {} (q_5);
\path[-] (q_3) edge node[swap] {} (q_5);
\path[-, bend right = 50] (q_2) edge node {} (q_3);
\path[-, bend right = 50] (q_3) edge node {} (q_4);
\end{tikzpicture}
\end{center}

Let $F_i=\{x_i=0\}\subset Y_0$, and let $C_{ij} = F_i \cap F_j$ for $i\ne j$ such that $\mathrm{dim}(F_i \cap F_j)=1$.
Consider the~$\mathbb{Z}^3$-grading  of $\mathrm{Pic}(Y_0)$ given by $M$.
If $D_1$ and $D_2$ are two divisors in $\mathrm{Pic}(Y_0)$,
then it follows from  \cite[Chapter~5]{CLS} that
$$
D_1\sim D_2 \iff \deg_M(D_1)=\deg_M(D_2).
$$
Moreover, we have
$$
\overline{\mathrm{Eff}(Y_0)} = \langle F_0, F_1, F_5 \rangle
$$
and
$$
\overline{\mathrm{NE}(Y_0)}=\langle C_{12}, C_{15}, C_{01} \rangle.
$$
In particular, a  divisor $D$ with $\deg_M(D)=(a,b,c)$ is effective $\iff$ all $a,b,c\geqslant 0$.

\begin{lemma}
\label{lemma:Y0-intersections}
Intersections of divisors $F_0$, $F_1$, $F_5$ are given the~following table:
\begin{center}
\renewcommand\arraystretch{1.6}
\begin{tabular}{|c|c|c|c|c|c|c|c|c|c|}
\hline
$F_0^3$&$F_0^2F_1$&$F_0^2F_5$&$F_0F_1^2$&$F_0F_1F_5$&$F_0F_5^2$&$F_1^3$&$F_1^2F_5$&$F_1F_5^2$&$F_5^3$\\
\hline
  $1$  &   $-1$   &   $0$    &    $1$   &    $0$    &    $0$   &  $-1$ &   $1$    &   $-2$   & $4$  \\
\hline
\end{tabular}
\end{center}
\end{lemma}

\begin{proof}
Recall that for distinct torus-invariant divisors $F_i, F_j, F_k$ we may compute their intersection using the~fan and the~cone structure (or the~irrelevant ideal)
$$
F_i F_j F_k = \begin{cases}
0                                             & x_i x_j x_k \in I_0 \\
\frac{1}{\big|\det\{v_i, v_j, v_k\}\big|} & \text{otherwise}.
\end{cases}
$$
This fact together with the~linear equivalences implies the~required assertion.
\end{proof}

Using Lemma \ref{lemma:Y0-intersections}, we obtain the~following intersection table:
\begin{center}
\renewcommand\arraystretch{1.6}
\begin{tabular}{|c||c|c|c| }
\hline
$\bullet$         &      $F_0$    &    $F_1$ 	 & $F_5$ \\
\hline
\hline
$C_{12}$ &        $1$    &    $-1$	  & $1$ \\
\hline
$C_{15}$ &       $0$     &    $1$	  & $-2$  \\
\hline
$C_{01}$ &       $-1$    &     $1$    &  $0$  \\
\hline
\end{tabular}
\end{center}

Now, we set $A=-(K_Y+\Delta)$. Take $u\in\mathbb{R}_{\geqslant 0}$. Set
$$
L(u)=\rho^*(A)-uF_0.
$$
Then $L(u)\sim_{\mathbb{R}}(3-u)F_0+3F_1+F_5$. So, the divisor $L(u)$ is pseudo-effective $\iff$ $u\leqslant 3$.
Let us find a~Zariski decomposition of the~divisor $L(u)$ for $u\in[0,3]$.

The divisor $L(u)$ is nef for $u\in[0,1]$.
We have $L(1) \cdot C_{12} = 0$.
Since $C_{12}$ is a flopping curve, we have to consider a small $\mathbb{Q}$-factorial modification $Y_0\dasharrow Y_1$
such that
$$
Y_1=\big(\mathbb{C}^6 \setminus Z(I_1)\big)/\mathbb{G}_m^3,
$$
where the~torus-action is the~same (given by the~matrix $M$) and the~irrelevant ideal
$$
I_1=\langle x_1, x_2 \rangle \cap \langle x_4, x_5 \rangle \cap \langle x_0, x_3 \rangle,
$$
which is obtained from $I_0$ by replacing $\langle x_0, x_5 \rangle$ with $\langle x_1, x_2 \rangle$.
The fan of $Y_1$ is generated by the~same vectors, but the~cone structure is different:
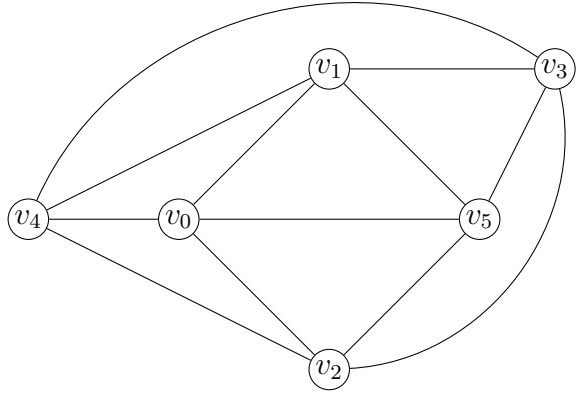
\begin{center}
\begin{tikzpicture}[node distance=1cm,auto]
\node[state, inner sep=1pt,minimum size=0pt] (q_0) {$v_0$};
\node[state, inner sep=1pt,minimum size=0pt] (q_1) at (2,2) {$v_1$};
\node[state,inner sep=1pt,minimum size=0pt] (q_2) at (2,-2) {$v_2$};
\node[state,inner sep=1pt,minimum size=0pt] (q_3) at (5,2) {$v_3$};
\node[state,inner sep=1pt,minimum size=0pt] (q_4) at (-2,0) {$v_4$};
\node[state,inner sep=1pt,minimum size=0pt] (q_5) at (4,0) {$v_5$};

\path[-] (q_0) edge node[swap] {} (q_1);
\path[-] (q_0) edge node[swap] {} (q_2);
\path[-] (q_0) edge node[swap] {} (q_4);
\path[-] (q_0) edge node[swap] {} (q_5); %added to obtain X_1
%\path[-] (q_1) edge node[swap] {} (q_2); removed from X_0
\path[-] (q_1) edge node[swap] {} (q_3);
\path[-] (q_1) edge node[swap] {} (q_4);
\path[-] (q_1) edge node[swap] {} (q_5);
\path[-] (q_2) edge node[swap] {} (q_4);
\path[-] (q_2) edge node[swap] {} (q_5);
\path[-] (q_3) edge node[swap] {} (q_5);
\path[-, bend right = 50] (q_2) edge node {} (q_3);
\path[-, bend right = 50] (q_3) edge node {} (q_4);
\end{tikzpicture}
\end{center}

Abusing our previous notations, we denote the~divisor $\{x_i = 0\}\subset Y_1$ also by $F_i$,
and we let $C_{ij} = F_i \cap F_j$ for $i\ne j$ such that $F_i \cap F_j$ is a curve.
As above, we see that
$$
\overline{\mathrm{NE}(Y_1)} = \langle C_{01}, C_{15}, C_{05} \rangle.
$$
Moreover, intersections of divisors on $Y_1$ are described in the~following table:
\begin{center}
\renewcommand\arraystretch{1.6}
\begin{tabular}{|c|c|c|c|c|c|c|c|c|c|}
\hline
$F_0^3$&$F_0^2F_1$&$F_0^2F_5$&$F_0F_1^2$&$F_0F_1F_5$&$F_0F_5^2$&$F_1^3$&$F_1^2F_5$&$F_1F_5^2$&$F_5^3$\\
\hline
  $0$  &    $0$   &   $-1$   &    $0$   &    $1$    &   $-1$   &  $0$  &   $0$    &   $-1$   & $3$   \\
\hline
\end{tabular}
\end{center}
Using these intersections, we obtain the~following intersection table:
\begin{center}
\renewcommand\arraystretch{1.6}
\begin{tabular}{|c||c|c|c| }
\hline
$\bullet$         & $F_0$ & $F_1$ & $F_5$ \\
\hline
\hline
$C_{05}$ & $-1$  &  $1$  & $-1$   \\
\hline
$C_{15}$ &  $1$  &  $0$  & $-1$    \\
\hline
$C_{01}$ &  $0$  &  $0$  &  $1$    \\
\hline
\end{tabular}
\end{center}

The~proper transform on $Y_1$ of the~divisor $L(u)$ is nef for $u\in [1,2]$,
and it intersects the~curve $C_{15}$ trivially for $u=2$.
Note that $C_{15} \sim C_{25}$ on the~surface $F_5$, which implies that the~divisor $F_5$ is contained in the~negative part of the~Zariski decomposition of the~proper transform of the~divisor $L(u)$.
In fact, we have $N(u)=(u-2)F_5$ and
$$
P(u) = (3-u) (F_0+F_5)+3 F_1,
$$
where $N(u)$ is the~negative part of the~decomposition, and $P(u)$ is the~positive part.

\begin{lemma}
\label{lemma:A1-S-F0}
One has $A_{Y,\Delta}(F_0)=2$ and $S_{A}(F_0)=\frac{49}{26}$, so that
$$
\frac{A_{Y,\Delta}(F_0)}{S_{A}(F_0)}=\frac{52}{49}.
$$
\end{lemma}

\begin{proof}
The equality $A_{Y,\Delta}(F_0)=2$ is obvious. Moreover, we have
$$
\mathrm{vol} (L(u))=\begin{cases}
-u^3+13 	   & u\in [0,1]\\
-3u^2+3u+12    & u\in [1,2]\\
3u^3-18u^2+27u & u\in [2,3].
\end{cases}
$$
Thus, we compute
$$
S_{A}(F_0) = \frac{1}{A^3} \int_0^3 \mathrm{vol} (L(u)) du = \frac{49}{26}
$$
as claimed.
\end{proof}

Now, we construct a common toric resolution $\widetilde{Y}$ for $Y_0$ and $Y_1$.
Such variety is easy to see from the~fans of $Y_0$ and $Y_1$, we want to add the~following ray:
$$
v_6 = (1,1,0) \in \langle v_1, v_2 \rangle \cap \langle v_0, v_5 \rangle,
$$
Set $\widetilde{Y}$ to be the~toric variety corresponding to $v_0,\dots,v_6$ with the~following cone structure:

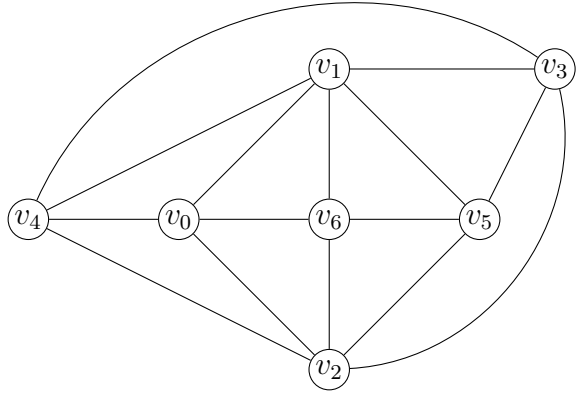
\begin{center}
\begin{tikzpicture}[node distance=1cm,auto]
\node[state, inner sep=1pt,minimum size=0pt] (q_0) {$v_0$};
\node[state, inner sep=1pt,minimum size=0pt] (q_1) at (2,2) {$v_1$};
\node[state,inner sep=1pt,minimum size=0pt] (q_2) at (2,-2) {$v_2$};
\node[state,inner sep=1pt,minimum size=0pt] (q_3) at (5,2) {$v_3$};
\node[state,inner sep=1pt,minimum size=0pt] (q_4) at (-2,0) {$v_4$};
\node[state,inner sep=1pt,minimum size=0pt] (q_5) at (4,0) {$v_5$};
\node[state,inner sep=1pt,minimum size=0pt] (q_6) at (2,0) {$v_6$};

\path[-] (q_0) edge node[swap] {} (q_1);
\path[-] (q_0) edge node[swap] {} (q_2);
\path[-] (q_0) edge node[swap] {} (q_4);
\path[-] (q_0) edge node[swap] {} (q_6);
\path[-] (q_1) edge node[swap] {} (q_3);
\path[-] (q_1) edge node[swap] {} (q_4);
\path[-] (q_1) edge node[swap] {} (q_5);
\path[-] (q_1) edge node[swap] {} (q_6);
\path[-] (q_2) edge node[swap] {} (q_4);
\path[-] (q_2) edge node[swap] {} (q_5);
\path[-] (q_2) edge node[swap] {} (q_6);
\path[-] (q_3) edge node[swap] {} (q_5);
\path[-] (q_5) edge node[swap] {} (q_6);
\path[-, bend right = 50] (q_2) edge node {} (q_3);
\path[-, bend right = 50] (q_3) edge node {} (q_4);
\end{tikzpicture}
\end{center}

Let $\varphi_0 \colon \widetilde{Y} \to Y_0$ and $\varphi_1 \colon \widetilde{Y} \to Y_1$ be the~corresponding toric birational maps.
Then
\begin{itemize}
\item $\varphi_0$ is the~blow up of $Y_0$ along the~curve $C_{12}$,
\item $\varphi_1$ is the~blow up of $Y_1$ along the~curve $C_{05}$.
\end{itemize}
Set $\widetilde{F}_i=\{x_i = 0\}\subset \widetilde{Y}$.
Then $\widetilde{F}_6$ is the~exceptional divisor of $\varphi_0$ and $\varphi_1$.

The Zariski decomposition of the~divisor $\varphi_0^*(L(u))$ can be described as follows:
$$
\widetilde{P}(u)\sim_{\mathbb{R}}\begin{cases}
(3-u)\widetilde{F}_0+3\widetilde{F}_1+\widetilde{F}_5+3\widetilde{F}_6     & u\in[0,1],\\
(3-u)\widetilde{F}_0+3\widetilde{F}_1+\widetilde{F}_5+(4-u)\widetilde{F}_6 & u\in[1,2],\\
(3-u)(\widetilde{F}_0+\widetilde{F}_5)+3\widetilde{F}_1+(6-2u)\widetilde{F}_6  & u\in[2,3],
\end{cases}
$$
and
$$
\widetilde{N}(u)=
\begin{cases}
0                                & u\in[0,1],\\
(u-1) \widetilde{F}_6                   & u\in[1,2],\\
(u-2) \widetilde{F}_5+(2u-3) \widetilde{F}_6 & u\in[2,3],
\end{cases}
$$
where $\widetilde{P}(u)$ is the~positive part, and $\widetilde{N}(u)$ is the~negative part. Note that

Let $\sigma\colon\widetilde{F}_0\to F_0$ be the~morphism induced by $\phi_0$.
Then $\sigma$ is a blow up at one point. So, we have $\widetilde{F}_0 \cong \mathbb{F}_1$.
Let $\mathbf{e}$ be the~$\sigma$-exceptional curve,
and let $\mathbf{f}$ be a fiber of the~natural projection $\widetilde{F}_0\to \mathbb{P}^1$.
Then $\widetilde{F}_0\vert_{\widetilde{F}_0} \sim -\mathbf{e}-\mathbf{f}$, $\widetilde{F}_1\vert_{\widetilde{F}_0} \sim \mathbf{f}$,
$\widetilde{F}_5\vert_{\widetilde{F}_0}\sim 0$, $\widetilde{F}_6\vert_{\widetilde{F}_0}=\mathbf{e}$,
which gives
$$
\widetilde{P}(u)\big\vert_{\widetilde{F}_0} = \begin{cases}
u(\mathbf{f}+\mathbf{e})  & u\in[0,1],\\
u\mathbf{f}+\mathbf{e}    & u\in[1,2],\\
u\mathbf{f}+(3-u)\mathbf{e} & u\in[2,3],
\end{cases}
$$
and
$$
\widetilde{N}(u)\big\vert_{\widetilde{F}_0} = \begin{cases}
0       & u\in[0,1],\\
(u-1)\mathbf{e}  & u\in[1,2],\\
(2u-3)\mathbf{e} & u\in[2,3].
\end{cases}
$$

We are ready to apply \cite{AbbanZhuang,Book,Fujita2021}.
Set $B_{F_0}=\rho_*^{-1}(B)\vert_{F_0}$ and $\Delta_{F_0}=\frac{1}{2}B_{F_0}$.
Set
$$
\delta\big(F_0,\Delta_{F_0};V^{\widetilde{F}_0}_{\bullet,\bullet}\big)=
\inf_{E/\widetilde{F}_0}\frac{A_{F_0,\Delta_{F_0}}(E)}{S(W^{\widetilde{F}_0}_{\bullet,\bullet};E)}
$$
where the~infimum is taken over all prime divisors $E$ over $\widetilde{F}_0$,
and
$$
S(W^{\widetilde{F}_0}_{\bullet,\bullet};E)=\frac{3}{A^3}\int\limits_{0}^3\big(\widetilde{P}(u)\big\vert_{\widetilde{F}_0}\big)^2\mathrm{ord}_{E}\big(\widetilde{N}(u)\big\vert_{\widetilde{F}_0}\big)du+
\frac{3}{A^3}\int\limits_{0}^3\int\limits_{0}^\infty\mathrm{vol}\big(\widetilde{P}(u)\big\vert_{\widetilde{F}_0}-vE\big)dvdu.
$$
Let $\mathbf{F}$ be a prime divisor over $Y$ such that $P=C_Y(\mathbf{F})$. Recall that
$$
\beta_{Y,\Delta}(\mathbf{F})=A_{Y,\Delta}(\mathbf{F})-S_A(\mathbf{F})=A_{Y,\Delta}(\mathbf{F})-\frac{1}{A^3}\int_{0}^\infty\mathrm{vol}\big(A-u\mathbf{F}\big)du.
$$
It follows from \cite[Theorem~4.8]{Fujita2021} and \cite[Corollary~4.9]{Fujita2021} that
\begin{equation}
\label{equation:A1-non-strict}
\frac{A_{Y,\Delta}(\mathbf{F})}{S_A(\mathbf{F})}\geqslant\delta_P(Y,\Delta)\geqslant
\min\left\{\frac{A_{Y,\Delta}(F_0)}{S_{A}(F_0)},
\delta\big(F_0,\Delta_{F_0};V^{\widetilde{F}_0}_{\bullet,\bullet}\big)\right\}.
\end{equation}

Suppose $\beta_{Y,\Delta}(\mathbf{F})\leqslant 0$.
Then it follows from \eqref{equation:A1-non-strict} and Lemma~\ref{lemma:A1-S-F0} that
there is a~prime divisor $E$ over $\widetilde{F}_0$ such that
\begin{equation}
\label{equation:A1-strict}
S(W^{\widetilde{F}_0}_{\bullet,\bullet};E)\geqslant A_{F_0,\Delta_{F_0}}(E).
\end{equation}
Let $Z$ be the~center of the~divisor $E$ on the~surface $\widetilde{F}_0$. Note that~$\sigma(\mathbf{e})\not\in B_{F_0}$.

\begin{lemma}
\label{lemma:A1-O-in-e}
One has $Z\cap \mathbf{e}=\varnothing$.
\end{lemma}

\begin{proof}
Note that $A_{F_0,\Delta_{F_0}}(\mathbf{e})=2$.
Let us compute $S(W^{\widetilde{F}_0}_{\bullet,\bullet};\mathbf{e})$.
For  $u\in[0,3]$, let
$$
t(u)=\sup\Big\{v\in \mathbb R_{\geqslant 0} \ \big|\ \text{$\widetilde{P}(u)\vert_{\widetilde{F}_0}-v\mathbf{e}$ is pseudoeffective}\Big\}.
$$
For every $v\in[0,t(u)]$, let us denote by $P(u,v)$ and $N(u,v)$ the~positive and the~negative parts of the~Zariski decompositions of
the~divisor $\widetilde{P}(u)\vert_{\widetilde{F}_0}-v\mathbf{e}$, respectively.
Then
$$
S(W^{\widetilde{F}_0}_{\bullet,\bullet};\mathbf{e})=
\frac{3}{A^3}\int\limits_{0}^3\big(P(u,0)\big)^2\mathrm{ord}_{\mathbf{e}}\big(\widetilde{N}(u)\big\vert_{\widetilde{F}_0}\big)du+\frac{3}{A^3}\int\limits_{0}^3\int\limits_{0}^{t(u)}\big(P(u,v)\big)^2dvdu.
$$
Observe that
$$
\mathrm{ord}_{\mathbf{e}}\big(\widetilde{N}(u)\big\vert_{\widetilde{F}_0}\big)=\begin{cases}
0	      & u\in[0,1],\\
u-1 	  & u\in[1,2],\\
2u-3      & u\in[2,3].
\end{cases}
$$
Moreover, we have
$$
t(u) = \begin{cases}
u & u\in[0,1],\\
1 & u\in[1,2],\\
3-u & u\in[2,3].
\end{cases}
$$
Furthermore, we have $N(u,v)=0$ for every $u\in[0,3]$ and $v\in[0,t(u)]$. Finally, we have
$$
P(u,v) = \begin{cases}
uf+(u-v) e & u\in[0,1], v\in[0,u],\\
uf+(1-v) e & u\in[1,2], v\in[0,1],\\
uf+(3-u-v) e & u\in[2,3], v\in[0,3-u],
\end{cases}
$$
which gives
$$
\big(P(u,v)\big)^2= \begin{cases}
u^2 - v^2           & u\in[0,1], v\in[0,u],\\
u^2 - (1-v-u)^2     & u\in[1,2], v\in[0,1],\\
u^2 - (3-2u-v)^2    & u\in[2,3], v\in[0,3-u].
\end{cases}
$$
Integrating, we get $S(W^{\widetilde{F}_0}_{\bullet,\bullet};\mathbf{e})=\frac{20}{13}<2=A_{F_0,\Delta_{F_0}}(\mathbf{e})$,
so that $Z\ne\mathbf{e}$ by \eqref{equation:A1-strict}.

Suppose that $Z\cap \mathbf{e}\ne\varnothing$. Let $O$ be a point of the~intersection $Z\cap \mathbf{e}$.
Then it follows from \cite[Theorem~4.17]{Fujita2021} and \cite[Corollary~4.18]{Fujita2021} that
$$
\frac{A_{F_0,\Delta_{F_0}}(E)}{S(W^{\widetilde{F}_0}_{\bullet,\bullet};E)}
\geqslant
\min\left\{\frac{2}{S(W^{\widetilde{F}_0}_{\bullet,\bullet};\mathbf{e})},
\frac{1}{S(W^{\widetilde{F}_0,\mathbf{e}}_{\bullet,\bullet,\bullet,};O)}\right\}=
\min\left\{\frac{13}{10},\frac{1}{S(W^{\widetilde{F}_0,\mathbf{e}}_{\bullet,\bullet,\bullet};O)}\right\},
$$
where
$$
S\big(W^{\widetilde{F}_0,\mathbf{e}}_{\bullet,\bullet,\bullet};O\big)=\frac{3}{A^3}\int\limits_0^3\int\limits_0^{t(u)}\big(P(u,v)\cdot\mathbf{e}\big)^2dvdu.
$$
Integrating, we get $S(W^{\widetilde{F}_0,\mathbf{e}}_{\bullet,\bullet,\bullet};O)=\frac{20}{13}$, which contradicts \eqref{equation:A1-strict}.
\end{proof}

Thus, we see that $Z$ is disjoint from $\mathbf{e}$.
In particular, we see that
$$
Z\cap\mathrm{Supp}\big(\widetilde{N}(u)\big\vert_{\widetilde{F}_0}\big)=\varnothing
$$
for every $u\in[0,3]$. This will simplify some formulas in the~following.

Let $B_{\widetilde{F}_0}$ be the~strict transform on $\widetilde{F}_0$ of the~curve $B_{F_0}$.
Then $B_{\widetilde{F}_0}$ is a smooth irreducible curve in $|2(\mathbf{e}+\mathbf{f})|$.
Set~$\Delta_{\widetilde{F}_0}=\frac{1}{2}B_{\widetilde{F}_0}$.
Let $O$ be a point in $Z$. We may assume that $O\in\mathbf{f}$.
Then there are three cases to consider:
\begin{itemize}
\item[$(\mathrm{1})$] $O\not\in B_{\widetilde{F}_0}$,
\item[$(\mathrm{2})$] $O\in B_{\widetilde{F}_0}\cap\mathbf{f}$, and $\mathbf{f}$  intersects $B_{\widetilde{F}_0}$ transversely at the~point $O$,
\item[$(\mathrm{3})$] $O=B_{\widetilde{F}_0}\cap\mathbf{f}$, and $\mathbf{f}$ is tangent to $B_{\widetilde{F}_0}$ at  the~point $O$.
\end{itemize}
Let $\theta\colon\widehat{F}_0\rightarrow \widetilde{F}_0$ be a plt blow up  of the~point $O$ defined as follows:
\begin{itemize}
\item the~map $\theta$ is an ordinary blow up in the~case when $O\not\in B_{\widetilde{F}_0}$, or when $O\in B_{\widetilde{F}_0}\cap\mathbf{f}$, and the~fiber $\mathbf{f}$  intersects the~curve $B_{\widetilde{F}_0}$ transversely at the~point $O$,
\item the~map $\theta$ is a weighted blow up at the~point $O=B_{\widetilde{F}_0}\cap\mathbf{f}$ with weights $(1,2)$ such that the~proper transforms on $\widehat{F}_0$ of  the~curves $B_{\widetilde{F}_0}$ and $\mathbf{f}$ are disjoint in the case when the~fiber $\mathbf{f}$ is tangent to the~curve $B_{\widetilde{F}_0}$ at  the~point $O$.
\end{itemize}
Let $C$ be the~$\theta$-exceptional curve.
We have $C\cong\mathbb{P}^1$.
Let $B_{\widehat{F}_0}$ be the~proper transform on the~surface $\widehat{F}_0$ of the~curve $B_{\widehat{F}_0}$.
Set $\Delta_{\widehat{F}_0}=\frac{1}{2}B_{\widehat{F}_0}$.
Let $\Delta_C$ be the~effective $\mathbb{Q}$-divisor on the~curve $C$ known as the~different, which can be defined via the~adjunction formula:
$$
K_C+\Delta_C=\big(K_{\widehat{F}_0}+\Delta_{\widehat{F}_0}\big)\big\vert_{C}.
$$
If $\theta$ is a~usual blow up, then $\Delta_C=\Delta_{\widehat{F}_0}\vert_{C}$.
Similarly, if $\theta$ is a weighted blow up, then
$$
\Delta_C=\Delta_{\widehat{F}_0}\big\vert_{C}+\frac{1}{2}\mathbf{o},
$$
where $\mathbf{o}$ is the~singular point of the~surface $\widehat{F}_0$ contained in $C$ --- $\mathbf{o}$ is an ordinary double point,
which is not contained in the~proper transforms of the~curves $B_{\widetilde{F}_0}$ and $\mathbf{f}$.

Now, for $u\in[0,3]$, we let
$$
\hat{t}(u)=\sup\Big\{v\in \mathbb R_{\geqslant 0} \ \big|\ \text{$\theta^*\big(\widetilde{P}(u)\vert_{\widetilde{F}_0}\big)-vC$ is pseudoeffective}\Big\}.
$$
For every $v\in[0,\hat{t}(u)]$, let us denote by $\widehat{P}(u,v)$ and $\widehat{N}(u,v)$ the~positive and the~negative parts of the~Zariski decompositions of
the~divisor $\theta^*(\widetilde{P}(u)\vert_{\widetilde{F}_0})-vC$, respectively.
Then
\begin{equation}
\label{equation:A1-blow-up}
1\geqslant\frac{A_{F_0,\Delta_{F_0}}(E)}{S(W^{\widetilde{F}_0}_{\bullet,\bullet};E)}
\geqslant
\min\left\{\frac{A_{F_0,\Delta_{F_0}}(C)}{S(W^{\widetilde{F}_0}_{\bullet,\bullet};C)},
\inf_{Q\in C}\frac{A_{C,\Delta_{C}}(Q)}{S\big(W^{\widehat{F}_0,C}_{\bullet,\bullet,\bullet};Q\big)}\right\}
\end{equation}
by \eqref{equation:A1-strict} and \cite[Corollary~4.18]{Fujita2021}, where the~infimum is taken by all points $Q\in C$, and
$$
S\big(W^{\widehat{F}_0,C}_{\bullet,\bullet,\bullet,};Q\big)=\frac{3}{A^3}\int\limits_0^3\int\limits_0^{\hat{t}(u)}\big(\widehat{P}(u,v)\cdot C\big)^2dvdu+F_Q\big(W^{\widehat{F}_0,C}_{\bullet,\bullet,\bullet}\big)
$$
for
$$
F_Q\big(W^{\widehat{F}_0,C}_{\bullet,\bullet,\bullet}\big)=\frac{6}{A^3}\int\limits_0^3\int\limits_0^{\hat{t}(u)}\big(\widehat{P}(u,v)\cdot C\big)\mathrm{ord}_{Q}\big(\widehat{N}(u,v)\big\vert_{C}\big)dvdu.
$$
Denote by $\widehat{\mathbf{e}}$ and $\widehat{\mathbf{f}}$ the~proper transforms of the curves $\mathbf{e}$ and $\mathbf{f}$, respectively.

\begin{lemma}
\label{lemma:A1-f-B-transversal}
Suppose that $\theta$ is an ordinary blow up. Let $Q$ be a point in $C$.
Then
$$
\frac{A_{F_0,\Delta_{F_0}}(C)}{S(W^{\widetilde{F}_0}_{\bullet,\bullet};C)}\geqslant\frac{39}{29}
$$
and
$$
\frac{A_{C,\Delta_{C}}(Q)}{S\big(W^{\widehat{F}_0,C}_{\bullet,\bullet,\bullet};Q\big)}\geqslant \frac{13}{10}.
$$
\end{lemma}

\begin{proof}
One has
$$
\theta^*\big(\widetilde{P}(u)\vert_{\widetilde{F}_0}\big)\sim_{\mathbb{R}}
\begin{cases}
u(\widehat{\mathbf{f}}+\widehat{\mathbf{e}}+C)     & u\in[0,1],\\
u(\widehat{\mathbf{f}}+C)+\widehat{\mathbf{e}}     & u\in[1,2],\\
u(\widehat{\mathbf{f}}+C)+(3-u)\widehat{\mathbf{e}} & u\in[2,3].
\end{cases}
$$
This easily implies that $\hat{t}(u)=u$ and
$$
\widehat{N}(u,v) = \begin{cases}
0 & u\in[0,1], v\in[0,u],\\
0 & u\in[1,2], v\in[0,1],\\
(v-1)\widehat{\mathbf{f}} & u\in[1,2], v\in[1,u],\\
0  & u\in[2,3], v\in[0,3-u],\\
(v+u-3) \widehat{\mathbf{f}} & u\in[2,3], v\in[3-u,u],
\end{cases}
$$
so that
$$
\widehat{P}(u,v) = \begin{cases}
u(\widehat{\mathbf{f}}+\widehat{\mathbf{e}})+(u-v)C              & u\in[0,1], v\in[0,u],\\
u \widehat{\mathbf{f}}+(u-v)C+\widehat{\mathbf{e}}             & u\in[1,2], v\in[0,1],\\
(u-v+1) \widehat{\mathbf{f}}+(u-v)C+\widehat{\mathbf{e}}       & u\in[1,2], v\in[1,u],\\
u \widehat{\mathbf{f}}+(u-v)C+\widehat{\mathbf{e}}             & u\in[2,3], v\in[0,3-u],\\
(3-v) \widehat{\mathbf{f}}+(u-v)C+(3-u)\widehat{\mathbf{e}}   & u\in[2,3], v\in[3-u,u],
\end{cases}
$$
which gives
$$
\big(\widehat{P}(u,v)\big)^2 = \begin{cases}
u^2 - v^2               & u\in[0,1], v\in[0,u],\\
-v^2+2u-1               & u\in[1,2], v\in[0,1],\\
2u-2v                   & u\in[1,2], v\in[1,u],\\
-3u^2-v^2+12u-9         & u\in[2,3], v\in[0,3-u],\\
-2u^2+2uv+6u-6v         & u\in[2,3], v\in[3-u,u].
\end{cases}
$$
Thus, integrating, we get $S(W^{\widetilde{F}_0}_{\bullet,\bullet};C)=\frac{29}{26}$. Note that
$$
A_{F_0,\Delta_{F_0}}(C)= \begin{cases}
\frac{3}{2}   &   O \in B_{\widetilde{F}_0},\\
2             &   O \not\in B_{\widetilde{F}_0}.
\end{cases}
$$
This gives the first required inequality. Similarly, we compute
$$
S\big(W^{\widehat{F}_0,C}_{\bullet,\bullet,\bullet};Q\big) = \frac{9}{26}+F_Q\big(W^{\widehat{F}_0,C}_{\bullet,\bullet,\bullet}\big)
$$
where
$$
F_Q\big(W^{\widehat{F}_0,C}_{\bullet,\bullet,\bullet}\big)=\begin{cases}
    \frac{11}{26}   & Q = \widehat{\mathbf{f}}\cap C,\\
    0               & \text{otherwise}.
\end{cases}
$$
Observe that
$$
A_{C,\Delta_C}(Q) = \begin{cases}
\frac{1}{2} & Q\in B_{\widehat{F}_0},\\
1           & Q \not\in B_{\widehat{F}_0}.
\end{cases}
$$
Moreover, if $O\in B_{\widetilde{F}_0}\cap\mathbf{f}$, the intersection $C\cap \widehat{\mathbf{f}}$ consists of a single point,
which is not contained in $B_{\widehat{F}_0}$.
Thus, we have
$$
\frac{A_{C,\Delta_{C}}(Q)}{S\big(W^{\widehat{F}_0,C}_{\bullet,\bullet,\bullet};Q\big)}=\begin{cases}
        \frac{13}{10}   & Q=C\cap\widehat{\mathbf{f}},\\
        \frac{13}{9}    & Q=C\cap B_{\widehat{F}_0},\\
        \frac{26}{9}    & \text{otherwise}.
\end{cases}
$$
which implies the second required inequality.
\end{proof}

Thus, it follows from \eqref{equation:A1-blow-up} and Lemma~\ref{lemma:A1-f-B-transversal}
that $O=B_{\widetilde{F}_0}\cap\mathbf{f}$, so $\mathbf{f}$ and $B_{\widetilde{F}_0}$ are tangent at the~point $O$.
Then $\theta$ is a weighted blow up with weights $(1,2)$. We have
$$
\theta^*\big(\widetilde{P}(u)\big\vert_{\widetilde{F}_0}\big)\sim_{\mathbb{R}}\begin{cases}
u (\widehat{\mathbf{f}}+\widehat{\mathbf{e}}+2C)     & u\in[0,1],\\
u (\widehat{\mathbf{f}}+2C)+\widehat{\mathbf{e}}       & u\in[1,2],\\
u (\widehat{\mathbf{f}}+2C)+(3-u) \widehat{\mathbf{e}} & u\in[2,3].
\end{cases}
$$
This gives $\hat{t}(u)=2u$. Moreover, we have
$$
\widehat{N}(u,v) = \begin{cases}
       0                         & u\in[0,1], v\in[0,u],\\
(v-u)(\widehat{\mathbf{f}}+\widehat{\mathbf{e}})           & u\in[0,1], v\in[u,2u],\\
        0                        & u\in[1,2], v\in[0,1],\\
\frac{v-1}{2} \widehat{\mathbf{f}}            & u\in[1,2], v\in[1,2u-1],\\
(v-u) \widehat{\mathbf{f}}+(v-2u+1)\widehat{\mathbf{e}}    & u\in[1,2], v\in[1,2u-1],\\
        0                        & u\in[2,3], v\in[0,3-u],\\
\frac{v+u-3}{2}\widehat{\mathbf{f}}           & u\in[2,3], v\in[0,3u-3],\\
(v-u)\widehat{\mathbf{f}}+(v+3-3u)\widehat{\mathbf{e}}   & u\in[2,3], v\in[3u-3,2u],
\end{cases}
$$
and
$$
\widehat{P}(u,v) = \begin{cases}
(2u-v)C +u \widehat{\mathbf{f}}+u\widehat{\mathbf{e}}                    & u\in[0,1], v\in[0,u],\\
(2u-v)(C+\widehat{\mathbf{f}}+\widehat{\mathbf{e}})                        & u\in[0,1], v\in[u,2u],\\
(2u-v)C+u\widehat{\mathbf{f}}+\widehat{\mathbf{e}}                     & u\in[1,2], v\in[0,1],\\
(2u-v)C+\frac{2u-v+1}{2}\widehat{\mathbf{f}}+\widehat{\mathbf{e}}            & u\in[1,2], v\in[1,2u-1],\\
(2u-v)(C+\widehat{\mathbf{f}}+\widehat{\mathbf{e}})                        & u\in[1,2], v\in[1,2u-1],\\
        (2u-v)C+u\widehat{\mathbf{f}}+(3-u)\widehat{\mathbf{e}}        & u\in[2,3], v\in[0,3-u],\\
(2u-v)C+\frac{u-v+3}{2}\widehat{\mathbf{f}}+(3-u)\widehat{\mathbf{e}}  & u\in[2,3], v\in[0,3u-3],\\
(2u-v)(C+\widehat{\mathbf{f}}+\widehat{\mathbf{e}})                        & u\in[2,3], v\in[3u-3,2u].
\end{cases}
$$
Then
$$
\big(\widehat{P}(u,v)\big)^2 = \begin{cases}
u^2-\frac{v^2}{2}          & u\in[0,1], v\in[0,u],\\
\frac{(2u-v)^2}{2}         & u\in[0,1], v\in[u,2u],\\
2u-1-\frac{v^2}{2}         & u\in[1,2], v\in[0,1],\\
2u-v-\frac{1}{2}           & u\in[1,2], v\in[1,2u-1],\\
\frac{(2u-v)^2}{2}         & u\in[1,2], v\in[1,2u-1],\\
12u-9-3u^2-\frac{v^2}{2}   & u\in[2,3], v\in[0,3-u],\\
\frac{(5u-2v-3)(u-3)}{2}   & u\in[2,3], v\in[0,3u-3],\\
\frac{(2u-v)^2}{2}         & u\in[2,3], v\in[3u-3,2u].
\end{cases}
$$
Now, integrating, we get $S(W^{\widetilde{F}_0}_{\bullet,\bullet};C)=\frac{49}{26}$.
Thus, since $A_{F_0,\Delta_{F_0}}(C)=2$, we get
$$
\frac{A_{F_0,\Delta_{F_0}}(C)}{S(W^{\widetilde{F}_0}_{\bullet,\bullet};C)}=\frac{52}{49},
$$
so it follows from \eqref{equation:A1-blow-up} that there is a point $Q\in C$ such that $S(W^{\widehat{F}_0,C}_{\bullet,\bullet,\bullet};Q)\geqslant A_{C,\Delta_{C}}(Q)$.
On the other hand, we compute
$$
S\big(W^{\widehat{F}_0,C}_{\bullet,\bullet,\bullet};Q\big) = \frac{9}{52}+F_Q\big(W^{\widehat{F}_0,C}_{\bullet,\bullet,\bullet}\big)
$$
where
$$
F_Q\big(W^{\widehat{F}_0,C}_{\bullet,\bullet,\bullet}\big)=\begin{cases}
    \frac{3}{4}   & Q=C\cap\widehat{\mathbf{f}},\\
    0             & \text{otherwise}.
\end{cases}
$$
Recall that $B_{\widehat{F}_0}$ and $\widehat{\mathbf{f}}$ are disjoint and do not contain the singular point of the~surface $\widehat{F}_0$.
Moreover, we have
$$
A_{C,\Delta_C}(Q) = \begin{cases}
\frac{1}{2} & Q=C\cap B_{\widehat{F}_0},\\
\frac{1}{2} & Q=\mathrm{Sing}(\widehat{F}_0),\\
1           & \text{otherwise}.
\end{cases}
$$
Thus, summarizing, we get
$$
\frac{A_{C,\Delta_{C}}(Q)}{S\big(W^{\widehat{F}_0,C}_{\bullet,\bullet,\bullet};Q\big)}=
\begin{cases}
        \frac{13}{12}   & Q = C\cap\widehat{\mathbf{f}},\\
        \frac{26}{9}    & Q=C\cap B_{\widehat{F}_0},\\
        \frac{26}{9}    & Q=\mathrm{Sing}(\widehat{F}_0),\\
        \frac{52}{9}    & \text{otherwise}.
\end{cases}
$$
In particular, we see that $S(W^{\widehat{F}_0,C}_{\bullet,\bullet,\bullet};Q)<A_{C,\Delta_{C}}(Q)$ in every possible case.
The obtained contradiction completes the proof of Proposition~\ref{proposition:A1}.

\subsection{Proof of Proposition~\ref{proposition:A2}}
\label{section:A2}

Let us use notations introduced earlier in this section before Proposition~\ref{proposition:A2},
and let $P$ be a singular point of type $\mathbb{A}_2$ of the surface $B\in |2S^+|$.
Then, up to a change of coordinates, we may assume that $P=(0,0,1,0,1)$ and
$$
f_4(x_1,x_2,1)=x_1^2+x_2^3+\text{higher order terms}.
$$

Let $\rho\colon Y_0\to Y$ be the~blow up if the point $P$ with weights $(3,2,3)$ with respect to variables $(x_1,x_2,x_4)$.
We may describe $Y_0$ as a toric variety given as $(\mathbb{C}^6\setminus Z(I_0))/\mathbb{G}_m^3$, where the~action is given by the~matrix
$$
M = \left(\begin{array}{cccccc}
x_0 & x_1 & x_2 & x_3 & x_4 & x_5 \\
 0  &  1  &  1  &  1  &  2  &  0  \\
 0  &  0  &  0  &  0  &  1  &  1  \\
 1  &  0  &  1  &  3  &  3  &  0
\end{array}\right),
$$
where the~irrelevant ideal $I_0=\langle x_1, x_2, x_3 \rangle \cap \langle x_1, x_2, x_4 \rangle \cap \langle x_4, x_5 \rangle \cap \langle x_0, x_3 \rangle \cap \langle x_0, x_5 \rangle$.
To~describe the~fan of the toric threefold $Y_0$,
we denote by $v_i$ the~vector generating the~ray corresponding to $x_i$.
Then
\begin{align*}
v_0&=(3,2,3), & v_1&=(1,0,0), & v_2&= (0,1,0),\\
v_3&=(-1,-1,-2), & v_4&=(0,0,1), & v_5&=(0,0,-1),
\end{align*}
and the cone structure can be visualized with the~following diagram:
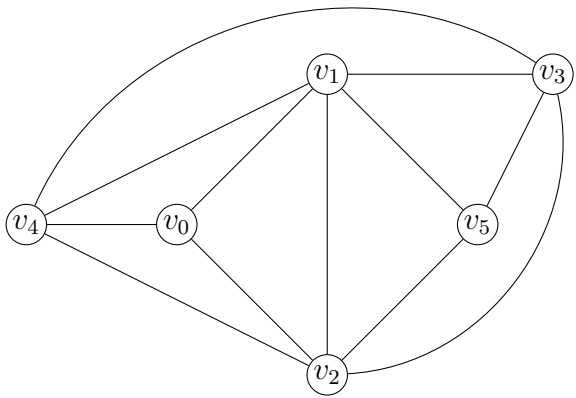
\begin{center}
\begin{tikzpicture}[node distance=1cm,auto]
\node[state, inner sep=1pt,minimum size=0pt] (q_0) {$v_0$};
\node[state, inner sep=1pt,minimum size=0pt] (q_1) at (2,2) {$v_1$};
\node[state,inner sep=1pt,minimum size=0pt] (q_2) at (2,-2) {$v_2$};
\node[state,inner sep=1pt,minimum size=0pt] (q_3) at (5,2) {$v_3$};
\node[state,inner sep=1pt,minimum size=0pt] (q_4) at (-2,0) {$v_4$};
\node[state,inner sep=1pt,minimum size=0pt] (q_5) at (4,0) {$v_5$};

\path[-] (q_0) edge node[swap] {} (q_1);
\path[-] (q_0) edge node[swap] {} (q_2);
\path[-] (q_0) edge node[swap] {} (q_4);
\path[-] (q_1) edge node[swap] {} (q_2);
\path[-] (q_1) edge node[swap] {} (q_3);
\path[-] (q_1) edge node[swap] {} (q_4);
\path[-] (q_1) edge node[swap] {} (q_5);
\path[-] (q_2) edge node[swap] {} (q_4);
\path[-] (q_2) edge node[swap] {} (q_5);
\path[-] (q_3) edge node[swap] {} (q_5);
\path[-, bend right = 50] (q_2) edge node {} (q_3);
\path[-, bend right = 50] (q_3) edge node {} (q_4);
\end{tikzpicture}
\end{center}

Let $F_i=\{x_i=0\}\subset Y_0$ and $C_{ij} = F_i \cap F_j$ for $i\ne j$ such that $\mathrm{dim}(F_i \cap F_j)=1$.
Then
$$
\overline{\mathrm{Eff}(Y_0)}=\langle F_0, F_1, F_5\rangle
$$
and
$$
\overline{\mathrm{NE}(Y_0)}= \langle C_{12}, C_{15}, C_{01} \rangle.
$$
Intersections of divisors $F_0$, $F_1$, $F_5$ are described in following table:
\begin{center}
\renewcommand\arraystretch{1.6}
\begin{tabular}{|c|c|c|c|c|c|c|c|c|c|c|}
\hline
$F_0^3$      & $F_0^2F_1$   &  $F_0^2F_5$ & $F_0F_1^2$    & $F_0F_1F_5$ & $F_0F_5^2$   &  $F_1^3$     & $F_1^2F_5$   &   $F_1F_5^2$  &  $F_5^3$\\
\hline
$\frac{1}{18}$ &$-\frac{1}{6}$&   $0$        &$\frac{1}{2}$ &    $0$      &      $0$     &$-\frac{3}{2}$&       $1$    & $-2$ & $4$\\
\hline
\end{tabular}
\end{center}
This gives the~following intersection table:
\begin{center}
\renewcommand\arraystretch{1.6}
\begin{tabular}{|c||c|c|c| }
\hline
 $\bullet$        &      $F_0$    &    $F_1$ 	 & $F_5$ \\
\hline
\hline$C_{12}$ & $\frac{1}{3}$ &    $-1$ 	 & $1$ \\
\hline
$C_{15}$ &      $0$      &    $1$	 & $-2$  \\
\hline
$C_{01}$ & $-\frac{1}{6}$ & $\frac{1}{2}$ &  $0$  \\
\hline
\end{tabular}
\end{center}

Now, we set $A=-(K_Y+\Delta)$. Take $u\in\mathbb{R}_{\geqslant 0}$. Set $L(u)=\rho^*(A) - u F_0$.
Then
$$
L(u)\sim_{\mathbb{R}} (9-u)F_0+3F_1+F_5,
$$
so $L(u)$ is pseudo-effective $\iff$ $u\leqslant 9$.
Let us find the~Zariski decomposition for $L(u)$.

Observe that $L(u)$ is nef for $u\in[0,3]$.
Since $L(3) \cdot C_{12} = 0$ and $C_{12}$ is unique in its numerical equivalence class,
we consider a small $\mathbb{Q}$-factorial modification $Y_0\dasharrow Y_1$ along the~curve $C_{12}$ such that
$$
Y_1=\big(\mathbb{C}^6 \setminus Z(I_1)\big)/\mathbb{G}_m^3,
$$
where the~torus-action is the~same, and the~irrelevant ideal
$$
I_1 = \langle x_1, x_2 \rangle \cap \langle x_4, x_5 \rangle \cap \langle x_0, x_3 \rangle.
$$
The fan of $Y_1$ is generated by the~same vectors, but the~cone structure is different:
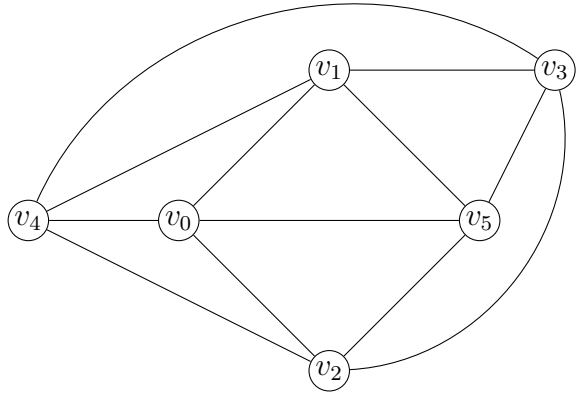
\begin{center}
\begin{tikzpicture}[node distance=1cm,auto]
\node[state, inner sep=1pt,minimum size=0pt] (q_0) {$v_0$};
\node[state, inner sep=1pt,minimum size=0pt] (q_1) at (2,2) {$v_1$};
\node[state,inner sep=1pt,minimum size=0pt] (q_2) at (2,-2) {$v_2$};
\node[state,inner sep=1pt,minimum size=0pt] (q_3) at (5,2) {$v_3$};
\node[state,inner sep=1pt,minimum size=0pt] (q_4) at (-2,0) {$v_4$};
\node[state,inner sep=1pt,minimum size=0pt] (q_5) at (4,0) {$v_5$};

\path[-] (q_0) edge node[swap] {} (q_1);
\path[-] (q_0) edge node[swap] {} (q_2);
\path[-] (q_0) edge node[swap] {} (q_4);
\path[-] (q_0) edge node[swap] {} (q_5); %added to obtain X_1
%\path[-] (q_1) edge node[swap] {} (q_2); removed from X_0
\path[-] (q_1) edge node[swap] {} (q_3);
\path[-] (q_1) edge node[swap] {} (q_4);
\path[-] (q_1) edge node[swap] {} (q_5);
\path[-] (q_2) edge node[swap] {} (q_4);
\path[-] (q_2) edge node[swap] {} (q_5);
\path[-] (q_3) edge node[swap] {} (q_5);
\path[-, bend right = 50] (q_2) edge node {} (q_3);
\path[-, bend right = 50] (q_3) edge node {} (q_4);
\end{tikzpicture}
\end{center}

Abusing our previous notations, we denote the~divisor $\{x_i = 0\}\subset Y_1$ also by $F_i$,
and we let $C_{ij} = F_i \cap F_j$ for $i\ne j$ such that $F_i \cap F_j$ is a curve.
Then $\overline{\mathrm{NE}(Y_1)} = \langle C_{01}, C_{15}, C_{05} \rangle$,
and intersections on $Y_1$ are described in the following two tables:
\begin{center}
\renewcommand\arraystretch{1.6}
\begin{tabular}{|c|c|c|c|c|c|c|c|c|c|c|}
\hline
$F_0^3$      & $F_0^2F_1$   &  $F_0^2F_5$ & $F_0F_1^2$    & $F_0F_1F_5$ & $F_0F_5^2$   &  $F_1^3$     & $F_1^2F_5$   &   $F_1F_5^2$  &  $F_5^3$\\
\hline

 $0$	     &    $0$	  &$-\frac{1}{6}$&     $0$      &$\frac{1}{2}$&$-\frac{1}{2}$&    $0$	    &$-\frac{1}{2}$&$-\frac{1}{2}$& $\frac{5}{2}$\\
\hline
\end{tabular}
\end{center}
\begin{center}
\renewcommand\arraystretch{1.6}
\begin{tabular}{|c||c|c|c| }
\hline
$\bullet$         &      $F_0$    &       $F_1$      & $F_5$ \\
\hline
\hline
$C_{05}$ & $-\frac{1}{6}$ & $\frac{1}{2}$ & $-\frac{1}{2}$ \\
\hline
$C_{15}$ &  $\frac{1}{2}$ & $-\frac{1}{2}$ & $-\frac{1}{2}$  \\
\hline
$C_{01}$ &       $0$ &     $0$ &  $\frac{1}{2}$  \\

\hline
\end{tabular}
\end{center}

Thus, we see that the~proper transform on $Y_1$  of the divisor $L(u)$ is nef for $u\in[3,5]$,
and it intersects the~curve $C_{15}$ trivially for $u=5$.
Since $C_{15}$ is unique in its numerical equivalence class,
we consider another small $\mathbb{Q}$-factorial modification $Y_1\dasharrow Y_2$
such that
$$
Y_2=\big(\mathbb{C}^6 \setminus Z(I_2)\big)/\mathbb{G}_m^3,
$$
where the~torus-action is again given by the~matrix $M$ and the~irrelevant ideal
$$
I_2 = \langle x_1, x_2 \rangle \cap \langle x_4, x_5 \rangle \cap \langle x_1, x_5 \rangle \cap \langle x_0, x_2, x_3 \rangle \cap \langle x_0, x_3, x_4 \rangle.
$$
Then the fan of $Y_2$ is generated by the~same vectors, but the~cone structure is different:

\begin{center}
\begin{tikzpicture}[node distance=1cm,auto]
\node[state, inner sep=1pt,minimum size=0pt] (q_0) {$v_0$};
\node[state, inner sep=1pt,minimum size=0pt] (q_1) at (2,2) {$v_1$};
\node[state,inner sep=1pt,minimum size=0pt] (q_2) at (2,-2) {$v_2$};
\node[state,inner sep=1pt,minimum size=0pt] (q_3) at (5,2) {$v_3$};
\node[state,inner sep=1pt,minimum size=0pt] (q_4) at (-2,0) {$v_4$};
\node[state,inner sep=1pt,minimum size=0pt] (q_5) at (4,0) {$v_5$};

\path[-] (q_0) edge node[swap] {} (q_1);
\path[-] (q_0) edge node[swap] {} (q_2);
\path[-] (q_0) edge node[swap] {} (q_3); %added to obtain X_2
\path[-] (q_0) edge node[swap] {} (q_4);
\path[-] (q_0) edge node[swap] {} (q_5); %added to obtain X_1
%\path[-] (q_1) edge node[swap] {} (q_2); removed from X_0
\path[-] (q_1) edge node[swap] {} (q_3);
\path[-] (q_1) edge node[swap] {} (q_4);
%\path[-] (q_1) edge node[swap] {} (q_5); removed from X_1
\path[-] (q_2) edge node[swap] {} (q_4);
\path[-] (q_2) edge node[swap] {} (q_5);
\path[-] (q_3) edge node[swap] {} (q_5);
\path[-, bend right = 50] (q_2) edge node {} (q_3);
\path[-, bend right = 50] (q_3) edge node {} (q_4);
\end{tikzpicture}
\end{center}

We abuse our~notations again and denote the~divisor $\{x_i = 0\}\subset Y_2$ also by $F_i$.
Similarly, we let $C_{ij} = F_i \cap F_j$ for $i\ne j$ such that $F_i \cap F_j$ is a curve.
Then $\overline{\mathrm{NE}}(Y_2)=\langle C_{01}, C_{03}, C_{05} \rangle$,
and intersections on $Y_2$ are described in the~following two tables:
\begin{center}
\renewcommand\arraystretch{1.6}
\begin{tabular}{|c|c|c|c|c|c|c|c|c|c|c|}
\hline
$F_0^3$      & $F_0^2F_1$   &  $F_0^2F_5$ & $F_0F_1^2$    & $F_0F_1F_5$ & $F_0F_5^2$   &  $F_1^3$     & $F_1^2F_5$   &   $F_1F_5^2$  &  $F_5^3$\\
\hline
$-\frac{1}{2}$&$\frac{1}{2}$ & $\frac{1}{3}$&$-\frac{1}{2}$&    $0$      &      $-1$    &$\frac{1}{2}$ &       $0$    & $0$ & $3$\\
\hline
\end{tabular}
\end{center}
\begin{center}
\renewcommand\arraystretch{1.6}
\begin{tabular}{|c||c|c|c| }
\hline
$\bullet$         &      $F_0$    &       $F_1$      & $F_5$ \\
\hline
\hline
$C_{05}$ & $\frac{1}{3}$ &       $0$      &     $-1$ \\
\hline
$C_{03}$ &  $\frac{-2}{3}$ &    $1$      &    $1$  \\
\hline
$C_{01}$ &  $\frac{1}{2}$ & $-\frac{1}{2}$ &    $0$  \\
\hline
\end{tabular}
\end{center}

The proper transform on $Y_2$ of the divisor $L(u)$ is nef for $u\in [5,6]$,
and it intersects both~curves $C_{01}$ and $C_{05}$ trivially for $u=6$.
Furthermore, if $u\in[6,9]$, then the negative part of the Zariski decomposition of the divisor $L(u)$ on the threefold $Y_2$ is
$$
N(u)=(u-6)F_1+\frac{u-6}{3}F_5,
$$
while the positive part is $P(u)\sim_{\mathbb{R}}(9-u)(F_0+F_1+\frac{1}{3}F_5)$.
This gives
$$
\mathrm{vol}\big(L(u)\big) = \begin{cases}
13-\frac{u^3}{18}			   & u\in [0,3],\\
\frac{-u^2+3+23}{2}				            & u\in [3,5],\\
\frac{1}{2}u^3 - 8 u^2+\frac{3}{2} u		& u\in [5,6],\\
-\frac{1}{9} u^3 +3 u^2 - 27u+81		    & u\in [6,9].
\end{cases}
$$
Integrating, we get $S_{A}(F_0)=\frac{127}{26}$. Since $A_{Y,\Delta}(F_0) = 5$, we get $\frac{A_{Y,\Delta}(F_0)}{S_{A}(F_0)}=\frac{130}{127}>1$.

Next we construct a partial common toric resolution for $Y_0$, $Y_1$, $Y_2$,
which is easy to see from fan toric picture: we want to add the~following rays:
\begin{align*}
&v_6 = (3,2,0) \in \langle v_1, v_2 \rangle \cap \langle v_0, v_5 \rangle, \\
&v_7 = (1,0,-1) \in \langle v_0, v_3 \rangle \cap \langle v_0, v_3 \rangle, \\
&v_8 = (3,1,0) \in \langle v_1, v_2 \rangle \cap \langle v_0, v_3 \rangle.
\end{align*}
Set $\widetilde{Y}$ be the toric variety corresponding to $v_0,\dots,v_8$ with the~following cone structure:
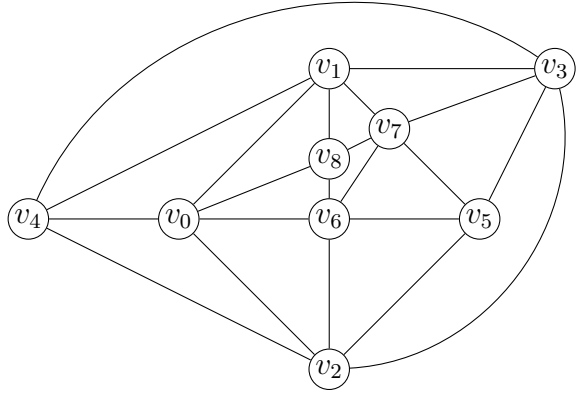
\begin{center}
\begin{tikzpicture}[node distance=1cm,auto]
\node[state, inner sep=1pt,minimum size=0pt] (q_0) {$v_0$};
\node[state, inner sep=1pt,minimum size=0pt] (q_1) at (2,2) {$v_1$};
\node[state,inner sep=1pt,minimum size=0pt] (q_2) at (2,-2) {$v_2$};
\node[state,inner sep=1pt,minimum size=0pt] (q_3) at (5,2) {$v_3$};
\node[state,inner sep=1pt,minimum size=0pt] (q_4) at (-2,0) {$v_4$};
\node[state,inner sep=1pt,minimum size=0pt] (q_5) at (4,0) {$v_5$};
\node[state,inner sep=1pt,minimum size=0pt] (q_6) at (2,0) {$v_6$};
\node[state,inner sep=1pt,minimum size=0pt] (q_7) at (2.8,1.2) {$v_7$};
\node[state,inner sep=1pt,minimum size=0pt] (q_8) at (2,0.8) {$v_8$};
\path[-] (q_0) edge node[swap] {} (q_1);
\path[-] (q_0) edge node[swap] {} (q_2);
\path[-] (q_0) edge node[swap] {} (q_4);
\path[-] (q_0) edge node[swap] {} (q_6);
\path[-] (q_0) edge node[swap] {} (q_8);
\path[-] (q_1) edge node[swap] {} (q_3);
\path[-] (q_1) edge node[swap] {} (q_4);
\path[-] (q_1) edge node[swap] {} (q_8);
\path[-] (q_1) edge node[swap] {} (q_7);
\path[-] (q_2) edge node[swap] {} (q_4);
\path[-] (q_2) edge node[swap] {} (q_5);
\path[-] (q_2) edge node[swap] {} (q_6);
\path[-] (q_3) edge node[swap] {} (q_5);
\path[-] (q_3) edge node[swap] {} (q_7);
\path[-] (q_5) edge node[swap] {} (q_6);
\path[-] (q_5) edge node[swap] {} (q_7);
\path[-] (q_6) edge node[swap] {} (q_8);
\path[-] (q_6) edge node[swap] {} (q_7);
\path[-] (q_7) edge node[swap] {} (q_8);
\path[-, bend right = 50] (q_2) edge node {} (q_3);
\path[-, bend right = 50] (q_3) edge node {} (q_4);
\end{tikzpicture}
\end{center}
Then we have the following toric diagram:
$$
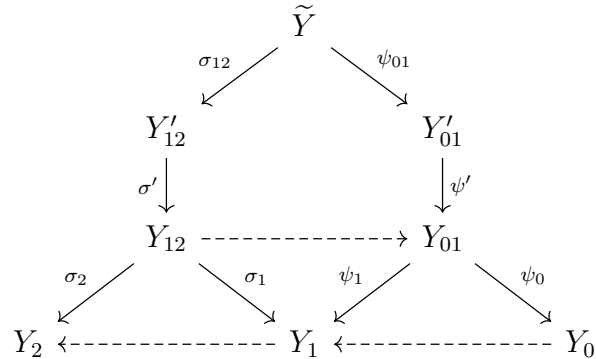
\begin{tikzcd}
	&& {\widetilde{Y}} \\
	& {Y_{12}^\prime} && {Y_{01}^\prime} \\
	& {Y_{12}} && {Y_{01}} \\
	{Y_2} && {Y_1} && {Y_0}
	\arrow[dashed, from=3-2, to=3-4]
	\arrow["{\sigma_2}"', from=3-2, to=4-1]
	\arrow["{\sigma_1}", from=3-2, to=4-3]
	\arrow["{\psi_1}"', from=3-4, to=4-3]
	\arrow["{\psi_0}", from=3-4, to=4-5]
	\arrow[dashed, from=4-5, to=4-3]
	\arrow[dashed, from=4-3, to=4-1]
	\arrow["{\psi_{01}}", from=1-3, to=2-4]
	\arrow["{\sigma_{12}}"', from=1-3, to=2-2]
	\arrow["{\sigma^\prime}"', from=2-2, to=3-2]
	\arrow["{\psi^\prime}", from=2-4, to=3-4]
\end{tikzcd}$$
where toric maps can be described as follows:
\begin{center}
\renewcommand\arraystretch{1.6}
\begin{tabular}{|c|c|c|c|c| }
\hline
 map        &    center    &  weights           &   exceptional divisor  & relation\\
\hline
\hline
$\psi_0$        &  $x_1 = x_2 = 0$ &    $(3,2)$             &     $\{x_6=0\}$              & $3 v_1+2 v_2 = v_6$  \\
\hline
$\psi_1$        &  $x_0 = x_5 = 0$ &    $(1,3)$             &     $\{x_6=0\}$              & $v_0+3 v_5 = v_6$  \\
\hline
$\sigma_1$      &  $x_1 = x_5 = 0$ &    $(1,1)$             &     $\{x_7=0\}$              & $v_1+v_5 = v_7$  \\
\hline
$\sigma_2$      &  $x_0 = x_3 = 0$ &    $(1,2)$             &     $\{x_7=0\}$              & $v_0+2 v_3 = v_7$  \\
\hline
$\psi^\prime$   &  $x_1 = x_5 = 0$ &    $(1,1)$             &     $\{x_7=0\}$              & $v_1+v_5 = v_7$  \\
\hline
$\sigma^\prime$ &  $x_0 = x_5 = 0$ &    $(1,3)$             &     $\{x_6=0\}$              & $v_0+3 v_5 = v_6$  \\
\hline
$\psi_{01}$     &  $x_1 = x_6 = 0$ &    $\frac{1}{2}(3,1)$  &     $\{x_8=0\}$              & $3 v_1+v_6 = 2 v_8$  \\
\hline
$\sigma_{12}$   &  $x_0 = x_7 = 0$ &    $\frac{1}{2}(1,3)$  &     $\{x_8=0\}$              & $v_1+3 v_7 = 2 v_8$  \\
\hline
\end{tabular}
\end{center}
Here, $\frac{1}{2}(a,b)$ indicates that the~variety has an $\mathbb{A}_1$-singularity along the~center of blow up.

Now, we set $\varphi_0=\psi_{01}\circ\psi^\prime\circ\psi_0$, $\varphi_1=\psi_{01}\circ\psi^\prime\circ\psi_1$, $\varphi_2=\sigma_{12} \circ \sigma^\prime \circ \sigma_2$.
Let $\widetilde{F}_i$ be the toric divisor $\{x_i=0\}\subset\widetilde{Y}$. Then
\begin{align*}
\varphi_0^*(F_0)&\sim_{\mathbb{Q}} \widetilde{F}_0,\\
\varphi_0^*(F_1)&\sim_{\mathbb{Q}} \widetilde{F}_1+3\widetilde{F}_6+\widetilde{F}_7+3\widetilde{F}_8,\\
\varphi_0^*(F_5)&\sim_{\mathbb{Q}} \widetilde{F}_5+\widetilde{F}_7,\\
\varphi_1^*(F_0)&\sim_{\mathbb{Q}} \widetilde{F}_0+\widetilde{F}_6+\frac{1}{2}\widetilde{F}_8,\\
\varphi_1^*(F_1)&\sim_{\mathbb{Q}} \widetilde{F}_1+\widetilde{F}_7+\frac{3}{2}\widetilde{F}_8,\\
\varphi_1^*(F_5)&\sim_{\mathbb{Q}} \widetilde{F}_5+3 \widetilde{F}_6+\widetilde{F}_7+\frac{3}{2}\widetilde{F}_8,\\
\varphi_2^*(F_0)&\sim_{\mathbb{Q}} \widetilde{F}_0+\widetilde{F}_6+\widetilde{F}_7+2\widetilde{F}_8,\\
\varphi_2^*(F_1)&\sim_{\mathbb{Q}} \widetilde{F}_1,\\
\varphi_2^*(F_5)&\sim_{\mathbb{Q}} \widetilde{F}_5+3\widetilde{F}_6.
\end{align*}
Using this, we describe the Zariski decomposition of the divisor $\varphi_0^*(L(u))$ as follows:
$$
\widetilde{P}(u) \sim_{\mathbb{R}} \begin{cases}
    (9-u) \widetilde{F}_0+3 \widetilde{F}_1+\widetilde{F}_5+9 \widetilde{F}_6+4 \widetilde{F}_7+9 \widetilde{F}_8 & u\in[0,3],\\
    (9-u) \widetilde{F}_0+3 \widetilde{F}_1+\widetilde{F}_5+(12-u) \widetilde{F}_6+4 \widetilde{F}_7+\frac{21-u}{2} \widetilde{F}_8 & u\in[3,5],\\
    (9-u) \widetilde{F}_0+3 \widetilde{F}_1+\widetilde{F}_5+(12-u) \widetilde{F}_6+(9-u) \widetilde{F}_7+2(9-u) \widetilde{F}_8 & u\in[5,6],\\
    (9-u) (\widetilde{F}_0+\widetilde{F}_1+\frac{1}{3} \widetilde{F}_5+2 \widetilde{F}_6+\widetilde{F}_7+2 \widetilde{F}_8) & u\in[6,9],
\end{cases}
$$
and
$$
\widetilde{N}(u)=\begin{cases}
0                                                                           & u\in[0,3],\\
(u-3)\widetilde{F}_6+\frac{u-3}{2}\widetilde{F}_8                       & u\in[3,5],\\
(u-3)\widetilde{F}_6+(u-5) \widetilde{F}_7+(2u-9)\widetilde{F}_8      & u\in[5,6],\\
(u-6)\widetilde{F}_1+\frac{u}{3}\widetilde{F}_5+(2u-9)\widetilde{F}_6+(u-5)\widetilde{F}_7+(2u-9)\widetilde{F}_8              & u\in[6,9].
\end{cases}
$$
where $\widetilde{P}(u)$ is the positive part, and $\widetilde{N}(u)$ is the negative part.

Now, we describe $\widetilde{P}(u)\vert_{\widetilde{F}_0}$ and $\widetilde{N}(u)\vert_{\widetilde{F}_0}$ for every $u\in[0,9]$.
We have $\widetilde{Y}=(\mathbb{C}^9 \setminus \widetilde{I})/\mathbb{G}_m^6$,
where the~torus action is given by the~matrix
$$
\widetilde{M}=\left(\begin{array}{ccccccccc}
x_0 & x_1 & x_2 & x_3 & x_4 & x_5 & x_6 & x_7 & x_8\\
 0  &  1  &  1  &  1  &  2  &  0  &  0  &  0  &  0 \\
 0  &  0  &  0  &  0  &  1  &  1  &  0  &  0  &  0 \\
 1  &  0  &  1  &  3  &  3  &  0  &  0  &  0  &  0 \\
 0  &  0  &  1  &  3  &  6  &  0  &  1  &  0  &  0 \\
 0  &  0  &  1  &  1  &  3  &  0  &  0  &  1  &  0 \\
 0  &  0  &  2  &  3  &  6  &  0  &  0  &  0  &  1
\end{array}\right),
$$
and the~irrelevant ideal
\begin{align*}
\widetilde{I} = &\langle x_0, x_3 \rangle \cap \langle x_0, x_5 \rangle \cap \langle x_0, x_7 \rangle \cap \langle x_1, x_2 \rangle \cap \langle x_1, x_5 \rangle \cap \langle x_1, x_6 \rangle \cap \langle x_2, x_7 \rangle \cap  \langle x_2, x_8 \rangle \\
& \cap  \langle x_3, x_6 \rangle
\cap \langle x_3, x_8 \rangle \cap  \langle x_4, x_5 \rangle \cap  \langle x_4, x_6 \rangle \cap  \langle x_4, x_7 \rangle \cap  \langle x_4, x_8 \rangle \cap  \langle x_5, x_8 \rangle.
\end{align*}
To obtain a similar description of the~surface $\widetilde{F}_0$, set $x_0=0$, eliminate the~first row in~$\widetilde{M}$,
and set $x_3=x_5=x_7=1$, since $\widetilde{I} \subset \langle x_0, x_3 \rangle \cap \langle x_0, x_5 \rangle \cap \langle x_0, x_7 \rangle$.
The resulting matrix is
$$
\left(\begin{array}{cccccc}
x_1 & x_2 & x_4 & x_6 & x_8\\
 3  &  2  &  3  &  0  &  0 \\
 0  &  0  &  3  &  1  &  0 \\
 0  &  1  &  3  &  0  &  1
\end{array}\right).
$$
Using this, we see that $\widetilde{F}_0=(\mathbb{C}^5 \setminus Z(I_{\widetilde{F}_0}))/\mathbb{G}_m^3$, where the~torus action is given by
$$
\left(\begin{array}{cccccc}
z_1 & z_2 & z_3 & z_4 & z_5\\
 1  &  1  &  2  &  0  &  0 \\
 0  &  1  &  0  &  1  &  0 \\
 0  &  1  &  1  &  0  &  1
\end{array}\right),
$$
and $I_{\widetilde{F}_0} = \langle z_1, z_3 \rangle \cap \langle z_1, z_4 \rangle \cap \langle z_2, z_4 \rangle \cap \langle z_2, z_5 \rangle \cap \langle z_3, z_5 \rangle$.
We can see from the~matrices that
$$
x_1\big\vert_{\widetilde{F}_0} = z_1,\quad x_2^3\big\vert_{\widetilde{F}_0} = z_3,\quad x_4\big\vert_{\widetilde{F}_0} = z_2,\quad x_6^3\big\vert_{\widetilde{F}_0} = z_4,\quad x_8^3\big\vert_{\widetilde{F}_0} = z_5.
$$
The fan of the toric surface $\widetilde{F}_0$ is given by
$$
w_1=(1,0), \quad w_2=(-1,-2), \quad w_3=(0,1), \quad w_4=(1,2), \quad w_5=(1,1)
$$
with obvious cone structure.
For $i\in\{1,2,3,4,5\}$, let $C_i$ be the~curve in $\widetilde{F}_0$ given $z_i=0$.
The cone of effective divisors of the~surface $\widetilde{F}_0$ is generated by the curves $C_1$, $C_4$, $C_5$,
and their intersection form is given in the~following table:
\begin{center}
\renewcommand\arraystretch{1.6}
\begin{tabular}{|c||c|c|c| }
\hline
$\bullet$         &     $C_1$    &       $C_4$      & $C_5$ \\
\hline
\hline
$C_1$ &  $-\frac{1}{2}$ &       $0$      &     $1$ \\
\hline
$C_4$ &         $0$     &       $-1$      &    $1$  \\
\hline
$C_5$ &      $1$        &        $1$       &    $-2$  \\
\hline
\end{tabular}
\end{center}
Further, we compute
$$
\widetilde{P}(u)\big\vert_{\widetilde{F}_0}\sim_{\mathbb{R}} \begin{cases}
 \frac{u}{3}C_1+\frac{u}{3} C_4+\frac{u}{3} C_5                    &    u \in [0,3]\\
 \frac{u}{3}C_1+C_4+(\frac{1}{2}+\frac{u}{6})C_5                 &    u \in [3,5]\\
 \frac{u}{3}C_1+C_4+(3 - \frac{u}{3}) C_5                          &    u \in [5,6]\\
 (6 - \frac{2u}{3}) C_1+(3 - \frac{u}{3}) C_4+(3 - \frac{u}{3}) C_5   &    u \in [6,9],
\end{cases}
$$
and
$$
\widetilde{N}(u)\big\vert_{\widetilde{F}_0} =
\begin{cases}
0                                       & u\in[3,5],\\
\frac{u-3}{6}(2C_4+C_5)              & u\in[3,5],\\
\frac{u-3}{3}C_4+\frac{2u-9}{3} C_5  &   u\in[5,6],\\
(u-6)C_1+\frac{2u-9}{3}(2C_4+C_5)   &   u\in[6,9].
\end{cases}
$$

Let $\theta\colon\widetilde{F}_0\to F_0$ be the morphism induced by $\varphi_0$.
Then $\theta$ is a birational morphism that contracts $C_4$ and $C_5$.
Set $\overline{C}_1=\theta(C_1)$, $\overline{C}_2=\theta(C_2)$, $\overline{C}_3=\theta(C_3)$,
identify $F_0=\mathbb{P}(1,1,2)$ with coordinates $\bar{z}_1$, $\bar{z}_2$, $\bar{z}_3$ such that
$\overline{C}_1=\{\bar{z}_1=0\}$, $\overline{C}_2=\{\bar{z}_2=0\}$, $\overline{C}_3=\{\bar{z}_3=0\}$,
where $\bar{z}_1$ and $\bar{z}_2$ are coordinates of weight $1$, and $\bar{z}_3$ is a coordinate of weight $2$.
Then
$$
\theta\big(C_4\big)=\theta\big(C_5\big)=\overline{C}_1\cap\overline{C}_3=[0:1:0],
$$
and $\theta$ is a composition of the ordinary blow up at the point $[0:1:0]$
with the~consecutive blow up at the~point on the proper transform of the curve $\overline{C}_3$.
Note that $C_5$ is the proper transform of the exceptional curve for the first blow up and $C_4$ is the exceptional curve for the second blow up.

Let $B_0$ be the proper transform on $Y_0$ of the surface $B$. Set $\Delta_0=\frac{1}{2}B_0$ and $B_{F_0}=B_0\vert_{F_0}$.
Then, changing the coordinates $\bar{z}_1$, $\bar{z}_2$, $\bar{z}_3$,
we may also assume that
$$
B_{F_0}=\big\{\overline{z}_1^2+\overline{z}_2^2=\overline{z}_3\big\}\subset F_0.
$$
This curve is smooth, it does not contain the~singular point of   $F_0$, and $[0:1:0]\not\in B_{F_0}$.
The geometry of the surface $F_0$ can be illustrated by the following picture:
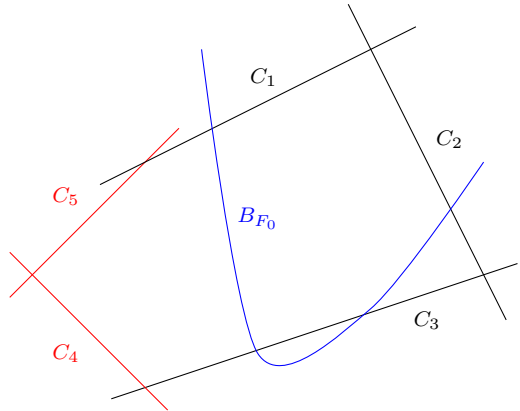
\begin{center}
\begin{tikzpicture}[scale=1.5]
\draw [red] (-0.2,-0.2) -- (1.3,1.3);
\node at (0.3,0.7) {\tiny$\color{red}C_5$};
\draw (0.6,0.8) -- (3.4,2.2);
\node at (2.05,1.75) {\tiny$C_1$};
\draw (2.8,2.4) -- (4.2,-0.4);
\node at (3.7,1.2) {\tiny$C_2$};
\draw (4.3,0.1) -- (0.7,-1.1);
\node at (3.5,-0.4) {\tiny$C_3$};
\draw [red] (1.2,-1.2) -- (-0.2,0.2);
\node at (0.3,-0.7) {\tiny$\color{red}C_4$};
\draw [blue] plot [smooth] coordinates {(1.5,2) (2,-0.7) (3,-0.3) (4,1)};
\node at (2,0.5) {\tiny{$\color{blue} B_{F_0}$}};
\end{tikzpicture}
\end{center}

Note that the surface $Y_0$ is singular along the curve $\overline{C}_3$. We set
$$
\Delta_{F_0}=\frac{1}{2}B_{F_0}+\frac{2}{3}\overline{C}_3.
$$
Then $K_{F_0}+\Delta_{F_0}\sim_{\mathbb{Q}}(K_{Y_0}+\Delta_0)\vert_{F_0}$,
and $\Delta_{F_0}$ is the~corresponding different \cite{Prokhorov}.

Now, we are ready to apply \cite{AbbanZhuang,Book,Fujita2021}.
Let $Q$ be a point in $F_0$, let $C$ be a smooth curve in the~surface $F_0$ that contains $Q$,
let $\widetilde{C}$ be its proper transform on~$\widetilde{F}_0$. For $u\in[0,9]$, let
$$
t(u)=\inf\Big\{v\in \mathbb R_{\geqslant 0} \ \big|\ \text{the divisor $\widetilde{P}(u)\big\vert_{\widetilde{F}_0}-v\widetilde{C}$ is pseudo-effective}\Big\}.
$$
For real number  $v\in[0,t(u)]$, let $P(u,v)$ and  $N(u,v)$ be the~positive part and the~negative part of the~Zariski decomposition of
the~divisor \mbox{$\widetilde{P}(u)\vert_{\widetilde{F}_0}-v\widetilde{C}$}, respectively.
Set
$$
S_L\big(W^{F_0}_{\bullet,\bullet};C\big)=\frac{3}{A^3}\int\limits_0^{9}\big(\widetilde{P}(u)\big\vert_{\widetilde{F}_0}\big)^2\mathrm{ord}_{\widetilde{C}}\big(\widetilde{N}(u)\big\vert_{\widetilde{F}_0}\big)du+\frac{3}{A^3}\int\limits_0^{9}\int\limits_0^{t(u)}\big(P(u,v)\big)^2dvdu.
$$
Write $\theta^*(C)=\widetilde{C}+\Sigma$ for an effective divisor $\Sigma$ on the~surface $\widetilde{F}_0$.
For $u\in[0,9]$, write
$$
\widetilde{N}(u)\big\vert_{\widetilde{F}_0}=d(u)\widetilde{C}+N^\prime(u),
$$
where $d(u)=\mathrm{ord}_{\widetilde{C}}(\widetilde{N}(u)\vert_{\widetilde{F}_0})$, and $N^\prime(u)$ is an effective divisor~on~$\widetilde{F}_0$.~Set
$$
S\big(W_{\bullet, \bullet,\bullet}^{F_0,C};Q\big)=\frac{3}{A^3}
\int\limits_0^{9}\int\limits_0^{t(u)}\big(P(u,v)\cdot\widetilde{C}\big)^2dvdu+F_Q\big(W_{\bullet, \bullet,\bullet}^{F_0,C}\big)
$$
for
$$
F_Q\big(W_{\bullet, \bullet,\bullet}^{F_0,C}\big)=\frac{6}{A^3}\int\limits_0^{9}\int\limits_0^{t(u)}\big(P(u,v)\cdot\widetilde{C}\big)\cdot\mathrm{ord}_Q\Big(\big(N^\prime(u)+N(u,v)-(v+d(u))\Sigma\big)\big|_{\widetilde{C}}\Big)dvdu,
$$
where we consider $Q$ as a point in $\widetilde{C}$ using the~isomorphism $\widetilde{C}\cong C$ induced by $\theta$.

We will choose $C$ such that the pair $(F_0,C+\Delta_{F_0}-\mathrm{ord}_C(\Delta_{F_0})C)$ has purely log terminal singularities.
In this case, the curve $C$ is equipped with an effective divisor $\Delta_C$
such that
$$
K_C+\Delta_C\sim_{\mathbb{Q}}\big(K_{F_0}+C+\Delta_{F_0}-\mathrm{ord}_C(\Delta_{F_0})C \big)\big\vert_{C},
$$
and the pair $(C,\Delta_C)$ has Kawamata log terminal singularities.
The $\mathbb{Q}$-divisor $\Delta_C$ is known as the~different,
and it can be computed locally near any point in $C$, see \cite{Prokhorov} for details.

Let $\mathbf{F}$ be a prime divisor over $Y$ such that $P=C_Y(\mathbf{F})$. Recall that
$$
\beta_{Y,\Delta}(\mathbf{F})=A_{Y,\Delta}(\mathbf{F})-S_A(\mathbf{F})=A_{Y,\Delta}(\mathbf{F})-\frac{1}{A^3}\int_{0}^\infty\mathrm{vol}\big(A-u\mathbf{F}\big)du.
$$
Suppose $\beta_{Y,\Delta}(\mathbf{F})\leqslant 0$. Then, using \cite[Corollary 4.18]{Fujita2021}, we obtain
$$
1\geqslant\frac{A_{Y,\Delta}(\mathbf{F})}{S_A(\mathbf{F})}\geqslant
\delta_P(Y,\Delta)\geqslant\min\left\{
\frac{A_{Y,\Delta}(F_0)}{S_{A}(F_0)},
\inf_{Q\in F_0}
\min\left\{\frac{A_{F_0,\Delta_{F_0}}(C)}{S_A(W^{F_0}_{\bullet,\bullet};C)},
\frac{A_{C,\Delta_C}(Q)}{S(W_{\bullet,\bullet,\bullet}^{F_0,C};Q)}\right\}\right\},
$$
where the choice of $C$ in the infimum depends on $Q$.
Thus, since $\frac{A_{Y,\Delta}(F_0)}{S_{A}(F_0)}\geqslant 1$, we have
$$
\inf_{Q\in F_0}
\min\left\{\frac{A_{F_0,\Delta_{F_0}}(C)}{S_A(W^{F_0}_{\bullet,\bullet};C)},
\frac{A_{C,\Delta_C}(Q)}{S(W_{\bullet,\bullet,\bullet}^{F_0,C};Q)}\right\}\leqslant 1.
$$
In fact, since $\frac{A_{Y,\Delta}(F_0)}{S_{A}(F_0)}=\frac{130}{127}>1$,
it follows from \cite[Corollary 4.18]{Fujita2021} and \cite[Theorem 3.3]{AbbanZhuang} that
we have a strict inequality:
$$
\inf_{Q\in F_0}
\min\left\{\frac{A_{F_0,\Delta_{F_0}}(C)}{S_A(W^{F_0}_{\bullet,\bullet};C)},
\frac{A_{C,\Delta_C}(Q)}{S(W_{\bullet,\bullet,\bullet}^{F_0,C};Q)}\right\}<1.
$$
Let us use this to obtain a contradiction, which would finish the proof of Proposition~\ref{proposition:A2}.

Namely, we will show that for every point $Q\in F_0$, there exists a smooth irreducible curve $C\subset F_0$ such that
$Q\in C$, the~log pair $(F_0,C+\Delta_{F_0}-\mathrm{ord}_C(\Delta_{F_0})C)$ has purely log terminal singularities,
and the following two inequalities hold:
\begin{equation}
\label{equation:A2-case-1}
S_A\big(W^{F_0}_{\bullet,\bullet};C\big)\leqslant A_{F_0,\Delta_{F_0}}(C)
\end{equation}
and
\begin{equation}
\label{equation:A2-case-2}
S(W_{\bullet,\bullet,\bullet}^{F_0,C};Q)\leqslant A_{C,\Delta_C}(Q).
\end{equation}
To be precise, we will choose the~curve $C$ as follows:
\begin{itemize}
\item if $Q\in\overline{C}_1$, we let $C=\overline{C}_1$,
\item if $Q\not\in\overline{C}_1$ and $Q\in\overline{C}_3$, we let $C=\overline{C}_3$,
\item if $Q\not\in\overline{C}_1\cup\overline{C}_3$, we let $C$ to be the unique curve in $|\overline{C}_1|$ such that $Q\in C$.
\end{itemize}

\begin{lemma}
\label{lemma:A2-Q-in-C1}
Let $Q$ be a point in $\overline{C}_1$. Set $C=\overline{C}_1$. Then \eqref{equation:A2-case-1} and \eqref{equation:A2-case-2} hold.
\end{lemma}

\begin{proof}
Note that $A_{F_0,\Delta_{F_0}}(C)=1$ and $\Sigma=\overline{C}_4+\overline{C}_5$. We have
$$
d(u)=\begin{cases}
0	& u\in[0,6]\\
u - 6	& u\in[6,9],
\end{cases}
$$
and
$$
t(u)=\begin{cases}
\frac{u}{3}		& u\in[0,6],\\
6-\frac{2u}{3}  & u\in[6,9].
\end{cases}
$$
Moreover we have
$$
N(u,v) = \begin{cases}
		   v(C_4+C_5)				&    u \in [0,3],~v\in[0,\frac{u}{3}], \\
		\frac{v}{2} C_5				&    u \in [3,5],~v\in[0,\frac{u}{3}-1],           \\
\frac{3v+3-u}{3}C_4 +\frac{6v+3-u}{6}C_5	&    u \in [3,5],~v\in[\frac{u}{3}-1,\frac{u}{3}],           \\
			0			        &    u \in [5,6],~v\in[0,u-5],\\
		\frac{v+5-u}{2} C_5		        &    u \in [5,6],~v\in[u-5,\frac{u}{3}-1],\\
   \frac{3v+3-u}{3}C_4+\frac{3v+9-2u}{3}C_5        &    u \in [5,6],~v\in[\frac{u}{3}-1,\frac{u}{3}],\\
			0				&    u \in [6,9],~v\in[0,3-\frac{u}{3}],\\
		\frac{3v+u-9}{3}(C_4+C_5)		&    u \in [6,9],~v\in[3-\frac{u}{3},6-\frac{2u}{3}],
\end{cases}
$$
and
$$
P(u,v)\sim_{\mathbb{R}} \begin{cases}
			\frac{u-3v}{3}(C_1+C_4+C_5)			&    u \in [0,3],~v\in[0,\frac{u}{3}], \\
\frac{u-3v}{3}C_1+C_4+\frac{3+u-3v}{6}C_5	&    u \in [3,5],~v\in[0,\frac{u}{3}-1],           \\
			\frac{u-3v}{3}(C_1+C_4+C_5)			&    u \in [3,5],~v\in[\frac{u}{3}-1,\frac{u}{3}],  \\
	\frac{u-3v}{3}C_1+C_4+\frac{9-u}{3}C_5			&    u \in [5,6],~v\in[0,u-5],\\
\frac{u-3v}{3}C_1+C_4+\frac{3+u-3v}{6}C_5     &    u \in [5,6],~v\in[u-5,\frac{u}{3}-1],\\
			\frac{u-3v}{3}(C_1+C_4+C_5)			&    u \in [5,6],~v\in[\frac{u}{3}-1,\frac{u}{3}],\\
 (\frac{18-2u-3v}{3}C_1+\frac{9-u}{3}(C_4+C_5)     &    u \in [6,9],~v\in[0,3-\frac{u}{3}],      \\
			\frac{18-2u-3v}{3}(C_1+C_4+C_5)		&    u \in [6,9],~v\in[3-\frac{u}{3},6-\frac{2u}{3}],
\end{cases}
$$
which gives
$$
\big(P(u,v)\big)^2 = \begin{cases}
		   \frac{(u-3v)^2}{18}			&    u \in [0,3],~v\in[0,\frac{u}{3}], \\
		\frac{u}{3}-v-\frac{1}{2}		&    u \in [3,5],~v\in[0,\frac{u}{3}-1],           \\
		\frac{(u-3v)^2}{18}				&    u \in [3,5],~v\in[\frac{u}{3}-1,\frac{u}{3}],           \\
 -\frac{u^2}{2}+uv-\frac{v^2}{2}-13+\frac{16}{3} u - 6v &    u \in [5,6],~v\in[0,u-5],\\
		\frac{u}{3}-v-\frac{1}{2}	        &    u \in [5,6],~v\in[u-5,\frac{u}{3}-1],\\
  	 	\frac{(u-3v)^2}{18}			        &    u \in [5,6],~v\in[\frac{u}{3}-1,\frac{u}{3}],\\
	-2u+9+\frac{u^2}{9} - \frac{v^2}{2}		&    u \in [6,9],~v\in[0,3-\frac{u}{3}],      \\
		\frac{(18-2u-3v)^2}{18}   	&    u \in [6,9],~v\in[3-\frac{u}{3},6-\frac{2u}{3}],
\end{cases}
$$
and
$$
P(u,v)\cdot C = \begin{cases}
\frac{u-3v}{6}	&    u \in [0,3],~v\in[0,\frac{u}{3}], \\
\frac{1}{2}		&    u \in [3,5],~v\in[0,\frac{u}{3}-1],           \\
\frac{u-3v}{6}			&    u \in [3,5],~v\in[\frac{u}{3}-1,\frac{u}{3}],           \\
\frac{6-u+v}{2} 	&    u \in [5,6],~v\in[0,u-5],\\
\frac{1}{2}	        &    u \in [5,6],~v\in[u-5,\frac{u}{3}-1],\\
\frac{u-3v}{6}        &    u \in [5,6],~v\in[\frac{u}{3}-1,\frac{u}{3}],\\
\frac{v}{2}		&    u \in [6,9],~v\in[0,3-\frac{u}{3}],      \\
\frac{18-2u-3v}{6}	&    u \in [6,9],~v\in[3-\frac{u}{3},6-\frac{2u}{3}].
\end{cases}
$$
Integrating, we get $S(W_{\bullet,\bullet}^{F_0};C)=\frac{10}{13}<1= A_{F_0,\Delta_{F_0}}(C)$, so \eqref{equation:A2-case-1} holds.

Similarly, we compute
$S(W_{\bullet,\bullet,\bullet}^{F_0,C};Q)= \frac{9}{52}+F_Q(W^{F_0,C}_{\bullet,\bullet,\bullet})$,
where
$$
F_Q\big(W^{F_0,C}_{\bullet,\bullet,\bullet}\big) = \begin{cases}
\frac{1}{12}	& Q = \overline{C}_1 \cap \overline{C}_3,\\
0		& \text{otherwise.}
\end{cases}
$$
Observe that
$$
A_{C,\Delta_C}(Q) = \begin{cases}
\frac{1}{2}	& Q=\overline{C}_1\cap B_{F_0},\\
\frac{1}{2}	& Q=\overline{C}_1\cap\overline{C}_2,\\
\frac{1}{3}	& Q=\overline{C}_1\cap\overline{C}_3,\\
1		& \text{otherwise}.
\end{cases}
$$
Thus, we have
$$
\frac{A_{C,\Delta_C}(Q)}{S(W_{\bullet,\bullet,\bullet}^{F_0,C};Q)} = \begin{cases}
\frac{13}{10}	& Q = \overline{C}_1 \cap \overline{C}_3,\\
\frac{26}{9}	& Q = \overline{C}_1 \cap \overline{C}_2,\\
\frac{26}{9}	& Q=\overline{C}_1\cap B_{F_0},\\
\frac{52}{9}	& \text{otherwise}.
\end{cases}
$$
which implies \eqref{equation:A2-case-2}.
\end{proof}

\begin{lemma}
\label{lemma:A2-Q-in-C3}
Let $Q$ be a point in $\overline{C}_3\setminus\overline{C}_1$. Set $C=\overline{C}_3$.
Then \eqref{equation:A2-case-1} and \eqref{equation:A2-case-2} hold.
\end{lemma}

\begin{proof}
For $u\in[0,9]$, we have $d(u)=0$ and $N^\prime(u)=\widetilde{N}(u)\vert_{\widetilde{F}_0}$.
Since $\widetilde{C}\sim C_3+2C_4+C_5$, we have
$$
t(u) = \begin{cases}
\frac{u}{6}	& u\in[0,6],\\
\frac{9-u}{3}	& u\in[6,9].
\end{cases}
$$
We compute
$$
N(u,v) = \begin{cases}
2 v C_4+v C_5					&    u \in [0,3],~v\in[0,\frac{u}{6}], \\
	0					&    u \in [3,5],~v\in[0,\frac{u-3}{6}],           \\
\frac{u-3}{6}(2C_4+C_5) 			&    u \in [3,5],~v\in[\frac{u-3}{6},\frac{u}{6}],\\
  	0					&    u \in [5,6],~v\in[0,\frac{6-u}{3}],\\
\frac{3v+u-6}{3}(C_4)				&    u \in [5,6],~v\in[\frac{6-u}{3},\frac{2u-9}{3}]\\
\frac{6v+3-u}{3}(C_4)+\frac{v+9-2u}{3}(C_4)	&    u \in [5,6],~v\in[\frac{2u-9}{3},\frac{u}{6}],\\
\frac{2u-9}{3}(C_4+C_5) 			&    u \in [6,9],~v\in[0,\frac{9-u}{3}],
\end{cases}
$$
and
$$
P(u,v) \sim \begin{cases}
	\frac{u-6v}{3}(C_1+C_4+C_5)		&    u \in [0,3],~v\in[0,\frac{u}{6}], \\
\frac{u-6v}{3}C_1+\frac{3+u-6v}{6}C_5+C_4	&    u \in [3,5],~v\in[0,\frac{u-3}{6}],           \\
	\frac{u-6v}{3}(C_1+C_4+C_5) 		&    u \in [3,5],~v\in[\frac{u-3}{6},\frac{u}{6}],\\
\frac{u-6v}{3}C_1+\frac{9-u-3v}{3} C_5+C_4  &    u \in [5,6],~v\in[0,\frac{6-u}{3}],\\
\frac{u-6v}{3}C_1+\frac{9-u-3v}{3}(C_5+C_4) &    u \in [5,6],~v\in[\frac{6-u}{3},\frac{2u-9}{3}],\\
	\frac{u-6v}{3}(C_1+C_4+C_5) 		&    u \in [5,6],~v\in[\frac{2u-9}{3},\frac{u}{6}],\\
	\frac{9-u-3v}{3}(2C_1+C_4+C_5)		&    u \in [6,9],~v\in[0,\frac{9-u}{3}],
\end{cases}
$$
which gives
$$
\big(P(u,v)\big)^2 = \begin{cases}
\frac{u^2}{18}+2v^2-\frac{2}{3}uv		&    u \in [0,3],~v\in[0,\frac{u}{6}], \\
\frac{u}{3}-2v-\frac{1}{2}			&    u \in [3,5],~v\in[0,\frac{u-3}{6}],           \\
\frac{u^2}{18}+2v^2-\frac{2}{3}uv		&    u \in [3,5],~v\in[\frac{u-3}{6},\frac{u}{6}],\\
\frac{16}{3}u-2v-\frac{13}{2}u^2		&    u \in [5,6],~v\in[0,\frac{6-u}{3}],\\
4u-6v-9-\frac{7}{18}u^2+v^2+\frac{2}{3}uv	&    u \in [5,6],~v\in[\frac{6-u}{3},\frac{2u-9}{3}]\\
9-6v-2u+v^2+\frac{u^2}{9}+\frac{2}{3}uv		&    u \in [5,6],~v\in[\frac{2u-9}{3},\frac{u}{6}],\\
\frac{2u-9}{3}(C_4+C_5) 			&    u \in [6,9],~v\in[0,\frac{9-u}{3}],
\end{cases}
$$
and
$$
P(u) \cdot C = \begin{cases}
\frac{u}{3}-2v	&    u \in [0,3],~v\in[0,\frac{u}{6}], \\
1	&    u \in [3,5],~v\in[0,\frac{u-3}{6}],           \\
\frac{u}{3}-2v	&    u \in [3,5],~v\in[\frac{u-3}{6},\frac{u}{6}],\\
1	&    u \in [5,6],~v\in[0,\frac{6-u}{3}],\\
3-v-\frac{u}{3}	&    u \in [5,6],~v\in[\frac{6-u}{3},\frac{2u-9}{3}]\\
\frac{u}{3}-2v	&    u \in [5,6],~v\in[\frac{2u-9}{3},\frac{u}{6}],\\
3-\frac{u}{3}-v	&    u \in [6,9],~v\in[0,\frac{9-u}{3}].
\end{cases}
$$
Thus, integrating we get $S(W^{F_0}_{\bullet,\bullet};C)=\frac{10}{39}<\frac{1}{3}=A_{F_0,\Delta_{F_0}}(C)$, so \eqref{equation:A2-case-1} holds.

Since $Q\neq \overline{C}_1 \cap \overline{C}_3$, we have $F_Q(W^{F_0,C}_{\bullet,\bullet,\bullet})=0$,
which gives $S(W_{\bullet,\bullet,\bullet}^{F_0,\overline{C}_3};Q) = \frac{9}{26}$.
But
$$
A_{C,\Delta_C}(Q) = \begin{cases}
\frac{1}{2}	& Q \in B_{F_0},\\
1	& Q \not\in B_{F_0}.
\end{cases}
$$
Thus, we have
$$
\frac{A_{C,\Delta_C}(Q)}{S(W^{F_0}_{\bullet,\bullet};C)} = \begin{cases}
\frac{13}{10}	& Q \in B_{F_0},\\
\frac{26}{9}	& Q \not\in B_{F_0},
\end{cases}
$$
which implies \eqref{equation:A2-case-1}.
\end{proof}

\begin{lemma}
\label{lemma:A2-Q-general}
Let $Q$ be a point in $F_0$ such that $Q\not\in\overline{C}_1 \cup \overline{C}_3$,
and let $C$ be the unique curve in the~pencil $|\overline{C}_1|$ that contains $Q$.
Then \eqref{equation:A2-case-1} and \eqref{equation:A2-case-2} hold.
\end{lemma}

\begin{proof}
Note that $A_{F_0,\Delta_{F_0}}(C)=1$, and $\widetilde{C}\sim C_1+C_4+C_5$.
We have
$$
t(u) = \begin{cases}
\frac{u}{3}	& u\in[0,3],\\
1       	& u\in[3,6],\\
\frac{9-u}{3}	    & u\in[6,9].
\end{cases}
$$
For every $u\in[0,9]$, we have $d(u)=0$ and $N^\prime(u)=\widetilde{N}(u)\vert_{\widetilde{F}_0}$.
We compute
$$
N(u,v) = \begin{cases}
	0	&    u \in [0,3],~v\in[0,\frac{u}{3}], \\
	0	&    u \in [3,5],~v\in[0,1],           \\
	0       &    u \in [5,6],~v\in[0,6-u],\\
  (v+u-6)C_1    &    u \in [5,6],~v\in[6-u,1],\\
   v C_1	&    u \in [6,9],~v\in[0,3-\frac{u}{3}],
\end{cases}
$$
and
$$
P(u,v) \sim \begin{cases}
 \frac{u-3v}{3}(C_1+C_4+C_5)     				 &    u \in [0,3],~v\in[0,\frac{u}{3}],\\
 \frac{u-3v}{3}C_1+(1-v) C_4+\frac{3+u-6v}{6}C_5  &    u \in [3,5],~v\in[0,1],\\
 \frac{u-3v}{3}C_1+(1-v) C_4+\frac{9-u-3v}{3}C_5           &    u \in [5,6],~v\in[0,6-u],\\
\frac{18-2u-6v}{3}C_1+(1-v) C_4+(\frac{9-u-3v}{3}C_5  	 &    u \in [5,6],~v\in[6-u,1],\\
\frac{9-u-3v}{3}(2 C_1+C_4+C_5) 	 	 &    u \in [6,9],~v\in[0,3-\frac{u}{3}].
\end{cases}
$$
which gives
$$
\big(P(u,v)\big)^2 = \begin{cases}
		   \frac{(u-3v)^2}{18}				 &    u \in [0,3],~v\in[0,\frac{u}{3}] \\
-\frac{1}{2}+\frac{u}{3}-\frac{1}{3}uv+\frac{1}{2}v^2		 &    u \in [3,5],~v\in[0,1]           \\
-\frac{u^2}{2}-\frac{uv}{3}+\frac{v^2}{2}-13+\frac{16}{3} u	 &    u \in [5,6],~v\in[0,6-u]\\
	5+\frac{2uv}{3}+v^2 - \frac{2u}{3} - 6v        	 &    u \in [5,6],~v\in[6-u,1]\\
		\frac{(3-\frac{u}{3}-v)^2}{2}				 &    u \in [6,9],~v\in[0,3-\frac{u}{3}],
\end{cases}
$$
and
$$
P(u) \cdot \widetilde{C} = \begin{cases}
\frac{u-3v}{6}	 &    u \in [0,3],~v\in[0,\frac{u}{3}] \\
\frac{u-3v}{6}	 &    u \in [3,5],~v\in[0,1]  \\
\frac{u-3v}{6}	 &    u \in [5,6],~v\in[0,6-u]\\
\frac{9-u-3v}{3} &    u \in [5,6],~v\in[6-u,1]\\
\frac{9-u-3v}{3} &    u \in [6,9],~v\in[0,3-\frac{u}{3}].
\end{cases}
$$
Thus, integrating we get $S(W^{F_0}_{\bullet,\bullet};C)=\frac{9}{26}<1=A_{F_0,\Delta_{F_0}}(C)$, so \eqref{equation:A2-case-1} holds.

Since $Q\not\in\overline{C}_1\cup \overline{C}_3$, we have $F_Q(W^{F_0,C}_{\bullet,\bullet,\bullet})=0$ and
$$
A_{C,\Delta_C}(Q) = \begin{cases}
\frac{1}{2}	& Q \in B_{F_0},\\
1	& Q \not\in B_{F_0}.
\end{cases}
$$
Integrating, we get $S(W_{\bullet,\bullet,\bullet}^{F_0,C};Q)=\frac{10}{39}$, so that
$$
\frac{A_{C,\Delta_C}(Q)}{S(W^{F_0}_{\bullet,\bullet};C)} = \begin{cases}
\frac{39}{20}	& Q \in B_{F_0},\\
\frac{39}{10}	& Q \not\in B_{F_0},
\end{cases}
$$
which implies \eqref{equation:A2-case-1}.
\end{proof}

Lemmas~\ref{lemma:A2-Q-in-C1}, \ref{lemma:A2-Q-in-C1}, \ref{lemma:A2-Q-general} completes the proof of Proposition~\ref{proposition:A2}.

\section{On the K-moduli spaces}
\label{section:moduli}

In this section, we prove Corollary \ref{corollary:4-2}.
The proof of Corollary~\ref{corollary:3-9} is almost identical, so we omit it.
To start with, let us present the following well known assertion.

\begin{lemma}
\label{lemma:HHR}
Let $X$ be a smooth Fano threefold. Then
$$
h^0\left(X,T_X\right)-h^1\left(X,T_X\right)
=\chi(X,T_X)=\frac{-K_X^3}{2}-18+b_2(X)-\frac{b_3(X)}{2},
$$
where $b_2(X)$ and $b_3(X)$ are the second and the third Betti numbers of $X$, respectively.
\end{lemma}

\begin{proof}
The required assertion immediately follows from the Akizuki--Nakano vanishing theorem and the Hirzebruch--Riemann--Roch theorem, since $-K_X\cdot c_2(X)=24$.
\end{proof}

Now, let us use notations and assumptions introduced in Corollary~\ref{corollary:4-2}.

\begin{lemma}
\label{lemma:GIT}
Let $f\in T$ and $X_f$ be the~Casagrande--Druel 3-fold constructed from $\{f=0\}$.
Suppose that $f$ is GIT semistable with respect to the $\Gamma$-action. Then $X_f$ is K-semistable.
\end{lemma}

\begin{proof}
There exists a one-parameter subgroup
$\lambda\colon\mathbb{G}_m\to\Gamma$ such that
$$
[f_0]=\lim_{t\to 0}\lambda(t)\cdot{[}f{]}
$$
is a GIT polystable point in $T$.
Let $X_{0}$ be the corresponding Casagrande--Druel threefold constructed from $\{f_0=0\}$.
Then it follows from Theorem~\ref{theorem:4-2} that $X_{0}$ is K-polystable.
On the~other hand, the~subgroup $\lambda$ gives isotrivial flat
degeneration of $X_{f}$ to $X_{0}$, which implies that $X_f$ is K-semistable, because K-semistability is an open condition.
\end{proof}

Now, we are ready to prove Corollary \ref{corollary:4-2}.

\begin{proof}[Proof of Corollary \ref{corollary:4-2}]
Since the~construction of Casagrande--Druel 3-folds is functorial,
there exists a $\Gamma$-equivariant flat morphism $\pi_T\colon X_T\to T$
such that
$$
\pi^{-1}_T({[}f{]})\cong X_f.
$$
We set $X_{T^{\operatorname{ss}}}=\pi_T^{-1}(T^{\operatorname{ss}})$.
Then the restriction morphism $X_{T^{\operatorname{ss}}}\to T^{\operatorname{ss}}$ is a $\Gamma$-equivariant
flat family of K-semistable Fano 3-folds by Lemma \ref{lemma:GIT}.

Let $\{T^{\operatorname{ss}}/\Gamma\}$ be the fibered category over $(\operatorname{Sch}/\mathbb{C})_{\operatorname{fppf}}$ in the sense of \cite[Example 4.6.7]{olsson}.
Then the family $X_{T^{\operatorname{ss}}}\to T^{\operatorname{ss}}$ gives a morphism $\{T^{\operatorname{ss}}/\Gamma\}\to\mathcal{M}^{\operatorname{Kss}}_{3,28}$ of fibered categories.
This induces the morphism
$$
\big[T^{\operatorname{ss}}/\Gamma\big]\to \mathcal{M}^{\operatorname{Kss}}_{3,28}
$$
between Artin stacks, since $[T^{\operatorname{ss}}/\Gamma]$ is the stackification of $\{T^{\operatorname{ss}}/\Gamma\}$ (see \cite[Remark 4.6.8]{olsson}).

Since $M$ is the good moduli space of $[T^{\operatorname{ss}}/\Gamma]$,
it follows from \cite[Theorem 6.6]{Alper} that there exists a natural morphism
$$
\Phi\colon M\to M^{\operatorname{Kps}}_{3,28}
$$
that maps $[f]$ to $[X_f]$. This morphism is injective.
Indeed, if $f_1$ and $f_2$ are points in $T$,
then the corresponding Casagrande--Druel 3-folds $X_{f_1}$ and $X_{f_2}$ are isomorphic if and only if the~points $f_1$ and $f_2$ are contained in one $\Gamma$-orbit.

Observe that $M$ is normal. Take $[f]\in M$. Since the deformations of the 3-fold $X_f$ are unobstructed by Proposition~\ref{proposition:unobstruction},
the variety $M^{\operatorname{Kps}}_{3,28}$ is also normal at $[X_f]$ by Luna's \'etale slice theorem \cite[Theorem 1.2]{AHR}.
Moreover, if $X_f$ is smooth, then
$$
\dim_{[X_f]}\big(M^{\operatorname{Kps}}_{3,28}\big)\leqslant h^1\big(X_f,T_{X_f}\big)=\mathrm{dim}(M)
$$
by Lemma \ref{lemma:HHR}, since $h^0(X,T_X)=\dim(\operatorname{Aut}(X))=1$.
Therefore, using the injectivity~of~$\Phi$, we see that the image $\Phi(M)\subset M^{\operatorname{Kps}}_{3,28}$ is a connected component,
and $\Phi$ is an isomorphism onto this connected component by Zariski's main theorem.
\end{proof}

The variety $M^{\operatorname{Kps}}_{(3.9)}$ is well-studied \cite{Stability-quartic-curves}.
Let us describe $M^{\operatorname{Kps}}_{(4.2)}\cong T^{\operatorname{ss}}\mathbin{/\mkern-6mu/}\Gamma$.
Recall that
$$
T=\mathbb{P}\left(H^0\left(V,\mathcal{O}_V(2,2)\right)^\vee\right)
$$
and $\Gamma=\left(\operatorname{SL}_2(\mathbb{C})\times\operatorname{SL}_2(\mathbb{C})\right)\rtimes\mumu_2$,
where $V=\mathbb{P}^1\times\mathbb{P}^1$.
Set $\Gamma_0=\operatorname{SL}_2(\mathbb{C})\times\operatorname{SL}_2(\mathbb{C})$.

\begin{proposition}[Noam Elkies]
\label{proposition:Elkies}
One has $T^{\operatorname{ss}}\mathbin{/\mkern-6mu/}\Gamma_0\cong T^{\operatorname{ss}}\mathbin{/\mkern-6mu/}\Gamma\cong \mathbb{P}(1,2,3)$.
\end{proposition}

\begin{proof}
Let $W=H^0\left(V,\mathcal{O}_V(2,2)\right)$,
let $S$ be the symmetric algebra of $W^\vee$,
let $S^{\Gamma_0}$ be its subalgebra of invariants for the natural $\Gamma_0$-action,
and let $H(t)$ be its Hilbert series:
$$
H(t)=\sum_{k\geqslant 0}\dim\Big(\big(\mathrm{Sym}^k(W^\vee)\big)^{\Gamma_0}\Big)t^k.
$$
Then its follows from \cite[\S 11.9]{Procesi} or \cite[\S 4.6]{DK} that
$$
H(t)=\int_0^1\int_0^1\frac{2-z_1^2-z_1^{-2}}{2}\cdot
\frac{2-z_2^2-z_2^{-2}}{2}\cdot
\prod_{j_1,j_2\in\{-1,0,1\}}\frac{1}{1-t\cdot
z_1^{2j_1}z_2^{2j_2}}d\phi_1 d\phi_2
$$
with $|t|<1$, where $z_1=e^{2\pi\sqrt{-1}\phi_1}$ and $z_2=e^{2\pi\sqrt{-1}\phi_2}$. This gives
$$
H(t)=\frac{1}{(1-t^2)(1-t^3)(1-t^4)}.
$$

Let us find generators of $S^{\Gamma_0}$. Consider the standard basis
$$
x_0^2y_0^2, x_0^2y_0y_1, x_0^2y_1^2, x_0x_1y_0^2, x_0x_1y_0y_1, x_0x_1y_1^2, x_1^2y_0^2, x_1^2y_0y_1, x_1^2y_1^2
$$
of the space $W$, let
$a_{00}, a_{01}, a_{02}, a_{10}, a_{11}, a_{12}, a_{20}, a_{21}, a_{22}$ be the dual basis of the space $W^\vee$,
and let $J_2$, $J_3$, $J_4$ be the~coefficients of the characteristic polynomial of the matrix
$$
\begin{pmatrix}
\frac{1}{2}a_{11} & -a_{10} & -a_{01} & 2a_{00} \\
a_{12} & -\frac{1}{2}a_{11} & -2a_{02} & a_{01} \\
a_{21} & -2a_{20} & -\frac{1}{2}a_{11} & a_{10} \\
2a_{22} & -a_{21} & -a_{12} & \frac{1}{2}a_{11}
\end{pmatrix}
$$
such that $J_k\in\mathrm{Sym}^k(W^\vee)$ for $k\in\{2,3,4\}$.
Then $J_2$, $J_3$, $J_4$ are $\Gamma_0$-invariant,
and these polynomials are algebraically independent, which gives $S^{\Gamma_0}=\mathbb{C}[J_2,J_3,J_4]$,
so that
$$
T^{\operatorname{ss}}\mathbin{/\mkern-6mu/}\Gamma_0\cong \mathbb{P}(2,3,4)\cong\mathbb{P}(1,2,3).
$$
Since the polynomials $J_2$, $J_3$, $J_4$ are also $\Gamma$-invariant, we also get $T^{\operatorname{ss}}\mathbin{/\mkern-6mu/}\Gamma_0\cong T^{\operatorname{ss}}\mathbin{/\mkern-6mu/}\Gamma$.
\end{proof}

\begin{remark}
\label{remark:Elkies}
In fact, Proposition~\ref{proposition:Elkies} is a classical result ---
Peano \cite{Peano} and Turnbull \cite{Turnbull} showed that $S^{\Gamma_0}$ is generated by $J_2$, $J_3$, $J_4$,
see  \cite[\S 12]{Turnbull} and \cite[Pages 242--246]{Olver}.
\end{remark}

The surface $M^{\operatorname{Kps}}_{(4.2)}$ is a component of the K-moduli space of smoothable Fano threefolds.
Another two-dimensional component of this K-moduli~space has been described in \cite{Banff},
and all its one-dimensional components have been described~in~\cite{London}.

\end{document}